\documentclass[12pt,reqno]{amsart}

\newcommand\version{April 24, 2023}

\usepackage{amsmath,amsfonts,amsthm,amssymb,amsxtra}
\usepackage{bbm} 

\newtheorem{theorem}{Theorem}
\newtheorem{proposition}[theorem]{Proposition}
\newtheorem{lemma}[theorem]{Lemma}

\theoremstyle{definition}

\theoremstyle{remark}

\newtheorem{remark}[theorem]{Remark}

\newtheorem*{example*}{Example}
\newtheorem*{remark*}{Remark}

\newcommand{\1}{\mathbbm{1}}

\newcommand{\const}{\mathrm{const}\ }

\renewcommand{\epsilon}{\varepsilon}

\newcommand{\loc}{{\rm loc}}

\newcommand{\N}{\mathbb{N}}

\renewcommand{\phi}{\varphi}
\newcommand{\R}{\mathbb{R}}

\newcommand{\Sph}{\mathbb{S}}

\newcommand{\Z}{\mathbb{Z}}

\DeclareMathOperator{\ran}{ran}

\DeclareMathOperator{\spa}{span}

\DeclareMathOperator{\Tr}{Tr}

\begin{document}

\title[The sharp Sobolev inequality and its stability --- \version]{The sharp Sobolev inequality and its stability:\\ An introduction}

\author{Rupert L. Frank}
\address[Rupert L. Frank]{Mathe\-matisches Institut, Ludwig-Maximilans Universit\"at M\"unchen, The\-resienstr.~39, 80333 M\"unchen, Germany, and Munich Center for Quantum Science and Technology, Schel\-ling\-str.~4, 80799 M\"unchen, Germany, and Mathematics 253-37, Caltech, Pasa\-de\-na, CA 91125, USA}
\email{r.frank@lmu.de}

\thanks{\copyright\, 2023 by the author. This paper may be reproduced, in its entirety, for noncommercial purposes.\\
Partial support through US National Science Foundation grants DMS-1954995, as well as through the Deutsche Forschungsgemeinschaft Excellence Strategy EXC-2111-390814868 is acknowledged.}

\begin{abstract}
	These notes are an extended version of a series of lectures given at the CIME Summer School in Cetraro in June 2022. The goal is to explain questions about optimal functional inequalities on the example of the sharp Sobolev inequality and its fractional generalizations. Topics covered include compactness theorems for optimizing sequences, characterization of optimizers and quantitative stability.
\end{abstract}

\maketitle

\section*{Introduction and outline}

The Sobolev inequality on $\R^d$, $d\geq 3$, states that
\begin{equation}
	\label{eq:sobintro}
	\int_{\R^d} |\nabla u|^2\,dx \gtrsim \left( \int_{\R^d} |u|^{2d/(d-2)}\,dx \right)^{(d-2)/2} \,,
\end{equation}
provided the function $u$ belongs to the homogeneous Sobolev space $\dot H^1(\R^d)$, defined as the completion of $C_c^1(\R^d)$ with respect to the $L^2$-norm of the gradient. We restrict ourselves in these lectures to real-valued functions. The Sobolev inequality \eqref{eq:sobintro} is of great importance in several areas of mathematical analysis, including the calculus of variations, the theory of PDEs, differential geometry and mathematical physics.

The $\gtrsim$-sign in \eqref{eq:sobintro} means that there is a positive constant, depending only on $d$, such that the inequality holds with that constant on the right side. In several applications one is interested in the optimal value of this constant, that is, in the number
$$
S_{d} := \inf_{0\neq u\in\dot H^1(\R^d)} \frac{\int_{\R^d} |\nabla u|^2\,dx}{\left( \int_{\R^d} |u|^{2d/(d-2)}\,dx \right)^{(d-2)/2}} \,.
$$
Related to this is the question whether the supremum defining $S_d$ is attained for some function $u$ and, if so, whether one can characterize all such functions. Again motivated by applications, once this has been carried out one would like to know whether the fact that for some function $0\neq u\in \dot H^1(\R^d)$ the quotient between the left and right sides of \eqref{eq:sobintro} is close to the optimal value $S_{d}$ already implies that $u$ is close to a function for which the supremum is attained. This question is deliberately vague. One needs to specify in which sense the closeness between two functions is understood, and in which sense the closeness of the quotient to the optimal constant is related to the closeness between $u$ and optimal functions. It turns out that in the context of the Sobolev inequality all these questions can be answered, and this is the topic of this series of lectures.

\medskip

Let us take a step back from this concrete problem. The Sobolev inequality is just one (although a paradigmatic) example of a functional inequality and the questions outlined above can be equally asked for other such inequalities. This suggests the following research program in the field of functional inequalities:
\begin{enumerate}
	\item[(0)] Prove the validity of the functional inequality with some constant.
	\item[(1a)] Show that there are optimizing functions.
	\item[(1b)] Show that optimizing sequences are relatively compact (up to symmetries).
	\item[(2a)] Determine the optimal constant.
	\item[(2b)] Characterize the optimizers.
	\item[(2c)] Show that the Hessian around optimizers is nondegenerate (up to symmetries).
	\item[(3)] Show stability of the functional inequality.
\end{enumerate}
The meaning of some of these assertions might not be clear at this point, but the hope is that it will be at the end of this series of lectures. The rough plan of this course is to devote each one of the first three lectures to one of the above Steps 1, 2 and 3 and to spend the fourth lecture on a related, but different inequality, where we repeat all three steps in this new setting. Step 0, namely in our case the validity of the Sobolev inequality \eqref{eq:sobintro} with some constant, will be taken for granted. In fact, an improved version of this inequality will be proved in the first lecture.

The methods used in these steps vary widely. Those in Step 1 are probably the most robust, while those in Step 2 are probably the most specialized. In the context of the Sobolev inequality Step 3 consists of a combination of the Steps 1b and 2c. In the fourth lecture, however, we will see an example of a functional inequality where additional input is needed in this step.

\medskip

In order to emphasize the general nature of this program, we consider, apart from the Sobolev inequality \eqref{eq:sobintro}, also its fractional counterpart
\begin{equation}
	\label{eq:sobfracintro}
	\int_{\R^d} |(-\Delta)^{s/2} u |^2\,dx \gtrsim \left( \int_{\R^d} |u|^{2d/(d-2s)}\,dx \right)^{(d-2s)/d},
\end{equation}
where $s$ is a real number satisfying $0<s<d/2$. In terms of the Fourier transform
$$
\widehat u(\xi) := (2\pi)^{-d/2} \int_{\R^d} e^{-i\xi\cdot x} u(x)\,dx \,,
$$
the left side of \eqref{eq:sobfracintro} is equal to
\begin{equation}
	\label{eq:fraclapl}
	\int_{\R^d} |(-\Delta)^{s/2} u |^2\,dx = \int_{\R^d} |\xi|^{2s} |\widehat u(\xi)|^2\,d\xi \,.
\end{equation}
Inequality \eqref{eq:sobfracintro} is valid for function $u$ in the homogeneous Sobolev space $\dot H^s(\R^d)$ of tempered distributions whose Fourier transform belongs to $L^1_\loc(\R^d)$ and for which the right side of \eqref{eq:fraclapl} is finite; see, e.g., \cite[Section 1.3]{BaChDa}.

It is easy to see that for $s=1$ this definition of $\dot H^1(\R^d)$ coincides with that given before and that
$$
\int_{\R^d} |(-\Delta)^{s/2} u |^2\,dx = \int_{\R^d} |\nabla u |^2\,dx \,.
$$
Therefore \eqref{eq:sobfracintro} is indeed a generalization of \eqref{eq:sobintro}.

\medskip

In Lecture 4 we will discuss a version of the Sobolev inequality \eqref{eq:sobintro} on the manifold $(\R/T\Z)\times\Sph^{d-1}$, depending on the parameter $T>0$. A specific feature of this case is that for a certain value of $T$ property (2c) fails. It is instructive to see a repetition of the previous steps, both for those values of $T$ where (2c) holds and where it fails.

\medskip

Finally, we note that there have been some developments concerning the stability question for the Sobolev inequality since this course took place at the CIME Summer School in Cetraro in June 2022. We have made the decision not to include those in order to keep the character of these notes more elementary and instead to refer to the preprints \cite{DoEsFiFrLo,Ko1,Ko2}; see also the brief remarks at the end of the third lecture.

\medskip

It is my pleasure to thank the organizers of the summer school, Andrea Cianchi, Vladimir Maz'ya and Tobias Weth, as well as Paolo Salani for their kind invitation, as well as the participants of the school for their interest in these topics. I am grateful to Jean Dolbeault for his help with references and to Tobias K\"onig and Jonas Peteranderl for many useful comments on these notes.


\section*{Lecture 1: Optimizing sequences}

In this first lecture we are interested in the optimization problem
\begin{equation}
	\label{eq:l1quotient}
	S_{d,s} := \inf_{0\neq u\in \dot H^s(\R^d)} \frac{\|(-\Delta)^{s/2}u\|_2^2}{\| u\|_q^2} \,,
\end{equation}
where, as always in this series of lectures,
$$
0<s<\tfrac d2
\qquad\text{and}\qquad
q = \tfrac{2d}{d-2s} \,.
$$
More specifically, we are interested in
\begin{enumerate}
	\item[(a)] existence of an optimizer
	\item[(b)] relative compactness (up to symmetries) of optimizing sequences. 
\end{enumerate}
The difference between (a) and (b) is that for (a) it suffices to find \emph{one} optimizing sequence that converges, whereas for (b) one wants to show that \emph{any} optimizing sequence has a subsequence that converges (up to symmetries); so (b) is stronger than (a). For the arguments in Lecture 2 property (a) would be enough, but in Lecture 3 we need property (b), so this is what we will prove in this lecture.

Let us explain the main difficulty when dealing with the behavior of optimizing sequences and, at the same time, explain the expression `up to symmetries' in (b). A basic strategy in the calculus of variations to solve an optimization problem is to show that from an optimizing sequence one can extract a convergent subsequence and that its limit is an optimizer. Typically, the extracted subsequence converges a priori not in the original sense (here strong convergence in $\dot H^s(\R^d)$), but only in a weaker sense (namely weakly in $\dot H^s(\R^d)$). At this point the (noncompact) symmetries of the variational problem enter. If $u\in\dot H^s(\R^d)$ is a given function and if $(a_n)\subset\R^d$ and $(\lambda_n)\subset\R_+=(0,\infty)$ are sequences with $|a_n|+\lambda_n + \lambda_n^{-1} \to\infty$, then the sequence
$$
\lambda_n^{-d/q} \, u(\lambda_n^{-1}(\cdot - a_n))
$$
converges weakly to zero in $\dot H^s(\R^d)$. Moreover, for these sequences the quotient in \eqref{eq:l1quotient} is independent of $n$, reflecting the translation and dilation invariance of the optimization problem. For instance, if $u$ is an optimizer for \eqref{eq:l1quotient}, then every element of this sequence is an optimizer as well, and we have constructed an optimizing sequence that converges weakly to zero. This explains why in (b) we can hope for relative compactness at most \emph{up to translations and dilations}. These are the symmetries in question.

We now formulate the main result of this lecture, which is due to Lions.

\begin{theorem}\label{comp}
	Let $0<s<\frac d2$ and $q:=\frac{2d}{d-2s}$. Let $(u_n)\subset\dot H^s(\R^d)$ with $\|(-\Delta)^{s/2} u_n\|_2=1$ and $\|u_n\|_q^2 \to S_{d,s}^{-1}$. Then there is a subsequence $(u_{n_k})$, as well as sequences $(a_k)\subset\R^d$ and $(\lambda_k)\subset\R_+$ such that the sequence of functions
	$$
	\lambda_k^{-d/q} u_{n_k}(\lambda_k^{-1} (\cdot - a_k))
	$$
	converges in $\dot H^s(\R^d)$ to an optimizer of \eqref{eq:l1quotient}.
\end{theorem}

Before embarking into the details of the proof, let us give a rough outline of the strategy. We argued above that translations and dilations are a possible loss of compactness. One key step in the proof of Theorem \ref{comp} is to show that translations and dilations (and their combination) are the \emph{only} possible loss of compactness: that is, after applying suitable translations and dilations one can always ensure that the weak limit of a subsequence is nonzero. The mathematical tool here is a refinement of the Sobolev inequality (see Proposition \ref{refined} below), which involves an extra term containing a supremum over dilation and translation parameters. This allows one to translate and dilate the elements of an optimizing sequence such that this extra term stays away from zero, which translates into the integral against a fixed function in $L^{q'}(\R^d)$ (here and throughout: $q'=q/(q-1)$) being bounded away from zero. This implies that the weak limit is nonzero. Thus, we have shown that `there is something somewhere'.

The second step in the proof of Theorem \ref{comp} is to show that `there is nothing else anywhere else'. To explain the argument we assume that no translation, no dilation and no subsequence is necessary and we denote by $u$ the nonzero weak limit of $(u_n)$. Writing $u_n = u+r_n$ one can show that both the numerator and the denominator in the quotient in \eqref{eq:l1quotient} asymptotically decouple in the sense that
$$
\|(-\Delta)^{s/2} u_n\|_2^2 = \|(-\Delta)^{s/2} u\|_2^2 + \|(-\Delta)^{s/2} r_n\|_2^2 + o(1)
\ \ \text{and}\ \
\| u_n\|_q^q = \|u\|_q^q + \|r_n\|_q^q + o(1)\,.
$$
The strict subadditivity of the function $\mu\mapsto \mu^{2/q}$ (since $q>2$) can then be used to show that it is favorable to keep all the mass together, that is, to have $r_n$ tending to zero. This argument is due to Lieb and referred to as the \emph{method of the missing mass}.

We now turn to the details of the proof of Theorem \ref{comp}.


\subsection*{Step 1. There is something somewhere.}

We begin by proving the following \emph{refined Sobolev inequality}.

\begin{proposition}\label{refined}
	Let $0<s<\frac d2$ and $q:=\frac{2d}{d-2s}$. Let $\chi\in C^\infty_c(\overline{\R_+})$ with $\chi=1$ near the origin. Then for all $u\in\dot H^s(\R^d)$,
	\begin{equation}
		\label{eq:refined}
			\| u \|_q \lesssim \left( \sup_{t>0} t^{(d-2s)/4} \| \chi(-t\Delta) u \|_\infty \right)^{1-2/q} \| (-\Delta)^{s/2} u \|_2^{2/q}.
	\end{equation}
\end{proposition}

The supremum in \eqref{eq:refined} is one of the possible, equivalent norms in the Besov space $\dot B^{s-d/2}_{\infty,\infty}(\R^d)$. Our presentation, however, is selfcontained and does not need anything from the theory of these spaces.

We call the inequality in the proposition a `refined' Sobolev inequality since it implies the Sobolev inequality (with nonsharp constant). To see this, we note that the Fourier multiplier $\chi(-t\Delta)$ acts as convolution with the function $t^{-d/2} g(t^{-1/2}\,\cdot)$, where
\begin{equation}
	\label{eq:defk}
	g(x) := (2\pi)^{-d} \int_{\R^d} \chi(|\xi|^2) e^{-i\xi\cdot x} \,d\xi \,.
\end{equation}
Since $g\in L^{q'}(\R^d)$, it follows from H\"older's inequality that
$$
\| \chi(-t\Delta) u \|_\infty \leq \| t^{-d/2} g(t^{-1/2}\,\cdot)\|_{q'} \|u\|_q = t^{-(d-2s)/4} \|g\|_{q'} \|u\|_q \,.
$$
Thus, the supremum in \eqref{eq:refined} is $\leq \|g\|_{q'} \|u\|_q$. Inserting this inequality into \eqref{eq:refined}, we obtain the Sobolev inequality with nonsharp constant.
	
\begin{proof}
	Using the layer cake representation (see, e.g., \cite[Theorem 1.13]{LiLo}), we write
	$$
	\| u \|_q^q = q \int_0^\infty |\{ |u|>\tau \}| \tau^{q-1}\,d\tau
	$$
	and bound for each fixed $\tau>0$ with some $t>0$ to be specified
	$$
	|\{ |u|>\tau \}| \leq |\{ |\chi(-t\Delta) u|>\tau/2 \}| + |\{ |(1-\chi(-t\Delta))u|>\tau/2 \}| \,.
	$$
	In particular, choosing $t=t_\tau$ such that
	$$
	\tau/2 = t_\tau^{-(d-2s)/4} M_u \,,
	\qquad\text{where}\ M_u := \sup_{t>0} t^{(d-2s)/4} \| \chi(-t\Delta) u \|_\infty \,,
	$$
	we see that
	$$
	|\{ |\chi(-t_\tau \Delta) u|>\tau/2 \}| = 0 \,.
	$$
	Meanwhile, we bound
	$$
	|\{ |(1-\chi(-t_\tau \Delta))u|>\tau/2 \}| \leq (2/\tau)^2 \| (1-\chi(-t_\tau\Delta)) u \|_2^2
	$$
	and arrive at
	\begin{align*}
		\| u \|_q^q & \leq q \int_0^\infty (2/\tau)^2 \| (1-\chi(-t_\tau\Delta)) u \|_2^2 \, \tau^{q-1}\,d\tau \\
		& = q \int_{\R^d} |\widehat u(\xi)|^2 \int_0^\infty (2/\tau)^2 |1-\chi(-t_\tau |\xi|^2)|^2 \tau^{q-1}\,d\tau\,d\xi \,,
	\end{align*}
	where $\widehat u$ denotes the Fourier transform of $u$. By scaling, we find
	$$
	q \int_0^\infty (2/\tau)^2 |1-\chi(-t_\tau |\xi|^2)|^2 \tau^{q-1}\,d\tau = C_{d,s,\chi} |\xi|^{2s} M_u^{q-2} \,.
	$$
	The constant $C_{d,s,\chi}$ is finite by the properties of $\chi$, thus proving the proposition.
\end{proof}

We now apply Proposition \ref{refined} in the setting of Theorem \ref{comp}. Let $(u_n)\subset\dot H^s(\R^d)$ be a sequence with
\begin{equation}
	\label{eq:normalized}
	\|(-\Delta)^{s/2} u_n\|_2 = 1
\end{equation}
and
\begin{equation}
	\label{eq:somelq}
	\limsup_{n\to\infty} \| u_n \|_q > 0 \,.
\end{equation}
Note that this is, in particular, satisfied for an optimizing sequence. In this step, however, we will use the weaker property \eqref{eq:somelq} rather than the optimizing property.

Our goal is to show that after a translation and a dilation, $(u_n)$ has a subsequence with nonzero weak limit.

Inserting \eqref{eq:normalized} and \eqref{eq:somelq} into \eqref{eq:refined}, we deduce that
$$
\limsup_{n\to\infty} \sup_{t>0 \,,\ a\in\R^d} t^{(d-2s)/4} |(\chi(-t\Delta) u_n)(a)| >0 \,.
$$
Choosing $t_n>0$ and $a_n\in\R^d$ such that
$$
t_n^{(d-2s)/4} |(\chi(-t_n\Delta) u_n)(a_n)| \geq \tfrac12\, \sup_{t>0 \,,\ a\in\R^d} t^{(d-2s)/4} |(\chi(-t\Delta) u_n)(a)| \,,
$$
we see that 
$$
\limsup_{n\to\infty} t_n^{(d-2s)/4} |(\chi(-t_n\Delta) u_n)(a_n)| >0 \,.
$$
Thus, the translated and dilated functions
$$
\widetilde u_n(x) := t_n^{(d-2s)/4} u_n(t_n^{1/2}(x-a_n))
$$
satisfy
$$
\| (-\Delta)^{s/2} \widetilde u_n \|_2 = \| (-\Delta)^{s/2} u_n \|_2 \,,
\qquad
\| \widetilde u_n \|_q= \| u_n \|_q
$$
and, with $g$ defined in \eqref{eq:defk},
$$
\int_{\R^d} g(x) \widetilde u_n(x)\,dx = t_n^{(d-2s)/4} (\chi(-t_n\Delta) u_n)(a_n) \,.
$$
By weak compactness, we obtain a subsequence and a $\widetilde u\in\dot H^s(\R^d)$ such that $\widetilde u_{n_k} \rightharpoonup \widetilde u$ in $\dot H^s(\R^d)$. We can choose the subsequence in such a way that in addition
$$
\liminf_{k\to\infty} \left| \int_{\R^d} g(x) \widetilde u_{n_k}(x)\,dx \right| > 0 \,.
$$
By the Sobolev inequality, $\widetilde u_{n_k}\rightharpoonup u$ in $L^q(\R^d)$ and, since $g\in L^{q'}(\R^d)$, we find
$$
\left| \int_{\R^d} g(x) \widetilde u(x)\,dx \right| =
\liminf_{k\to\infty} \left| \int_{\R^d} g(x) \widetilde u_{n_k}(x)\,dx \right| > 0 \,.
$$
Thus, $\widetilde u\neq 0$, as we set out to prove.

\subsection*{Step 2. There is nothing else anywhere else.}

Let $(u_n)$ be a minimizing sequence for \eqref{eq:l1quotient}. We normalize the sequence as in \eqref{eq:normalized}. After translations, dilations and passing to a subsequence, we may assume that
$$
u_n \rightharpoonup u \ \text{in}\ \dot H^s(\R^d)
\qquad\text{with}\ u\neq 0 \,.
$$
We write
$$
u_n = u + r_n
\qquad\text{with}\ r_n\rightharpoonup 0 \ \text{in}\ \dot H^s(\R^d) \,.
$$
From the Hilbert space structure of $\dot H^s(\R^d)$ and the normalization \eqref{eq:normalized} we immediately deduce that
\begin{equation}
	\label{eq:mm1}
	t:= \lim_{n\to\infty} \|(-\Delta)^{s/2} r_n\|_2^2
	\qquad\text{exists and satisfies}\qquad
	1 = \|(-\Delta)^{s/2} u \|_2^2 + t \,. 
\end{equation}

We now argue that
\begin{equation}
	\label{eq:mm2}
	m:= \lim_{n\to\infty} \| r_n\|_q^q
	\qquad\text{exists and satisfies}\qquad
	S_{d,s}^{-q/2} = \| u \|_q^q + m \,. 
\end{equation}
Indeed, from the weak convergence $u_n\rightharpoonup u$ in $\dot H^s(\R^d)$ one can deduce that $u_n\to u$ in $L^2_\loc(\R^d)$ (arguing as in \cite[Theorem 8.6]{LiLo}) and then, after passing to a subsequence, $u_n\to u$ almost everywhere. Thus, by the Brezis--Lieb lemma \cite[Theorem 1.9]{LiLo},
\begin{equation}
	\label{eq:bl}
	\lim_{n\to\infty} \int_{\R^d} \left| |u_n|^q - |u|^q - |u_n-u|^q \right|dx = 0 \,.
\end{equation}
The optimizing property of $(u_n)$ and the normalization \eqref{eq:normalized} imply that $\| u_n \|_q^2 \to S_{d,s}^{-1}$. Inserting this information into \eqref{eq:bl}, we obtain \eqref{eq:mm2} along a subsequence. By a standard argument it holds in fact along the full sequence. (Otherwise, there existed a subsequence such that $\lim_{k\to\infty} \|r_{n_k}\|_q$ exists and is different from $S_{d,s}^{-q/2} - \|u\|_q^q$. Repeating the above argument for this subsequence, we arrive at a contradiction.) This proves~\eqref{eq:mm2}.

From the Sobolev inequality, we know that $\|(-\Delta)^{s/2}r_n\|_2^2 \geq S_{d,s}\|r_n\|_q^2$ and therefore
\begin{equation}
	\label{eq:mm3}
	t \geq S_{d,s} m^{2/q} \,.
\end{equation}

Putting \eqref{eq:mm1}, \eqref{eq:mm2} and \eqref{eq:mm3} together, we find
\begin{align*}
	1 & = \|(-\Delta)^{s/2} u \|_2^2 + t \geq \|(-\Delta)^{s/2} u \|_2^2 + S_{d,s} m^{2/q} = 
	\|(-\Delta)^{s/2} u \|_2^2 + \left( 1 - S_{d,s}^{q/2} \|u\|_q^q \right)^{2/q} \\
	& \geq 	\|(-\Delta)^{s/2} u \|_2^2 + 1 - S_{d,s} \|u\|_q^2 \,. 
\end{align*}
In the last inequality we used the elementary fact that
\begin{equation}
	\label{eq:elem}
	(a+b)^{2/q} \leq a^{2/q} + b^{2/q}
	\qquad\text{for all}\ a,b\geq 0 \,,
\end{equation}
which relies on the fact that $q\geq 2$. Thus, we have shown that $\|(-\Delta)^{s/2} u \|_2^2 \leq S_{d,s} \|u\|_q^2$, which, taking into account that $u\neq 0$, implies that $u$ is an optimizer for $S_{d,s}$. Thus, we have accomplished our first goal, namely showing the existence of an optimizer.

To reach our second goal, namely showing relative compactness of optimizing sequences, we observe that, since $q>2$, equality in \eqref{eq:elem} occurs only when $a$ or $b$ is zero. Since in our application $b=S_{d,s}^{q/2}\|u\|_q^q$ is nonzero, we conclude that $a=1- S_{d,s}^{q/2} \|u\|_q^q$ is zero. According to \eqref{eq:mm2} this means that $m=0$. Since we also need to have equality in \eqref{eq:mm3}, we conclude that $t=0$, that is, $r_n\to 0$ in $\dot H^s(\R^d)$. Note that this is \emph{strong} convergence. Thus, we have shown $u_n\to u$ in $\dot H^s(\R^d)$, as claimed.


\subsection*{Appendix: The Hardy--Littlewood--Sobolev inequality}

Several results mentioned in this series of lectures were originally proved for a family of functional inequalities called \emph{Hardy--Littlewood--Sobolev inequalities}, which is in a certain sense dual to the family of Sobolev inequalities considered here. While we have consistently used the latter formulation, it is worthwhile to explain this connection.

The family of Hardy--Littlewood--Sobolev (HLS) inequalities is a two-parameter family of inequalities, depending on parameters $0<\lambda<d$ and $1<p<\frac{d}{d-\lambda}$, and states that
$$
\left\| |x|^{-\lambda} * f \right\|_q \lesssim \left\| f \right\|_p
\qquad\text{with}\ \tfrac{1}{q} = \tfrac1p - \tfrac{d-\lambda}{d} \,.
$$
This is a generalization of Young's convolution inequality, where the functions $|x|^{-\lambda}$ do not belong to the Lebesgue space $L^{d/\lambda}(\R^d)$, but only to its weak counterpart.

Relevant for us are two one-parameter families, namely those corresponding to $p=2$ and to $q=2$. The inequalities in these cases are dual to each other, which means, in particular, that their optimal constants coincide. In this appendix we will explain what this duality implies for the questions of existence and characterization of optimizers, as well as the relative compactness of optimizing sequences. (There is yet another family, corresponding to $q=p'$, that is equivalent, but we will not discuss it here.)

The relation between the Sobolev inequalities discussed in the main part of these lectures and the HLS inequalities discussed in this appendix comes from the well-known fact (see, e.g., \cite[Theorem 5.9]{LiLo}) that the operator $(-\Delta)^{-s}$ has integral kernel
\begin{equation}
	\label{eq:intkernelinvfraclap}
	(-\Delta)^{-s}(x,x') = 2^{-2s}\, \pi^{-d/2}\, \frac{\Gamma(\frac d2 -s)}{\Gamma(s)}\, |x-x'|^{-d+2s} \,.
\end{equation}
Thus, writing the Sobolev inequality
$$
\| (-\Delta)^{s/2} u \|_2^2 \geq S_{d,s} \|u\|_q^2
\qquad\text{for all}\ u\in \dot H^s(\R^d) \,,
$$
in the equivalent form
\begin{equation}
	\label{eq:hlsprimal}
	\left\| (-\Delta)^{-s/2} f \right\|_q \leq S_{d,s}^{-1/2} \| f\|_2 
	\qquad\text{for all}\ f\in L^2(\R^d) \,,
\end{equation}
we obtain the HLS inequality with $p=2$. Moreover, $S_{d,s}$ being the sharp constant in the Sobolev inequality means that $S_{d,s}^{-1/2}$ is the norm of the operator $(-\Delta)^{-s/2}$ from $L^2(\R^d)$ to $L^q(\R^d)$, and therefore, up to the prefactor in the integral kernel, the optimal constant in the HLS inequality with the exponent $2$ on the right side.

We apply duality and pass to the HLS inequality with the exponent $2$ on the left side of the inequality. Duality implies that $S_{d,s}^{-1/2}$ is equal to the norm of the operator $(-\Delta)^{-s/2}$ from $L^{q'}(\R^d)$ to $L^2(\R^d)$,
\begin{equation}
	\label{eq:hlsdual}
	\left\| (-\Delta)^{-s/2} g \right\|_2 \leq S_{d,s}^{-1/2} \| g\|_{q'} 
	\qquad\text{for all}\ g\in L^{q'}(\R^d) \,.
\end{equation}
This is the form of the HLS inequality in which it is most naturally studied in connection with sharp constants, compactness and conformal invariance. Note also that, by \eqref{eq:intkernelinvfraclap},
$$
\left\| (-\Delta)^{-s/2} g \right\|_2^2 = 2^{-2s}\, \pi^{-d/2}\, \frac{\Gamma(\frac d2 -s)}{\Gamma(s)}\, \iint_{\R^d\times\R^d} \frac{g(x)\,g(x')}{|x-x'|^{d-2s}}\,dx\,dx' \,.
$$

We now show that optimizing sequences for \eqref{eq:hlsprimal} and \eqref{eq:hlsdual} are in one-to-one correspondence with each other and that convergence of optimizing sequences is equivalent for both problems. We carry this out in a more general setting.

\begin{lemma}\label{dual}
	Let $\mathcal H$ be a Hilbert space, let $X$ be a measure space and $1<q<\infty$. Let $A:\mathcal H\to L^q(X)$ be a bounded linear operator, let $A^*: L^{q'}(X)\to\mathcal H$ be its adjoint and let $\alpha:= \|A\|=\|A^*\|$.
	\begin{enumerate}
		\item[(a1)] If $f\in\mathcal H$ satisfies $\|f\|_{\mathcal H}=1$ and $\|Af\|_q=\alpha$, then 
		$$
		g:=\|Af\|_q^{1-q} |Af|^{q-2}Af
		$$
		satisfies $\|g\|_{q'}=1$ and $\|A^* g\|_{\mathcal H}= \alpha$.
		\item[(a2)] If $(f_n)\subset\mathcal H$ satisfies $\|f_n\|_{\mathcal H}=1$ and $\|Af_n\|_q\to\alpha$, then 
		$$
		g_n:=\|Af_n\|_q^{1-q} |Af_n|^{q-2}Af_n
		$$
		satisfies $\|g_n\|_{q'}=1$ and $\|A^* g_n\|_{\mathcal H}\to \alpha$. If, in addition, $g_n\to g$ in $L^{q'}(X)$, then $f_n\to \|A^* g\|_{\mathcal H}^{-1} A^* g$ in $\mathcal H$.
		\item[(b1)] If $g\in L^{q'}(X)$ satisfies $\|g\|_{q'}=1$ and $\|A^*g\|_{\mathcal H} = \alpha$, then
		$$
		f := \|A^* g\|_{\mathcal H}^{-1} A^* g
		$$
		satisfies $\|f\|_{\mathcal H}=1$ and $\|A f\|_q= \alpha$.
		\item[(b2)] If $(g_n)\subset L^{q'}(X)$ satisfies $\|g_n\|_{q'}=1$ and $\|A^*g_n\|_{\mathcal H} \to \alpha$, then
		$$
		f_n := \|A^* g_n\|_{\mathcal H}^{-1} A^* g_n
		$$
		satisfies $\|f_n\|_{\mathcal H}=1$ and $\|A f_n\|_q\to \alpha$. If, in addition, $f_n\to f$ in $\mathcal H$, then $g_n\to \|Af\|_q^{1-q}|Af|^{q-2} Af$ in $L^{q'}(X)$.
	\end{enumerate}
\end{lemma}

We apply this lemma with $\mathcal H = L^2(\R^d)$, $X=\R^d$ and $A=(-\Delta)^{-s/2}$. We infer that the optimal constant in \eqref{eq:hlsprimal} is attained if and only if that in \eqref{eq:hlsdual} is attained, and that optimizers are in one-to-one correspondence. Moreover, convergence of an optimizing sequence for one inequality is equivalent to that for the other and, in particular, relative compactness (up to symmetries) for one inequality implies the same for the other. This explains what we mean by the `equivalence' of the two optimization problems.

\begin{proof}
	The lemma is valid both when the underlying field is that of real and complex numbers. So, while in the rest of these lectures we deal exclusively with real-valued functions, here we will use complex notation.
	
	The proof of (a1) and (b1) is a variation of the proof of the first part of (a2) and (b2), so we only prove the latter. For (a2) we have, clearly, $\|g_n\|_{q'}=1$ and $\|A^*g_n\|_{\mathcal H} \leq \alpha \|g_n\|_{q'} = \alpha$. Meanwhile, since $\|f_n\|_{\mathcal H}=1$,
	$$
	\| A^* g_n\|_{\mathcal H} \geq \langle f_n, A^* g_n\rangle_{\mathcal H} = \int_X \overline{(Af_n)}\, g_n \,dx = \|Af_n\|_q = \alpha + o(1) \,.
	$$
	Thus, $\|A^* g_n\|_{\mathcal H}\to \alpha$, as claimed. Now assume that $g_n\to g$ in $L^{q'}(X)$. Then $A^*g_n\to A^*g$ in $\mathcal H$ and, passing to the limit in the above chain of inequalities $\alpha \geq \|A^* g_n\|_{\mathcal H} \geq \langle f_n, A^* g_n\rangle  = \alpha + o(1)$, we see that any weak limit point $f$ of $(f_n)$ satisfies $\| A^* g \|_{\mathcal H} = \langle f, A^* g\rangle$. Since $\|f\|_{\mathcal H}\leq 1$, we see that we have equality in the Schwarz inequality and consequently $f=\|A^* g\|_{\mathcal H}^{-1} A^* g$. In particular, $\|f\|_{\mathcal H} = 1 = \|f_n\|_{\mathcal H}$, which implies that the convergence to $f$ is strong. Uniqueness of the limit point proves that in fact the full sequence converges to $f$. This completes the proof of (a2).
	
	The proof of (b2) is similar to that of (a2). The assumption $1<q<\infty$ implies that we still have weak compactness; see \cite[Theorem 2.18]{LiLo} or \cite[Proposition 4.49]{vN}. We also make use of the characterization of equality in H\"older's inequality \cite[Theorem 2.3]{LiLo}. We omit the details.	
\end{proof}


\subsection*{Bibliographic remarks}

The existence of an optimizer for \eqref{eq:l1quotient} for general $s$ is due to Lieb \cite{Li} in the dual formulation of an HLS inequality. Lieb's proof uses the technique of symmetric decreasing rearrangement. Even if this argument does not yield the relative compactness of general optimizing sequences, several ingredients of it are still crucial for the latter problem.

The relative compactness of optimizing sequences is due to Lions; see \cite{Lio1} for the case $s=1$ and \cite{Lio2} for the case of general $s$ in the dual formulation.

The proof presented here is different from Lions's original one, although there are some similarities in the overall structure. In Lions's terminology, showing that there is something somewhere is excluding `vanishing' and showing that there is nothing else anywhere else is excluding `dichotomy'.

The first step in the proof of Theorem \ref{comp} that we presented is close to an argument that appears in \cite{KiVi} and has its roots in the work of G\'erard \cite{Ge}. Both \cite{KiVi} and \cite{Ge} iterate the argument of extracting a weak limit to obtain a so-called profile decomposition, which is of importance in several areas of analysis. As shown in \cite{FrLi} and here, to prove the relative compactness up to symmetries of optimizing sequences, a full profile decomposition is not necessary and it suffices to extract one profile. The refined Sobolev inequality in Proposition \ref{refined} is due to G\'erard, Meyer and Oru \cite{GeMeOr} and our presentation of the proof follows \cite[Theorem 1.43]{BaChDa}. For an alternative proof for $s=1$, which extends to the $p$-norm of the gradient, we refer to \cite{Le}.

Instead of the refined Sobolev inequality in terms of Besov spaces, one can also use an improvement of the Sobolev inequality in the scale of Lorentz spaces, namely,
$$
\|(-\Delta)^{s/2} u \|_2 \gtrsim \| u \|_{L^{q,2}} \,.
$$
Since $\|u\|_{L^q} = \| u \|_{L^{q,q}} \lesssim \| u \|_{L^{q,2}}^{2/q} \| u \|_{L^{q,\infty}}^{1-2/q}$, we obtain that, along a minimizing sequence, $\| u_n \|_{L^{q,\infty}}\gtrsim 1$. From this one can deduce the existence of a nontrivial weak limit point. Indeed, for $H^1(\R^d)$ this is a result of Lieb \cite{Li1}, but a similar proof works for $\dot H^1(\R^d)$, $d\geq 3$. For general $s$, see \cite{BeFrVi}. A Lorentz space improvement is also used in \cite{Li} for a similar, but slightly different purpose.

Yet another proof, based on a different kind of refined inequality, will be presented in the appendix to the next lecture.

The second step in the proof of Theorem \ref{comp} that we presented, including the Brezis--Lieb lemma and the use of the elementary inequality \eqref{eq:elem}, is taken from Lieb's proof of the existence of an optimizer \cite{Li}. The final argument, upgrading weak convergence to strong convergence, is attributed to Browder in \cite{BrNi}.

For a recent review of compactness methods similar to those employed in this lecture we refer to \cite{Sa}.


\section*{Lecture 2: Optimizers}

Our main goal in this lecture is to solve the optimization problem
\begin{equation}
	\label{eq:l2quotient}
	S_{d,s} := \inf_{0\neq u\in \dot H^s(\R^d)} \frac{\|(-\Delta)^{s/2}u\|_2^2}{\| u\|_q^2}
\end{equation}
where, as always in this series of lectures,
$$
0<s<\tfrac d2
\qquad\text{and}\qquad
q = \tfrac{2d}{d-2s} \,.
$$
By `solving the optimization problem' we mean that we will compute the number $S_{d,s}$ explicitly and characterize all $u\in\dot H^s(\R^d)$ for which the infimum is achieved. It is quite remarkable that this is possible. The following theorem is due to Lieb.

\begin{theorem}\label{charmin}
	Let $0<s<\frac d2$. Then	
	$$
	S_{d,s} = \frac{\Gamma(\frac d2+s)}{\Gamma(\frac d2-s)}\, |\Sph^d|^{2s/d} \,.
	$$
	Moreover, the infimum in \eqref{eq:l2quotient} is attained if and only if there are $a\in\R^d$, $\lambda>0$ and $c\in\R\setminus\{0\}$ such that
	\begin{equation}
		\label{eq:charmin}
		u(x) = c \lambda^{-(d-2s)/2} Q(\lambda^{-1}(x-a)) \,,
	\end{equation}
	where
	$$
	Q(x) = \left( \frac{2}{1+|x|^2} \right)^{(d-2s)/2}.
	$$
\end{theorem}

In fact, we will present the proof of a stronger result, which says that the optimization problem \eqref{eq:l2quotient} does not have any local minimizers except for those stated in the theorem. Here, a function $0\neq u_*\in\dot H^s(\R^d)$ is called a \emph{local minimizer} of \eqref{eq:l2quotient} if for all $\phi\in\dot H^s(\R^d)$
$$
\frac{d}{dt}\Big|_{t=0}  \frac{\|(-\Delta)^{s/2} (u_*+t\phi)\|_2^2}{\|u_*+t\phi\|_q^2} = 0
\qquad\text{and}\qquad
\frac{d^2}{dt^2}\Big|_{t=0}  \frac{\|(-\Delta)^{s/2} (u_*+t\phi)\|_2^2}{\|u_*+t\phi\|_q^2} \geq 0 \,.
$$

\begin{theorem}\label{charlocmin}
	A function $0\neq u_*\in\dot H^s(\R^d)$ is a local minimizer of \eqref{eq:l2quotient} if and only if it is of the form \eqref{eq:charmin} for some $a\in\R^d$, $\lambda>0$ and $c\in\R\setminus\{0\}$.
\end{theorem}

The proof of this theorem that we present in this lecture relies on a `hidden' symmetry. This symmetry will be discussed next in detail.


\subsection*{Conformal invariance}

In the previous lecture we have already discussed the invariance of our optimization problem under translations and dilations. Another obvious invariance concerns that by orthogonal transformations of $\R^d$. There is a nonobvious invariance as well, namely under the \emph{inversion on the unit sphere} $x\mapsto x/|x|^2$, which is implemented on functions $u$ on $\R^d$ by
\begin{equation}
	\label{eq:inversion}
	\tilde u(x) := |x|^{-d+2s} u(|x|^{-2} x) \,.
\end{equation}
Clearly, $\tilde u$ belongs to $L^q(\R^d)$ if and only if $u$ does, and we have
\begin{equation}
	\label{eq:inversionq}
	\| \tilde u \|_q = \| u \|_q \,.
\end{equation}
The important observation is that $\tilde u$ belongs to $\dot H^s(\R^d)$ if and only if $u$ does, and that in this case
\begin{equation}
	\label{eq:inversionkin}
	\|(-\Delta)^{s/2} \tilde u\|_2 = \|(-\Delta)^{s/2} u \|_2 \,.
\end{equation}
For $s=1$, this can be proved directly by replacing $(-\Delta)^{1/2}$ under the norm by $\nabla$. For general $s$, in particular noninteger ones, a direct proof is more tedious and it is preferable to deduce this result from the discussion below. In the following we will not consider \eqref{eq:inversionkin} as proved, but rather use it as a motivation.

We recall a theorem of Liouville (see, e.g., \cite[Theorem A.3.7]{BePe}) that says that the Euclidean motions, together with dilations and the inversion on the unit sphere, generate the so-called \emph{conformal group}, that is, the group of deformations that preserve angles. Thus, if $\Phi:\R^d\cup\{\infty\}\to\R^d\cup\{\infty\}$ is conformal with Jacobian $J_\Phi:=|\det D\Phi|$ and if $u\in\dot H^s(\R^d)$, then
$$
u_\Phi(x) := J_\Phi(x)^{1/q} u(\Phi(x))
$$
belongs to $\dot H^s(\R^d)$ and
\begin{equation}
	\label{eq:confinvrd}
	\|(-\Delta)^{s/2} u_\Phi\|_2 = \|(-\Delta)^{s/2} u \|_2 \,,
	\qquad
	\| u_\Phi \|_q = \| u \|_q \,.
\end{equation}
We emphasize that, by Liouville's theorem, \eqref{eq:confinvrd} is a consequence of \eqref{eq:inversionq} and \eqref{eq:inversionkin}. Therefore, \eqref{eq:confinvrd} is not considered proved at this point of the lecture (unless for $s=1$). A proof will be provided later on.

We should stress that our lectures do not really rely on Liouville's theorem. If we call a \emph{M\"obius transformation} any element of the subgroup of the conformal group generated by Euclidean motions, dilations and the inversion on the unit sphere, then everything we say remains valid when we substitute `conformal' by `M\"obius' and, in the setting of the sphere that will appear momentarily, `conformal' by `conjugate of M\"obius under stereographic projection'. We have opted for the use of `conformal' for the sake of simplicity of the terminology and adherence to tradition in this field.

The conformal invariance allows us to reformulate the variational problem on the unit sphere $\Sph^d$ in $\R^{d+1}$. The inverse stereographic projection $\mathcal S:\R^d\to\Sph^d$ is given by
$$
\mathcal S_j(x):= \frac{2x_j}{1+|x|^2} \,,\ j=1,\ldots,d \,,
\qquad
\mathcal S_{d+1}(x) := \frac{1-|x|^2}{1+|x|^2} \,.
$$
This map is conformal and has Jacobian
$$
J_\mathcal S(x) = \left( \frac{2}{1+|x|^2} \right)^d \,.
$$
Sometimes we will extend $\mathcal S$ by $\mathcal S(\infty):=(0,\ldots,0,-1)^{\rm T}$ to a map $\R^d\cup\{\infty\}\to\Sph^d$.

Assume that a function $u$ on $\R^d$ and a function $U$ on $\Sph^d$ are related via
\begin{equation}
	\label{eq:uU}
	u(x)= J_\mathcal S(x)^{1/q}\, U(\mathcal S(x)) \,.
\end{equation}
Then clearly $U\in L^q(\Sph^d)$ if and only if $u\in L^q(\R^d)$, and
\begin{equation}
	\label{eq:uUlq}
	\| U \|_q = \| u \|_q \,.
\end{equation}
Here the $L^q$-norm on $\Sph^d$ is defined with respect to the (unnormalized) surface measure; moreover, integration with respect to this measure is denoted by $d\omega$.

The crucial point is that $u\in\dot H^s(\R^d)$ if and only if $U\in H^s(\Sph^d)$, and that in this case
\begin{equation}
	\label{eq:uUkin}
	\mathcal E_s[U] = \|(-\Delta)^{s/2} u \|_2^2
\end{equation}
for a certain energy functional $\mathcal E_s$ that we are about to introduce.

We recall that the space $L^2(\Sph^d)$ is the orthogonal direct sum of subspaces of spherical harmonics; see, e.g., \cite[Section IV.2]{StWe} or \cite[Subsection 3.8.2]{FrLaWe}. For $\ell\in\N_0$ we denote by $P_\ell$ the orthogonal projection in $L^2(\R^d)$ onto the subspace of spherical harmonics of degree $\ell$. We define for $U\in H^s(\Sph^d)$,
$$
\mathcal E_s[U] := \sum_{\ell=0}^\infty \frac{\Gamma(\ell + \frac{d}{2} + s)}{\Gamma(\ell + \frac{d}{2} - s)}\, \| P_\ell U \|_2^2 \,.
$$
Since the quotient of gamma functions in this definition is positive and grows like $\ell^{2s}$ (by Stirling's formula), we see that $\mathcal E_s[U]$ is equivalent to $\|U\|_{H^s(\Sph^d)}^2$. Moreover, for $s=1$ we see that
\begin{equation}
	\label{eq:e1}
	\mathcal E_1[U] = \int_{\Sph^d} \left( |\nabla U|^2 + \tfrac{d(d-2)}{4} U^2 \right)d\omega \,,
\end{equation}
where $\nabla$ denotes the gradient in the sense of Riemannian geometry. Identity \eqref{eq:e1} follows from the functional equation of the gamma function and the fact that $(-\Delta)P_\ell = \ell(\ell+d-1)P_\ell$, where $-\Delta$ is the Laplace--Beltrami operator.

Identity \eqref{eq:uUkin} for $s=1$ (with the left side replaced by the right side of \eqref{eq:e1}) follows by a straightforward computation. As a preparation for the proof for general $s$ we introduce the operator
\begin{equation}
	\label{eq:defas}
	A_s := \sum_{\ell=0}^\infty \frac{\Gamma(\ell + \frac{d}{2} + s)}{\Gamma(\ell + \frac{d}{2} - s)}\, P_\ell \,.
\end{equation}
This is an unbounded, selfadjoint operator in $L^2(\Sph^d)$, which is positive definite and has form domain $H^s(\Sph^d)$ and operator domain $H^{2s}(\Sph^d)$. The operator $A_s$ is connected with the quadratic form $\mathcal E_s$ by
$$
\langle U, A_s U \rangle = \mathcal E_s[U] \,,
$$
valid for $U\in H^{2s}(\Sph^d)$ (or even for $U\in H^s(\Sph^d)$, provided one interprets the left side as the duality pairing between $H^s$ and $H^{-s}$). Note also that
$$
A_1 = -\Delta + \tfrac{d(d-2)}{4} \,.
$$
This operator is called the \emph{conformal Laplacian}. The operator $A_2$ is called the \emph{Paneitz operator}. The family of operators $A_s$ with integer $s$ is referred to as the family of \emph{GJMS operators} on the sphere.

Having introduced the necessary objects, we can now show the claimed identity~\eqref{eq:uUkin}.

\begin{proof}[Proof of \eqref{eq:uUkin}]
	We denote by $T$ the operator $U\mapsto u$ and by $T^*$ its $L^2$-adjoint, that is,
	$$
	T U= J_\mathcal S^{1/q}\, (U\circ\mathcal S) \,,
	\qquad
	T^* u = J_{\mathcal S^{-1}}^{1/q'} \, (u\circ\mathcal S^{-1}) \,.
	$$
	Then \eqref{eq:uUkin} can be written as $A_s = T^* (-\Delta)^{s}\, T$, which is equivalent to
	\begin{equation}
		\label{eq:equalityinverses}
		A_s^{-1} = T^{-1} (-\Delta)^{-s}\, T^{-*} \,.
	\end{equation}
	Combining the form \eqref{eq:intkernelinvfraclap} of the integral kernel of $(-\Delta)^{-s}$ with the fact that
	$$
	|\mathcal S(x) - \mathcal S(x')|^2 = \frac{2}{1+|x|^2}\, |x-x'|^2 \, \frac{2}{1+|x'|^2} \,,
	$$
	we obtain
	$$
	T^{-1} (-\Delta)^{-s}\, T^{-*}(\omega,\omega') = 2^{-2s}\, \pi^{-d/2}\, \frac{\Gamma(\frac d2 -s)}{\Gamma(s)}\, |\omega-\omega'|^{-d+2s} \,.
	$$
	Since this kernel only depends on $\omega\cdot\omega'$, the latter operator is diagonal with respect to the decomposition of $L^2(\Sph^d)$ into spherical harmonics and, according to Lemma \ref{evsphere} below, its eigenvalue on the space of spherical harmonics of degree $\ell$ is equal to	
	$$
	\frac{\Gamma(\ell+\frac d2 -s)}{\Gamma(\ell+\frac d2 +s)} \,,
	$$
	which is the same as the eigenvalue of $A_s^{-1}$ on that space. This implies \eqref{eq:equalityinverses}.
\end{proof}

\begin{lemma}\label{evsphere}
	Let $0<\alpha<\frac d2$ and $\ell\in\N_0$. The eigenvalue of the operator in $L^2(\Sph^d)$ with kernel $(1-\omega\cdot\omega')^{-\alpha}$ on the subspace $\ran P_\ell$ is given by
	$$
	(4\pi)^{d/2}\, 2^{-\alpha} \, \frac{\Gamma(\frac d2-\alpha)}{\Gamma(\alpha)}\, \frac{\Gamma(\ell+\alpha)}{\Gamma(\ell+d-\alpha)} \,.
	$$
\end{lemma}

\begin{proof}
	This is a computation based on the Funk--Hecke formula and facts about Gegenbauer polynomials. Its details can be found in \cite[Corollary 4.3]{FrLi0}. Here we only explain why formula \cite[(4.7)]{FrLi0} is the same as that in the lemma. First, using the duplication formula for the gamma function, we see that $\kappa_N = (4\pi)^{N/2}$ for all $N\geq 1$. Second, using the reflection formula for the gamma function twice, we see that
	$$
	\frac{(-1)^\ell\,\Gamma(1-\alpha)}{\Gamma(-\ell+1-\alpha)} = \frac{(-1)^\ell\,\sin\pi(\ell+\alpha)}{\sin\pi\alpha}\,\frac{\Gamma(\ell+\alpha)}{\Gamma(\alpha)} = \frac{\Gamma(\ell+\alpha)}{\Gamma(\alpha)} \,.
	$$
	This leads to the formula in the lemma.
\end{proof}

The identities \eqref{eq:uUlq} and \eqref{eq:uUkin} allow us to reformulate our optimization problem \eqref{eq:l2quotient} on Euclidean space as an optimization problem on the sphere,
\begin{equation}
	\label{eq:l2quotsphere}
	S_{d,s} = \inf_{0\neq U\in H^s(\Sph^d)} \frac{\mathcal E_s[U]}{\| U \|_q^2} \,.
\end{equation}
Moreover, optimizers for the problems on $\R^d$ and on $\Sph^d$ are in one-to-one correspondence via \eqref{eq:uU}.

At this point we can give the long delayed proof of the invariance of \eqref{eq:l2quotient} under inversions on the unit sphere.

\begin{proof}[Proof of \eqref{eq:inversionkin}]
	If $\tilde U$ is related to $\tilde u$ as in \eqref{eq:uU}, then
	$$
	\tilde U(\omega) = U(\omega_1,\ldots,\omega_d,-\omega_{d+1}) \,.
	$$
	That is, the inversion on the unit sphere for the $\R^d$-problem corresponds to the reflection on $\{\omega_{d+1}=0\}$ for the $\Sph^d$-problem. Since $\|P_\ell \tilde U\|_2 = \|P_\ell U\|_2$ we deduce that $\mathcal E_s[\tilde U]=\mathcal E_s[U]$. The claimed equality \eqref{eq:inversionkin} is therefore a consequence of \eqref{eq:uUkin}.
\end{proof}

At this point the conformal invariance of the problem on $\R^d$, namely \eqref{eq:confinvrd}, is completely proved. 

We also obtain the conformal invariance on $\Sph^d$. That is, if $\Psi:\Sph^d\to\Sph^d$ is a conformal transformation and if $U\in H^s(\Sph^d)$, then
$$
U_\Psi(\omega):= J_\Psi(\omega)^{1/q}\, U(\Psi(\omega))
$$
belongs to $H^s(\Sph^d)$ and
$$
\mathcal E_s[U_\Psi] = \mathcal E_s[U] \,,
\qquad
\| U_\Psi \|_q = \| U \|_q \,.
$$
This follows from the corresponding result on $\R^d$ by noting that $\Psi$ is a conformal transformation of $\Sph^d$ if and only if $\Phi:=\mathcal S^{-1} \Psi \mathcal S$ is a conformal transformation of $\R^d$.

\begin{remark*}
	We emphasize that for $s=1$ the argument given above is unnecessarily complicated. In this case we first verify directly the invariance under inversions \eqref{eq:inversionkin} and deduce the conformal invariance \eqref{eq:confinvrd} on $\R^d$. Then we verify directly identity \eqref{eq:uUkin} with left side given by the right side of \eqref{eq:e1}, and obtain as a consequence of the conformal invariance on $\R^d$ that on $\Sph^d$. In particular, Lemma \ref{evsphere} is not needed.
\end{remark*}

\begin{example*}
	Let $Q$ be as in Theorem \ref{charmin} and let $u(x) = c \lambda^{-(d-2s)/2} Q(\lambda^{-1}(x-a))$ with $a\in\R^d$, $\lambda>0$ and $c\in\R$. Then the corresponding $U$ on $\Sph^d$ is given by
	$$
	U(\omega) = c \left( \frac{\sqrt{1-|\zeta|^2}}{1-\zeta\cdot\omega} \right)^{(d-2s)/2}
	$$
	with $\zeta:=(2\eta - \lambda^2(1 + \eta_{n+1})e_{n+1})/(2 + \lambda^2(1 + \eta_{n+1}))$ and $\eta:=\mathcal S(a)$. This follows by a direct computation. Moreover, it is a simple exercise to show that the map $\R^d\times\R_+ \ni (a,\lambda)\mapsto \zeta \in \{z\in\R^{d+1}:\ |z|<1\}$ is a bijection. 
\end{example*}

As a final preliminary we note that
\begin{equation}
	\label{eq:detsphere}
	\left\{ J_\Psi :\ \Psi \ \text{conformal transformation of $\Sph^d$} \right\}
	= \left\{\left( \frac{\sqrt{1-|\zeta|^2}}{1-\zeta\cdot\omega} \right)^d :\ |\zeta|<1 \right\}.
\end{equation}
To prove this, we note that $\Phi=\mathcal S^{-1}\Psi\mathcal S$ gives a bijection between conformal transformations $\Psi$ of $\Sph^d$ and $\Phi$ of $\R^d$. Therefore, the claim is that
$$
J_{\mathcal S}(\Phi(\mathcal S^{-1}(\omega)))\, J_\Phi(\mathcal S^{-1}(\omega))\, J_{\mathcal S^{-1}}(\omega) = \left( \frac{\sqrt{1-|\zeta|^2}}{1-\zeta\cdot\omega} \right)^d,
$$
in the sense that, as $\Phi$ runs through the conformal group, $\zeta$ runs through the unit ball. Equivalently,
$$
J_{\mathcal S}(\Phi(x))\, J_\Phi(x)\, = \left( \frac{\sqrt{1-|\zeta|^2}}{1-\zeta\cdot\mathcal S(x)} \right)^d J_{\mathcal S}(x) \,.
$$
By Liouville's theorem it suffices to verify the latter identity separately for Euclidean motions, dilations and the inversion on the unit sphere. This is a tedious, but straightforward computation.


\subsection*{An equivalent formulation of the theorem}

After all these preparations, we now formulate the analogue of Theorem \ref{charlocmin} on the sphere. A function $0\neq U_*\in H^s(\Sph^d)$ is called a \emph{local minimizer} of \eqref{eq:l2quotsphere} if for all $\phi\in H^s(\Sph^d)$,
\begin{equation}
	\label{eq:locminsphere}
	\frac{d}{dt}\Big|_{t=0} \frac{\mathcal E_s[U_*+t\phi]}{\| U_*+t\phi\|_q^2} = 0
	\qquad\text{and}\qquad
	\frac{d^2}{dt^2}\Big|_{t=0} \frac{\mathcal E_s[U_*+t\phi]}{\| U_*+t\phi\|_q^2} \geq 0 \,.
\end{equation}
The conformal invariance discussed above shows that if $u$ and $U$ are related by \eqref{eq:uU}, then $U$ is a local minimizer of \eqref{eq:l2quotsphere} if and only if $u$ is a local minimizer of \eqref{eq:l2quotient}.

\begin{theorem}\label{charlocminsphere}
	A function $0\neq U_*\in H^s(\Sph^d)$ is a local minimizer of  \eqref{eq:l2quotsphere} if and only if $U_* = cJ_\Psi^{1/q}$ for a conformal transformation $\Psi$ of $\Sph^d$ and a constant $c\in\R\setminus\{0\}$.
\end{theorem}

Theorem \ref{charlocmin} is an immediate consequence of Theorem \ref{charlocminsphere}. The equivalence of the characterization of optimizers follows from the above example and \eqref{eq:detsphere}.

\medskip

To prepare for the proof of Theorem \ref{charlocminsphere}, we compute the derivatives appearing in the definition of a local minimizer. We begin with the case $s=1$, where we use the form \eqref{eq:e1} of the functional. We have
$$
\frac{d}{dt}\Big|_{t=0} \frac{\mathcal E_1[U_*+t\phi]}{\| U_*+t\phi\|_q^2}
= \frac{2}{\| U_*\|_q^2} \int_{\Sph^d} \left( \nabla\phi\cdot\nabla U_* + \tfrac{d(d-2)}{4}\phi U_* - \tfrac{\mathcal E_1[U_*]}{\|U_*\|_q^q} \phi |U_*|^{q-2} U_*\right)d\omega\,,
$$
so the first condition in \eqref{eq:locminsphere} is satisfied if and only if $U_*$ solves the equation
\begin{equation}
	\label{eq:firstder}
	-\Delta U_* +  \tfrac{d(d-2)}{4} U_* - \tfrac{\mathcal E_1[U_*]}{\|U_*\|_q^q} |U_*|^{q-2} U_* = 0
	\qquad\text{on}\ \Sph^d \,.
\end{equation}
Moreover, for $U_*$ for which the first derivative of $\mathcal E_1$ vanishes, we compute
\begin{align*}
	\frac{d^2}{dt^2}\Big|_{t=0} \frac{\mathcal E_1[U_*+t\phi]}{\| U_*+t\phi\|_q^2}
	= \frac{2}{\| U_*\|_q^2} & \left( \int_{\Sph^d} \left( |\nabla\phi|^2 + \tfrac{d(d-2)}{4} \phi^2 - (q-1) \tfrac{\mathcal E_1[U_*]}{\|U_*\|_q^q} |U_*|^{q-2} \phi^2 \right)d\omega \right. \\
	& \quad \left. + (q-2) \tfrac{\mathcal E_1[U_*]}{\|U_*\|_q^{2q}} \left( \int_{\Sph^d} |U_*|^{q-2} U_*\phi\,d\omega \right)^2 \right).
\end{align*}
Thus, the second condition in \eqref{eq:locminsphere} is satisfied if and only if the operator
\begin{equation}
	\label{eq:secondder}
	-\Delta + \tfrac{d(d-2)}{4} - (q-1) \tfrac{\mathcal E_1[U_*]}{\|U_*\|_q^q} |U_*|^{q-2} + (q-2)\, \tfrac{\mathcal E_1[U_*]}{\|U_*\|_q^{2q}} \left| |U_*|^{q-2} U_* \right\rangle \left\langle |U_*|^{q-2} U_* \right|
\end{equation}
in $L^2(\Sph^d)$ is positive semidefinite. (Here $|f\rangle\langle f|$ denotes the rank one operator $\phi\mapsto \langle f,\phi\rangle f$.) This operator is considered as an unbounded, selfadjoint operator in $L^2(\Sph^d)$. It is bounded from below and has form domain $H^1(\Sph^d)$.

The computation for general $s$ is similar. We recall that the operator $A_s$ was introduced in \eqref{eq:defas}. We see that in terms of this operator the first condition in \eqref{eq:locminsphere} is equivalent to the equation
\begin{equation}
	\label{eq:firstders}
	A_s U_* - \tfrac{\mathcal E_s[U_*]}{\|U_*\|_q^q}\, |U_*|^{q-2} U_* = 0
	\qquad\text{on}\ \Sph^d
\end{equation}
and, for $U_*$ satisfying \eqref{eq:firstders}, the second condition in \eqref{eq:locminsphere} is equivalent to the positive semidefiniteness of the operator
\begin{equation}
	\label{eq:secondders}
	A_s - (q-1)\, \tfrac{\mathcal E_s[U_*]}{\|U_*\|_q^q}\, |U_*|^{q-2} + (q-2)\, \tfrac{\mathcal E_s[U_*]}{\|U_*\|_q^{2q}} \left| |U_*|^{q-2} U_* \right\rangle \left\langle |U_*|^{q-2} U_* \right|.
\end{equation}


\subsection*{Local minimality of constants}

We turn to the proof of the first part of Theorem~\ref{charlocminsphere}. Let $\Psi$ be a conformal transformation of $\Sph^d$, $c\in\R\setminus\{0\}$ and $U_* = c J_\Psi^{1/q}$. We wish to show that $U_*$ is a local minimizer of \eqref{eq:l2quotsphere}. By homogeneity of the problem it suffices to consider $c=1$ and by conformal invariance it suffices to consider $\Psi={\rm id}_{\Sph^d}$.

\medskip

We begin with the case $s=1$, where we need to show that $U_*=1$ satisfies equation \eqref{eq:firstder} and that the operator in \eqref{eq:secondder} is positive semidefinite. Verification of \eqref{eq:firstder} is straightforward. The operator in \eqref{eq:secondder} becomes
$$
\mathcal L_1 := - \Delta -d + d\, |\Sph^d|^{-1} \left| 1 \rangle \langle 1 \right| .
$$
We recall (see, e.g., \cite[Theorem 3.49]{FrLaWe}) that the spectrum of the Laplace--Beltrami operator $-\Delta$ in $L^2(\Sph^d)$ consists of the discrete eigenvalues $\ell(\ell+d-1)$, $\ell\in\N_0$, (of certain known multiplicities which, however, are irrelevant for us at this point). The lowest eigenvalue is $0$ and the corresponding eigenfunctions are precisely the constants. The operator $|\Sph^d|^{-1} \left| 1 \rangle \langle 1 \right|$ is the projection $P_0$ onto constants in $L^2(\Sph^d)$. Since this operator commutes with $-\Delta$, we can use the above facts to describe the spectrum of $\mathcal L_1$. It consists precisely of the eigenvalues $\ell(\ell+d-1) - d$, $\ell\in\N$. In particular, its spectrum is contained in $\overline{\R_+}$, and therefore the operator $\mathcal L_1$ is positive semidefinite, as we wanted to show.

For later purposes we recall that the eigenvalue $d$ of $-\Delta$ in $L^2(\Sph^d)$ has multiplicity $d+1$ and a basis of corresponding eigenfunctions is given by the coordinate functions $\omega_j$, $j=1,\ldots,d+1$. Therefore, the eigenvalue $0$ of $\mathcal L_1$ has multiplicity $d+2$ and a basis of eigenfunctions is given by constants and the coordinate functions. These $d+2$ zero modes of $\mathcal L_1$ reflect the invariances of the variational problem: the coordinate functions come from the translation and dilation invariance, while the constant function comes from the homogeneity of the problem. In this sense the constant function $1$ is a \emph{nondegenerate} local minimizer: the only zero modes come from the invariances.

\medskip

The argument for general $s$ is similar. Equation \eqref{eq:firstders} for $U_*=1$ follows immediately from
$$
\tfrac{\mathcal E_s[1]}{\|1\|_q^q} = \tfrac{\Gamma(\frac d2+s)}{\Gamma(\frac d2-s)} \,,
$$
which also shows that the operator in \eqref{eq:secondders} becomes
\begin{align}
	\label{eq:defls}
	\mathcal L_s & := A_s - (q-1)\, \tfrac{\Gamma(\frac d2+s)}{\Gamma(\frac d2-s)} + (q-2)\, \tfrac{\Gamma(\frac d2+s)}{\Gamma(\frac d2-s)}\, |\Sph^d|^{-1} \left| 1 \rangle \langle 1\right| \notag \\
	& = \sum_{\ell=2}^\infty \left( \tfrac{\Gamma(\ell + \frac{d}{2} + s)}{\Gamma(\ell + \frac{d}{2} - s)} - \tfrac{\Gamma(1 + \frac{d}{2} + s)}{\Gamma(1 + \frac{d}{2} - s)} \right) P_\ell \,.
\end{align}
The second equality here uses the functional equation of the gamma function. We emphasize that the terms with $\ell=0$ and $\ell=1$ vanish. Thus, spherical harmonics of degrees $0$ and $1$ are in the kernel of $\mathcal L_s$. Moreover, by the log-convexity of the gamma function, $\frac{d}{dt}\ln\frac{\Gamma(t+s)}{\Gamma(t-s)} = (\ln\Gamma)'(t+s) - (\ln\Gamma)'(t-s)> 0$ for all $t>s>0$, so
\begin{equation}\label{eq:gammamono}
	\ell \mapsto \tfrac{\Gamma(\ell+\frac d2 +s)}{\Gamma(\ell+\frac d2 -s)}
	\qquad\text{is increasing}. 
\end{equation}
It follows that $\mathcal L_s$ is positive definite on the orthogonal complement of the range of $P_0+P_1$. Thus, we have shown that $\mathcal L_s$ is positive semidefinite, as we wanted to show.

This completes the proof of the first part of Theorem \ref{charlocminsphere}.


\subsection*{Classification of local minimizers}

It remains to prove the second part of Theorem~\ref{charlocminsphere}. As a preparation for the proof we first establish the following lemma, which will allow us to fix the center of mass by a conformal transformation.

\begin{lemma}\label{brouwer}
	Let $f\in L^1(\Sph^d)$ with $\int_{\Sph^d} f(\omega)\,d\omega\neq 0$. Then there is a conformal transformation $\Psi$ of $\Sph^d$ such that
	$$
	\int_{\Sph^d} \Psi^{-1}(\omega) f(\omega)\,d\omega = 0 \,.
	$$
\end{lemma}

\begin{proof}
	\emph{Step 1.} In this preliminary step we define a family of conformal transformations $\gamma_{\delta,\xi}$ of $\Sph^d$ depending on two parameters $\delta>0$ and $\xi\in\Sph^d$. To do so, we denote dilations on $\R^d$ by $\mathcal D_\delta$, that is, $\mathcal D_\delta(x)=\delta x$. Moreover, for any $\xi\in\Sph^d$ we choose an orthogonal $(d+1)\times(d+1)$ matrix $O$ such that $O\xi=(0,\ldots,0,1)^{\rm T}$ and we put
	$$
	\gamma_{\delta,\xi}(\omega) := 
	\begin{cases}
	O^{\rm T} \mathcal S\left(\mathcal D_\delta\left(\mathcal S^{-1}\left(O\omega\right)\right)\right) & \text{if}\ \omega\neq -\xi \,,\\
	-\xi & \text{if}\ \omega = -\xi \,.
	\end{cases}
	$$
	This transformation depends only on $\xi$ and $\delta$ and not on the particular choice of $O$. Indeed, a straightforward computation shows that
	\begin{equation*}
		\gamma_{\delta,\xi}(\omega) = \frac{2\delta}{(1+\omega\cdot\xi)+\delta^2 (1-\omega\cdot\xi)} \ \left(\omega- (\omega\cdot\xi) \ \xi \right) 
		+ \frac{(1+\omega\cdot\xi)-\delta^2 (1-\omega\cdot\xi)}{(1+\omega\cdot\xi)+\delta^2 (1-\omega\cdot\xi)} \ \xi \,.
	\end{equation*}
	Since $\gamma_{\delta,\xi}$ is a composition of conformal transformations, it is conformal.
	
	\medskip

	\emph{Step 2.} We now turn to the main part of the proof, where we may assume that $f\in L^1(\Sph^d)$ is normalized by $\int_{\Sph^d} f(\omega) \,d\omega=1$. We will show that the $\R^{d+1}$-valued map
	$$
	F(r\xi) := \int_{\Sph^d} \gamma_{1-r,\xi}(\omega) f(\omega) \,d\omega\,,
	\qquad 0< r< 1\,,\ \xi\in\Sph^d \,, 
	$$
	has a zero. Once we have shown this, we deduce the assertion of the lemma by taking $\Psi = \gamma_{1-r_0,\xi_0}^{-1}$, where $r_0\xi_0$ is the zero of $F$.
	
	First, note that because of $\gamma_{1,\xi}(\omega)= \omega$ for all $\xi$ and all $\omega$, the limit of $F(r\xi)$ as $r\to 0$ is independent of $\xi$, so $F$ extends to a continuous function on the open unit ball of $\R^{d+1}$. In order to understand its boundary behavior, one easily checks that for any $\omega\neq-\xi$ one has $\lim_{\delta\to 0} \gamma_{\delta,\xi}(\omega)= \xi$, and that  this convergence is uniform on $\{(\omega,\xi) \in \Sph^d\times\Sph^d :\ 1+\omega\cdot\xi \geq \epsilon\}$ for any $\epsilon>0$. This implies that
	\begin{equation*}
		\lim_{r\to 1} F(r\xi) = \xi
		\qquad\text{uniformly in}\ \xi\,.
	\end{equation*}
	Hence, $F$ is a continuous function on the \emph{closed} unit ball and is the identity on the boundary. The assertion is now a consequence of one of the equivalent forms of Brouwer's fixed point theorem; see, e.g., \cite[Appendix]{Pe}.
\end{proof}

We now prove the second part of Theorem \ref{charlocminsphere}. Again we begin with the case $s=1$. Let $0\neq U_*\in H^1(\Sph^d)$ be a local minimizer of \eqref{eq:l2quotsphere}. We wish to show that there is a conformal transformation $\Psi$ of $\Sph^d$ and a constant $c\in\R\setminus\{0\}$ such that $U_* = c J_\Psi^{1/q}$.

According to Lemma \ref{brouwer} we can choose the conformal transformation $\Psi$ in such a way that
$$
0 = \int_{\Sph^d} \Psi^{-1}(\omega) |U_*(\omega)|^q\,d\omega = \int_{\Sph^d} \omega |(U_*)_\Psi(\omega)|^q\,d\omega \,.
$$
Note that by conformal invariance $(U_*)_\Psi$ is also a local minimizer. Thus, by replacing $U_*$ by $(U_*)_\Psi$, it suffices to show that if $U_*$ is a local minimizer satisfying
\begin{equation}
	\label{eq:com}
	\int_{\Sph^d} \omega |U_*|^q\,d\omega = 0 \,,
\end{equation}
then $U_*$ is a constant. To prove this, we make use of the positive semidefiniteness of the linear operator \eqref{eq:secondder}. Evaluating the operator on the function $\omega_j U_*$, $j=1,\ldots,d+1$, (that is, choosing $\phi=\omega_j U_*$ in the second condition in \eqref{eq:locminsphere}) we obtain
\begin{equation}
	\label{eq:proofunique1}
	\int_{\Sph^d} \left( |\nabla(\omega_j U_*)|^2 + \tfrac{d(d-2)}{4}\omega_j^2 U_*^2 - (q-1) \tfrac{\mathcal E_1[U_*]}{\|U_*\|_q^q} \omega_j^2 |U_*|^q \right)d\omega \geq 0 \,.
\end{equation}
Here we used \eqref{eq:com} to see that the rank-one term in the operator \eqref{eq:secondder} vanishes on the chosen function. We have
\begin{align*}
	|\nabla(\omega_j U_*)|^2 & = \omega_j^2 |\nabla U_*|^2 + 2 \omega_j U_* \nabla\omega_j\cdot\nabla U_* + U_*^2 |\nabla\omega_j|^2 \\
	& = \omega_j^2 |\nabla U_*|^2 + U_* \nabla(\omega_j^2) \cdot\nabla U_* + U_*^2 \left( 1- \omega_j^2\right).
\end{align*}
(Recall that $\nabla$ denotes the Riemannian gradient and that $\nabla\omega_j = e_j - \omega_j \omega$.) Inserting this into \eqref{eq:proofunique1}, we find
\begin{align*}
	& \int_{\Sph^d} \!\! \left( \omega_j^2 |\nabla U_*|^2 + U_* \nabla(\omega_j^2) \!\cdot\! \nabla U_* + U_*^2 \left( 1- \omega_j^2\right) + \tfrac{d(d-2)}{4}\omega_j^2 U_*^2 - (q-1) \tfrac{\mathcal E_1[U_*]}{\|U_*\|_q^q} \omega_j^2 |U_*|^q \right)d\omega \notag \\
	& \quad \geq 0 \,.
\end{align*}
Summing these inequalities with respect to $j=1,\ldots,d+1$ and using $\sum_j \omega_j^2 = 1$, we arrive at
\begin{align*}
	0 & \leq \int_{\Sph^d} \left( |\nabla U_*|^2 + d U_*^2 + \tfrac{d(d-2)}{4} U_*^2 - (q-1) \tfrac{\mathcal E_1[U_*]}{\|U_*\|_q^q} |U_*|^q \right)d\omega = -(q-2) \int_{\Sph^d} |\nabla U_*|^2\,d\omega \,.
\end{align*}
Note that the coefficients of $U_*^2$ cancel. Since $q>2$, we have shown $\int_{\Sph^d} |\nabla U_*|^2\,d\omega \leq 0$, which implies that $U_*$ is a constant, as claimed. This completes the proof of Theorem~\ref{charlocminsphere} for $s=1$.

\medskip

We now discuss the case of general $s$. As before we can use Lemma \ref{brouwer} to reduce the proof to showing that, if $U_*$ is a local minimizer satisfying \eqref{eq:com}, then $U_*$ is constant. Evaluating the operator \eqref{eq:secondders} on the function $\omega_j U_*$, $j=1,\ldots,d+1$, and recalling \eqref{eq:com}, we obtain
$$
\mathcal E_s[\omega_j U_*] - (q-1) \tfrac{\mathcal E_s[U_*]}{\|U_*\|_q^q} \int_{\Sph^d} \omega_j^2 |U_*|^q \,d\omega \geq 0 \,.
$$
Summing with respect to $j$ gives
\begin{equation}
	\label{eq:secondvarineq}
	\sum_{j=1}^{d+1} \mathcal E_s[\omega_j U_*] - (q-1) \mathcal E_s[U_*] \geq 0 \,.
\end{equation}
To simplify the first term on the left side, we need an auxiliary result about spherical harmonics.

\begin{lemma}\label{multbycoord}
	For all $\ell\geq 0$,
	$$
	\sum_{j=1}^{d+1} \omega_j P_\ell\, \omega_j = \tfrac{\ell+1}{2\ell+d+1}\, P_{\ell+1} + \tfrac{\ell+d-2}{2\ell+d-3}\,P_{\ell-1} \,,
	$$
	with the conventions that, if $\ell=0$, $\tfrac{\ell+d-2}{2\ell+d-3}\,P_{\ell-1}=0$ and that, if $d=1$ and $\ell=1$, $\frac{\ell+d-2}{2\ell+d-3}=\frac12$  
\end{lemma}

\begin{proof}
	We prove the equality of integral kernels
	$$
	(\omega\cdot\omega')\, P_\ell(\omega,\omega') = \tfrac{\ell+1}{2\ell+d+1}\, P_{\ell+1}(\omega,\omega') + \tfrac{\ell+d-2}{2\ell+d-3}\,P_{\ell-1}(\omega,\omega') \,.
	$$
	We have, for all $\ell$,
	$$
	P_\ell(\omega,\omega') = \frac{\nu_\ell}{|\Sph^d|\,C_\ell^{((d-1)/2)}(1)}\, C_\ell^{((d-1)/2)}(\omega\cdot\omega') \,,
	$$
	where $\nu_\ell$ is the dimension of the space of spherical harmonics of degree $\ell$ and where $C_\ell^{((d-1)/2)}$ is a Gegenbauer polynomial. For this formula, without the explicit value of the constant, see, e.g. \cite[Theorem IV.2.14]{StWe}. The value of the constant is determined by the relation $\Tr P_\ell = \nu_\ell$.
	
	The claimed formula now follows from the recursion relation for Gegenbauer polynomials \cite[(22.7.3)]{AbSt},
	$$
	2(\ell+\alpha)\, t C_\ell^{(\alpha)}(t) = (\ell+1) C_{\ell+1}^{(\alpha)}(t) + (\ell+2\alpha-1) C_{\ell-1}^{(\alpha)}(t) \,
	$$
	together with the normalization \cite[(22.2.3)]{AbSt}
	$$
	C_\ell^{(\alpha)}(1) = \binom{\ell+2\alpha-1}{\ell} \ \ \text{if}\ \alpha>0 \,,
	\qquad
	C_\ell^{(0)}(1) = \begin{cases} 1 & \text{if}\ \ell=0 \,,\\ \frac2\ell & \text{if}\ \ell\geq 1 \,,
		\end{cases}
	$$
	and the multiplicity formula \cite[Section IV.2]{StWe}
	$$
	\nu_\ell = \frac{(d-2+\ell)!\, (d+2\ell-1)}{\ell!\, (d-1)!}
	$$
	(with the convention that $\nu_0=1$ if $d=1$).
\end{proof}

It follows from Lemma \ref{multbycoord} that
\begin{align*}
	\sum_{j=1}^{d+1} \mathcal E_s[\omega_j U] & = \sum_{\ell=0}^\infty \tfrac{\Gamma(\ell + \frac{d}{2} + s)}{\Gamma(\ell + \frac{d}{2} - s)} \sum_{j=1}^{d+1} \|P_\ell\omega_j U\|_2^2 \\
	& = \sum_{\ell=0}^\infty \tfrac{\Gamma(\ell + \frac{d}{2} + s)}{\Gamma(\ell + \frac{d}{2} - s)}
	\left( \tfrac{\ell+1}{2\ell+d+1} \left\| P_{\ell+1} U \right\|_2^2 + \tfrac{\ell+d-2}{2\ell+d-3}\left\| P_{\ell-1} U \right\|_2^2 \right) \\
	& = \sum_{\ell=0}^\infty \left( \tfrac{\Gamma(\ell-1 + \frac{d}{2} + s)}{\Gamma(\ell-1 + \frac{d}{2} - s)} \tfrac{\ell}{2\ell+d-1} + \tfrac{\Gamma(\ell+1 + \frac{d}{2} + s)}{\Gamma(\ell+1 + \frac{d}{2} - s)} \tfrac{\ell+d-1}{2\ell+d-1} \right) \left\| P_{\ell} U \right\|_2^2 \\
	& =  \sum_{\ell=0}^\infty \tfrac{\Gamma(\ell + \frac{d}{2} + s)}{\Gamma(\ell + \frac{d}{2} - s)} \left( \tfrac{\ell-1+\frac d2-s}{\ell-1+\frac d2 + s}\, \tfrac{\ell}{2\ell+d-1} + \tfrac{\ell+\frac d2 +s}{\ell+\frac d2-s} \, \tfrac{\ell+d-1}{2\ell+d-1} \right) \left\| P_{\ell} U \right\|_2^2	\,.
\end{align*}
Thus,
\begin{equation}
	\label{eq:secondvarsimpli}
	\sum_{j=1}^{d+1} \mathcal E_s[\omega_j U] - (q-1)\, \mathcal E_s[U]
	= - \sum_{\ell=0}^\infty \tfrac{\Gamma(\ell + \frac{d}{2} + s)}{\Gamma(\ell + \frac{d}{2} - s)}\, w_s(\ell) \left\| P_{\ell} U \right\|_2^2
\end{equation}
with
$$
w_s(\ell) := \tfrac{d+2s}{d-2s} - \tfrac{\ell-1+\frac d2-s}{\ell-1+\frac d2 + s}\, \tfrac{\ell}{2\ell+d-1} - \tfrac{\ell+\frac d2 +s}{\ell+\frac d2-s} \, \tfrac{\ell+d-1}{2\ell+d-1} \,.
$$
A tedious but straightforward computation shows that
\begin{align*}
	w_s(\ell) = \frac{4s}{d-2s} \, \frac{\ell(\ell+d-1)}{(\ell-1+\frac d2+s)(\ell+\frac d2-s)} \,,
\end{align*}
so
\begin{equation}
	\label{eq:weight}
	w_s(\ell)\geq 0 
	\qquad\text{with equality if and only if}\ \ell=0 \,.
\end{equation}
Taking $U=U_*$ in \eqref{eq:secondvarsimpli}, recalling the second variation inequality \eqref{eq:secondvarineq} and using \eqref{eq:weight}, we deduce that $P_\ell U_*=0$ for all $\ell\geq 1$, that is, $U_*$ is a constant, as we wanted to prove. This completes the proof of Theorem \ref{charlocminsphere}.


\subsection*{Appendix: Subcritical interpolation inequalities}

While the main focus of these lectures is on the Sobolev inequality with critical exponent, in this appendix we make a brief digression to the subcritical case and show the following result.

\begin{theorem}\label{bvvb}
	Let $2\leq q<\infty$ if $d=1,2$ and $2\leq q<\frac{2d}{d-2}$ if $d\geq 3$. Then for all $U\in H^1(\Sph^d)$,
	\begin{equation}
		\label{eq:bvvb}
		\int_{\Sph^d} \left( |\nabla U|^2 + \frac{d}{q-2}\, U^2 \right)d\omega \geq \frac{d}{q-2}\, |\Sph^d|^{1-2/q} \left( \int_{\Sph^d} |U|^q\,d\omega\right)^{2/q}.
	\end{equation}
	with equality if and only if $U$ is constant.
\end{theorem}

Inequality \eqref{eq:bvvb} for $q=\frac{2d}{d-2}$ turns into the inequality $\mathcal E_1[U]\geq S_{d,1}\|U\|_q^2$, which we have shown in the main part of this lecture. We emphasize, however, that in this critical case the set of functions attaining equality is strictly larger than in the subcritical case.

\begin{proof}
	Given $q$ we define $s:=d(\frac12 -\frac1q)$, so that $q=\frac{2d}{d-2s}$. Then, by the main theorem of this lecture, we have $S_{d,s} \| U\|_q^2 \leq \mathcal E_s[U]$, and our goal is to bound $\mathcal E_s[U]$ from above by a constant times the left side in \eqref{eq:bvvb}. Expanding $U$ into spherical harmonics, the task becomes to find the smallest constant $C$ in the inequality
	$$
	\frac{\Gamma(\ell+\frac d2 + s)}{\Gamma(\ell+\frac d2 -s)} 
	= \frac{\Gamma(\ell+\frac d2 + d(\frac12-\frac1q))}{\Gamma(\ell+\frac d2 -d(\frac12-\frac1q))} 
	\leq C \left( \ell(\ell+d-1) + \frac{d}{q-2} \right)
	\qquad\text{for all}\ \ell\in\N_0 \,.
	$$
	Using properties of gamma functions, one can show that the optimal constant is attained exactly at $\ell=0$, which implies \eqref{eq:bvvb}.
\end{proof}


\subsection*{Appendix: Optimizing sequences}

In this appendix we want to show that the technique of fixing the center of mass in Lemma \ref{brouwer} can also be useful to prove the relative compactness up to symmetries of optimizing sequences. We give the argument for $s=1$.

The crucial ingredient is the following refined inequality of Aubin. For $d\geq 3$ and $\epsilon>0$ there is a $C_\epsilon<\infty$ such that for all $U\in H^1(\Sph^d)$ with
$$
\int_{\Sph^d} \omega\, |U|^q\,d\omega = 0
$$
one has
\begin{equation}
	\label{eq:aubincom}
	\mathcal E_1[U] \geq (1-\epsilon) 2^{2/d} S_{d,1} \|U\|_q^2 - C_\epsilon \|U\|_2^2 \,.
\end{equation}
This inequality implies that for suitably normalized functions one obtains a Sobolev constant that is $>S_{d,1}$, at the expense of an $L^2$-term, which in many applications is harmless.

\medskip

Let us use \eqref{eq:aubincom} to prove relative compactness of optimizing sequences. Let $(U_n)\subset H^1(\Sph^d)$ with $\mathcal E_1[U_n]=1$ and $\| U_n\|_q^2 \to S_{d,1}^{-1}$. By Lemma \ref{brouwer} there is a conformal transformation $\Psi_n$ of $\Sph^d$ such that $\widetilde U_n := (U_n)_{\Psi_n}$ satisfies $\int_{\Sph^d} \omega |\widetilde U_n|^q\,d\omega = 0$. Moreover, by conformal invariance, $\mathcal E_1[\widetilde U_n]=1$ and $\| \widetilde U_n\|_q^2 \to S_{d,1}^{-1}$. After passing to a subsequence, we may assume that $\widetilde U_n\rightharpoonup \widetilde U$ in $H^1(\Sph^d)$. By Rellich's lemma, $\widetilde U_n\to \widetilde U$ in $L^2(\Sph^d)$. Therefore, applying \eqref{eq:aubincom} to $\widetilde U_n$ and passing to the limit, we obtain
$$
1 \geq (1-\epsilon) 2^{2/d} - C_\epsilon \|\widetilde U\|_2^2 \,.
$$
This is valid for any fixed $\epsilon>0$. Choosing it so that $1<(1-\epsilon) 2^{2/d}$, we deduce that $\widetilde U\neq 0$. Thus, we have shown that there is something somewhere. The proof that there is nothing else anywhere else is as in the first lecture.

\medskip

For the sake of completeness we present Aubin's proof of \eqref{eq:aubincom}.

\begin{proof}[Proof of \eqref{eq:aubincom}]
	Let $\delta>0$ and set $h(t):= (t^2 +\delta^2)^{-(q-1)/(2q)}t$ for $t\in[-1,1]$. For $\Omega,\omega\in\Sph^d$ we set $h_\Omega(\omega):=h(\Omega\cdot\omega)$ and note that
	$$
	Z:= \int_{\Sph^d} h_\Omega(\omega)^2\,d\Omega =  |\Sph^{d-1}| \int_0^\pi h(\cos\theta)^2 \sin^{d-1}\theta\,d\theta
	$$
	is independent of $\omega$.
	It follows that
	\begin{align*}
		\| U \|_q^2 & = \|U^2\|_{q/2} = \left\| Z^{-1} \int_{\Sph^d} h_\Omega^2 U^2\,d\Omega \right\|_{q/2} \leq Z^{-1} \int_{\Sph^d} \| h_\Omega^2 U^2 \|_{q/2}\,d\Omega \\
		& = Z^{-1} \int_{\Sph^d} \|h_\Omega U \|_q^2\,d\Omega \,.
	\end{align*}
	We will show that there is a constant $C$, depending only on $q$, such that for each $\Omega\in\Sph^d$,
	\begin{align}
		\label{eq:aubincomproof}
		2^{-2/q} S_{d,1} \|h_\Omega U \|_q^2 & \leq 2^{-1} \int_{\Sph^d} h_\Omega^2 \left( |\nabla U|^2 + \tfrac{d(d-2)}{4} |U|^2\right) d\omega + 2^{-1} \int_{\Sph^d} |\nabla h_\Omega|^2 |U|^2\,d\omega \notag \\
		& \quad + 2^{-1} \int_{\Sph^d} (\nabla h_\Omega^2)\cdot U\nabla U\,d\omega 
		+ C^2 \delta^{2/q} \mathcal E_1[U] \,.
	\end{align}
	Integrating this bound with respect to $\Omega$, we obtain
	\begin{align*}
		2^{-2/q} S_{d,1} Z^{-1} \int_{\Sph^d} \|h_\Omega U \|_q^2\,d\Omega & \leq \left( 2^{-1} + C^2 Z^{-1} |\Sph^d|\delta^{2/q} \right) \mathcal E_1[U] + 2^{-1} \widetilde C \|U\|_2^2
	\end{align*}
	with $\widetilde C := Z^{-1} \int_{\Sph^d} |\nabla h_\Omega(\omega)|^2\,d\Omega$ (which is independent of $\omega$). Noting that $Z$ converges as $\delta\to 0$, we obtain the claimed bound in \eqref{eq:aubincom}. (In contrast, we note that $\widetilde C$ diverges as $\delta\to 0$, since $|t|^{-(q-1)/q}t$ is not $H^1$ near $t=0$.)
	
	It remains to prove \eqref{eq:aubincomproof}. We may assume that $\mathcal E_1[(h_\Omega)_+\nabla U]\geq \mathcal E_1[(h_\Omega)_-\nabla U]$, the opposite case being similar. We use $0 \leq t - h(t)^q\leq 2C^q\delta$ for all $t\in[0,1]$ and some $C$, depending only on $q$, to bound
	\begin{align*}
		\| h_\Omega U \|_q^q & \leq \int_{\Sph^d} (\Omega\cdot\omega)_+ |U|^q \,d\omega + \int_{\Sph^d} (h_\Omega)_-^q |U|^q\,d\omega \\
		& = \int_{\Sph^d} (\Omega\cdot\omega)_- |U|^q \,d\omega + \int_{\Sph^d} (h_\Omega)_-^q |U|^q\,d\omega \\
		& \leq 2 \int_{\Sph^d} \left( (h_\Omega)_-^q + C^q \delta \right) |U|^q\,d\omega \\
		& \leq 2 \int_{\Sph^d} \left( (h_\Omega)_-^2 + C^2\delta^{2/q} \right)^{q/2} |U|^q\,d\omega \,.
	\end{align*}
	The triangle inequality and Sobolev's inequality imply
	\begin{align*}
		2^{-2/q} S_{d,1} \|h_\Omega U\|_q^2 & \leq S_{d,1} \left( \| (h_\Omega)_-U\|_q^2 + C^2 \delta^{2/q} \| U \|_q^2 \right) \\
		& \leq \mathcal E_1[(h_\Omega)_- U] + C^2\delta^{2/q} \mathcal E_1[U] \\
		& \leq 2^{-1}\left( \mathcal E_1[(h_\Omega)_- U] + \mathcal E_1[(h_\Omega)_+ U] \right) + C^2\delta^{2/q} \mathcal E_1[U] \,.
	\end{align*}
	Using $(h_\Omega)_-^2 + (h_\Omega)_+^2= h_\Omega^2$ and
	$$
	|\nabla (h_\Omega)_- U|^2 + |\nabla (h_\Omega)_+ U|^2 = h_\Omega^2 |\nabla U|^2 + |\nabla h_\Omega|^2 |U|^2 + (\nabla h_\Omega^2)\cdot U\nabla U \,,
	$$
	we obtain the claimed bound \eqref{eq:aubincomproof}.	
\end{proof}


\subsection*{Bibliographic remarks}

There are a number of alternative proofs of Theorem~\ref{charmin}, in particular for $s=1$. Here we review some of them and give references.

The proof of Theorem \ref{charmin} presented in this lecture in the case $s=1$ is from \cite{FrLi0}; see also \cite{FrLi}. The proof for general $s$ is a new variant of the recent proof in \cite{Ya}. It simplifies the corresponding argument in \cite{FrLi0} (where duality was invoked and only positive functions were considered). The new ingredient in \cite{Ya}, compared to \cite{FrLi0}, is a commutator identity, which for integer $s$ is due to \cite{Cas}. Our Lemma \ref{multbycoord} serves a similar purpose.

The presented proof may not be the most direct proof of Theorem \ref{charmin}, but it has the advantage of yielding Theorem \ref{charlocmin}, which may be new (at least for functions that are not necessarily nonnegative). Also, this proof is natural in this lecture series in view of the second variation analysis in the next lecture. Further, the approach presented in this lecture works in the setting of the Heisenberg group, where several other techniques mentioned below (for instance those based on symmetric decreasing rearrangement or the moving plane method) seem not to work. It has also been applied in the fully nonlinear setting \cite{Cas,CaWa}.

Lemma \ref{brouwer} is due to Hersch \cite{He}, where it was used in the problem of maximizing the first nontrivial eigenvalue of the Laplace--Beltrami operator on $\Sph^2$ over all metrics with fixed area and conformal to the standard metric.

As far as we know, the optimal value of the constant $S_{d,1}$ with $s=1$ as well as the optimizing functions appeared for the first time in the unpublished preprint of Rodemich \cite{Rod} and in the papers of Aubin \cite{Au} and Talenti \cite{Ta}. These works deal with the more general situation of an $L^p$-norm of the gradient with $1\leq p<d$ (and the correspondingly modified $q$). They use rearrangement techniques to reduce the problem to a one-dimensional problem that had been solved by Bliss \cite{Bl}. Since the relevant rearrangement inequality for the gradient can be an equality without the functions being radial, it does not seem possible to derive the characterization of optimizers using these techniques.

We also mention \cite{Ro} where local minimality of $Q$ is shown for $s=1$ and $d=3$.

Theorem \ref{charmin} for general $s$ is due to Lieb \cite{Li}, who found the optimal value of $S_{d,s}$ and characterized all optimizers. He carried this out in the dual formulation of the Hardy--Littlewood--Sobolev inequality. In this connection he observed and utilized the conformal invariance for general $s$. He also used a strict rearrangement inequality from \cite{Li0}. The interplay between rearrangement and conformal invariance (`competing symmetries') is also crucial for the alternative proof of Carlen and Loss \cite{CaLo}. The role of conformal invariance was emphasized in \cite{Be}. There Lemma \ref{evsphere} appeared, albeit without proof. For GJMS operators on more general manifolds than $\Sph^d$ we refer to \cite{GrJeMaSp,GrZw}.

The minimization problem $S_{d,1/2}$ with $s=1/2$ is equivalent to finding the best constant in a Sobolev trace inequality on $\R^d\times\R_+$; see, e.g., \cite{Be}. The latter problem was solved by Escobar \cite{Es} by an adaptation of Obata's method mentioned below. For an alternative proof see \cite{CaLo1}.

Even before \cite{Au,Ta}, Obata \cite{Ob} has characterized all `sufficiently nice' solutions of the Euler--Lagrange equation corresponding to the minimization problem for $S_{d,1}$ on $\Sph^d$. Up to proving the existence of an optimizer and showing that it is `sufficiently nice', this leads to the sharp value of the constant $S_{d,1}$ and the characterization of its optimizers. The fact that optimizers, which are weak solutions of the equation, are `sufficiently nice' is due to Trudinger \cite{Tr}. The method of Obata was extended in \cite{GiSp,BVVe}. A related method appears in \cite[Theorem 6.10]{Ba} in the setting of diffusion semigroups satisfying a curvature-dimension condition.

Another result concerning the Euler--Lagrange equation corresponding to the minimization problem for $S_{d,1}$ on $\R^d$ was obtained in \cite{GiNiNi}. There, using the method of moving planes, it was shown that any positive, classical and sufficiently fast decaying solution is necessarily radial about some point and decreasing with respect to the distance from that point. This reduces the classification of all solutions to a simple ODE analysis. We remark that the relevant ODE becomes autonomous in logarithmic coordinates. It was observed in \cite{CaGiSp} that the decay assumption in \cite{GiNiNi} can be removed by employing the invariance under inversions on the unit sphere.

The method of moving planes (and its relative, the method of moving spheres) has been adapted in \cite{ChLiOu,LiY} to the Euler--Lagrange equation for the optimization problem $S_{d,s}$ in the dual formulation for general $s$. Combined with the conformal invariance this leads to a classification of all positive solutions and, consequently, of all minimizers.

A related reflection/inversion technique was used in \cite{FrLi00,FrLi01} to give another proof of the characterization of optimizers under the additional assumption $s\leq 1$. 

A proof of Theorem \ref{charmin} via optimal transport theory appears in \cite{CENaVi} (for $s=1$) and in \cite{Na} (for $s=1/2$).

Theorem \ref{charmin} with $s=1$ can also be proved using nonlinear (porous medium or fast diffusion) flows; see \cite{De} and, in the dual setting of an HLS inequality, \cite{CaCaLo}.

The subcritical Sobolev inequality in Theorem \ref{bvvb} is classical for $d=1$. For general $d$ it appeared around the same time in works of Bidaut-V\'eron--V\'eron \cite[Appendix B]{BVVe}, Bakry \cite[Theorem 6.10]{Ba} and Beckner \cite[Theorem 4]{Be}. Our presentation follows the latter paper. The method of \cite{Ba,BVVe} is related to that of Obata \cite{Ob} (see also \cite[Appendix B]{GiSp}) and also classifies solutions of the corresponding Euler--Lagrange equation. An alternative proof of the subcritical Sobolev inequality on the sphere is to first use symmetric decreasing rearrangement on the sphere and then to use the one-dimensional result in \cite[pp. 204--205]{BaEm}. Yet another proof is by nonlinear flows \cite{De}; for more on linear and nonlinear flows see \cite{DoEsKoLo,DoEsLo1,DoEsLo2}. These works, in particular, bring into evidence a relation between the `elliptic' proofs of \cite{Ob,GiSp,BVVe,Ba} and the `parabolic' proofs in \cite{BaEm,De}. For a remarkable recent results obtained by elliptic methods, see \cite{DoEsLo3}.

Aubin's inequality \eqref{eq:aubincom} is from \cite{Au2}. Inequalities of this type have recently attracted some attention; see, e.g., \cite{ChHa,HaWa,ChWeWu}. The idea of using Aubin's inequality to deduce that a weak limit is nonzero is implicit in \cite[Lemme 3]{Au2} and \cite[Lemma 5.7]{ChYa} in the context of prescribing scalar curvature.

\section*{Lecture 3: Stability}

Our goal in this lecture is to prove a stability result for the sharp Sobolev inequality
$$
\|(-\Delta)^{s/2} u\|_2^2 \geq S_{d,s} \| u \|_q^2
\qquad\text{for all}\ u\in\dot H^s(\R^d) \,.
$$
That is, we want to prove that if $\|(-\Delta)^{s/2} u\|_2^2/\| u \|_q^2$ is close to $S_{d,s}$, then $u$ is close in $\dot H^s(\R^d)$ to an optimizer. We denote by
$$
\mathcal G := \left\{ g \in \dot H^s(\R^d) :\ \|(-\Delta)^{s/2} g\|_2^2 = S_{d,s} \| g \|_q^2 \right\}
$$
the set of all optimizers (and zero).

The compactness theorem from Lecture 1 already gives a qualitative version of this stability. Specifically, it implies that for any $\epsilon>0$ there is a $\delta>0$ such that, if $\|(-\Delta)^{s/2} u\|_2^2\leq (1+\delta) S_{d,s} \| u \|_q^2$, then $\inf_{g\in\mathcal G} \|(-\Delta)^{s/2}(u-g)\|_2 \leq \epsilon \|(-\Delta)^{s/2} u\|_2$.

In this lecture we are interested in a \emph{quantitative} stability result, which shows an explicit dependence of $\delta$ on $\epsilon$. That is, we want to bound the normalized Sobolev deficit $\|(-\Delta)^{s/2} u\|_2^2/\| u \|_q^2 - S_{d,s}$ from below by a power of the normalized distance $\inf_{g\in\mathcal G} \|(-\Delta)^{s/2}(u-g)\|_2 / \|(-\Delta)^{s/2} u\|_2$,
$$
\frac{\|(-\Delta)^{s/2} u\|_2^2}{\| u \|_q^2} - S_{d,s}
\gtrsim \left( \inf_{g\in\mathcal G} \frac{\|(-\Delta)^{s/2}(u-g)\|_2}{\|(-\Delta)^{s/2} u\|_2} \right)^\alpha.
$$
Since $\inf_{g\in\mathcal G} \|(-\Delta)^{s/2}(u-g)\|_2 / \|(-\Delta)^{s/2} u\|_2\leq 1$ (see (a) in Lemma \ref{dist} below), this inequality is stronger the smaller the power $\alpha$ is. Meanwhile, since we expect the left side to be sufficiently smooth and since minima of smooth functions are of quadratic or higher order, we do not expect a better power than $\alpha=2$.

The following theorem provides such a bound with the desired power $\alpha=2$. For $s=1$ it is due to Bianchi and Egnell. (Strictly speaking, the following theorem proves the above bound with the right side multiplied by a factor $\|u\|_q^2/\|(-\Delta)^{s/2}u\|_2^2\leq S_{d,s}^{-1}$. This difference, however, is immaterial as long as we do not care about constants: If $\|u\|_q^2/\|(-\Delta)^{s/2}u\|_2^2\geq (2S_{d,s})^{-1}$, say, then the prefactor is harmless, while if $\|u\|_q^2/\|(-\Delta)^{s/2}u\|_2^2< (2S_{d,s})^{-1}$, then the inequality is anyway trivially true in view of the bound $\inf_{g\in\mathcal G} \|(-\Delta)^{s/2}(u-g)\|_2 / \|(-\Delta)^{s/2} u\|_2\leq 1$.)

\begin{theorem}\label{be}
	Let $0<s<\frac d2$ and $q:=\frac{2d}{d-2s}$. Then, for all $u\in\dot H^s(\R^d)$,
	$$
	\|(-\Delta)^{s/2} u\|_2^2 - S_{d,s} \| u \|_q^2 \gtrsim \inf_{g\in\mathcal G} \|(-\Delta)^{s/2}(u-g)\|_2^2 \,.
	$$
\end{theorem}

The implicit constant in the inequality in the theorem depends on $d$ and $s$. The argument that we present is via compactness and does not yield an explicit constant. For recent progress on the problem of giving a constructive proof, see the remarks at the end of this lecture.

We will also prove the reverse inequality
\begin{equation}
	\label{eq:bereverse}
	\inf_{u\not\in\mathcal G} \frac{\|(-\Delta)^{s/2} u\|_2^2 - S_{d,s} \| u \|_q^2}{\inf_{g\in\mathcal G} \|(-\Delta)^{s/2}(u-g)\|_2^2} \leq \frac{4s}{d+2s+2} \,,
\end{equation}
which shows, in particular, that the power two of the distance to $\mathcal G$ in the theorem cannot be replaced by a smaller power.


\subsection*{The upper bound}

To get some intuition into the mechanism behind the proof of the theorem, we begin by proving \eqref{eq:bereverse}. It is natural to approach this problem by taking $u=u_*+\epsilon r$ with an optimizer $u_*$ and a function $r$, to be determined, and by expanding the relevant quotient in $\epsilon$. There will be a coefficient in front of the leading order in $\epsilon$ and this coefficient is a functional of $r$. The idea is to determine $r$ so as to minimize this functional. It will turn out that the problem for $r$ is a spectral problem that can be solved explicitly. This gives the constant $\frac{4s}{d+2s+2}$ in \eqref{eq:bereverse} as a certain spectral gap.

It is more convenient to carry out this idea in the equivalent setting of the inequality on the sphere, that is, to prove
\begin{equation}
	\label{eq:bereversesphere}
	\inf_{U\not\in\mathcal H} \frac{\mathcal E_s[U] - S_{d,s} \| U \|_q^2}{\inf_{h\in\mathcal H} \mathcal E_s[U-h]} \leq \frac{4s}{d+2s+2} \,,
\end{equation}
where
$$
\mathcal H := \left\{ h\in H^s(\Sph^d):\ \mathcal E_s[h] = S_{d,s} \|h \|_q^2 \right\}.
$$
The ansatz is then $U=U_*+\epsilon R$ and, by conformal invariance, we may assume $U_*=1$. Similarly to the derivative computations in the previous lecture, we find
\begin{equation}
	\label{eq:beuppernum}
	\lim_{\epsilon\to 0} \epsilon^{-2} \left( \mathcal E_s[1+\epsilon R] - S_{d,s} \| 1+\epsilon R \|_q^2 \right) = \langle R,\mathcal L_s R \rangle
\end{equation}
with the operator $\mathcal L_s$ introduced in \eqref{eq:defls}. This gives the behavior of the numerator on the left side of \eqref{eq:bereversesphere}. We would like to show that the denominator also behaves quadratically in $\epsilon$ and, more precisely, that under suitable assumptions on $R$ we have
$$
\inf_{h\in\mathcal H} \mathcal E_s[(1+\epsilon R)-h] = \epsilon^2 \mathcal E_s[R] \,.
$$
Note that here we always have $\leq$ since we can always choose $h=1\in\mathcal H$.

In the following lemma, we summarize some properties of the distance function
$$
\delta[U] := \inf_{h\in\mathcal H} \sqrt{\mathcal E_s[U-h]} \,.
$$

\begin{lemma}\label{dist}
	Let $0<s<\frac d2$.
	\begin{enumerate}
		\item[(a)] $\delta[U]^2\leq \mathcal E_s[U]$ with strict inequality if and only if $U\neq 0$.
		\item[(b)] For any $U\in \dot H^s(\Sph^d)$, there is an $h\in\mathcal H$ such that $\mathcal E_s[U-h] = \delta[U]^2$. If $U\neq 0$ and $\tau_U:=\delta[U]/\sqrt{\mathcal E_s[U] - \delta[U]^2}$, then
		$$
		\|U-h\|_q \leq \tau_U \|h\|_q \,.
		$$
		\item[(c)] If $\mathcal E_s[U-1] = \delta[U]^2$, then $R:=U-1$ satisfies
		\begin{equation}
			\label{eq:beortho}
			\int_{\Sph^d} R(\omega)\,d\omega = 0
			\qquad\text{and}\qquad
			\int_{\Sph^d} \omega\,R(\omega)\,d\omega = 0 \,.
		\end{equation}
		\item[(d)] There is an $\epsilon_0>0$ such that, if $\mathcal E_s[R]<\epsilon_0$ and \eqref{eq:beortho} holds, then $\delta[1+R]=\sqrt{\mathcal E_s[R]}$.		
	\end{enumerate}
\end{lemma}

Before proving this lemma, we use it to complete the proof of \eqref{eq:bereversesphere}.

It follows from Lemma \ref{dist} that, if $R$ satisfies \eqref{eq:beortho}, then for all sufficiently small $\epsilon$ we have $\inf_{h\in\mathcal H} \mathcal E_s[(1+\epsilon R)-h] = \epsilon^2 \mathcal E_s[R]$. Combining this with \eqref{eq:beuppernum} we find
$$
\lim_{\epsilon\to 0} \frac{\mathcal E_s[1+\epsilon R] - S_{d,s} \| 1+\epsilon R \|_q^2}{\inf_{h\in\mathcal H} \mathcal E_s[1+\epsilon R-h]}
= \frac{\langle R,\mathcal L_s R \rangle }{\mathcal E_s[R]} \,.
$$
At this point we can choose $R$ so as to minimize the right side. We will show below that
$$
\inf \left\{ \frac{\langle R,\mathcal L_s R \rangle }{\mathcal E_s[R]}:\ R \ \text{satisfies}\ \eqref{eq:beortho} \right\} = \frac{4s}{d+2s+2} \,.
$$
The argument given there also shows that the infimum is attained if and only if $R$ is a spherical harmonic of degree 2, so this is the optimal choice of $R$. This completes the proof of \eqref{eq:bereversesphere} and, therefore, of \eqref{eq:bereverse}.

\begin{proof}[Proof of Lemma \ref{dist}]
	We recall that the elements in $\mathcal H$ are of the form $c\,Q_\zeta$ with $c\in\R$ and $\zeta\in\R^{d+1}$ with $|\zeta|<1$. Here we set $Q_\zeta(\omega)= (\sqrt{1-|\zeta|^2}/(1-\zeta\cdot\omega))^{(d-2s)/2}$. For later purposes we record the normalizations 
	\begin{equation}
		\label{eq:qzetanorm}
		\|Q_\zeta\|_q^q = |\Sph^d| \,,
		\qquad
		\mathcal E_s[Q_\zeta] = \tfrac{\Gamma(\frac d2+s)}{\Gamma(\frac d2-s)}\, |\Sph^d|
	\end{equation}
	and the Euler--Lagrange equation 
	\begin{equation}
		\label{eq:qzetael}
		\mathcal E_s[V,Q_\zeta] = \tfrac{\Gamma(\frac d2+s)}{\Gamma(\frac d2-s)}\, \int_{\Sph^d} V Q_\zeta^{q-1} \,d\omega
	\qquad\text{for all}\ V\in H^s(\Sph^d) \,,
	\end{equation}
	Indeed, the first equality in \eqref{eq:qzetanorm} follows from the characterization of $Q_\zeta$ as a Jacobian in the previous lecture and the second one from the minimality, which also gives the Euler--Lagrange equation.
	
	\medskip
	
	(a) We write $\mathcal E_s[\cdot,\cdot]$ for the bilinear form associated to the quadratic form $\mathcal E_s[\cdot]$. Since
	$$
	\mathcal E_s[U-c\,Q_\zeta] = \mathcal E_s[U] - \frac{\mathcal E_s[Q_\zeta,U]^2}{\mathcal E_s[Q_\zeta]} + \mathcal E[Q_\zeta] \left( c - \frac{\mathcal E_s[Q_\zeta,U]}{\mathcal E_s[Q_\zeta]} \right)^2 \,,
	$$
	it follows that
	\begin{equation}
		\label{eq:distoptsphere}
		\delta[U]^2 = \inf_{c,\zeta} \mathcal E_s[U-c\,Q_\zeta] = \mathcal E_s[U] - \sup_\zeta \frac{\mathcal E_s[Q_\zeta,U]^2}{\mathcal E_s[Q_\zeta]}
	\end{equation}
	and that for each $\zeta$ the optimal $c$ is given by $c=\mathcal E_s[Q_\zeta,U]/\mathcal E_s[Q_\zeta]$. Item (a) follows immediately from \eqref{eq:distoptsphere} and the nondegeneracy of $\mathcal E_s$.
	
	\medskip
	
	(b) In view of \eqref{eq:qzetanorm} and \eqref{eq:qzetael} we can write \eqref{eq:distoptsphere} as
	\begin{equation}
		\label{eq:distoptsphere2}
		\delta[U]^2 = \mathcal E_s[U] - \tfrac{\Gamma(\frac d2+s)}{\Gamma(\frac d2-s)}\, |\Sph^d|^{-1} \sup_\zeta \left( \int_{\Sph^d} Q_\zeta^{q-1}U \,d\omega \right)^2.
	\end{equation}
	It is easy to see that $\zeta\mapsto \int_{\Sph^d} Q_\zeta^{q-1}U \,d\omega$ is continuous and tends to zero as $|\zeta|\to 1$, and therefore the supremum over $\zeta$ is attained, as claimed.
	
	Moreover, if $U\neq 0$ and if the infimum is attained at $h = c_0\,Q_{\zeta_0}$, then, recalling the expression for the optimal $c=c_0$, 
	$$
	\sup_\zeta \frac{\mathcal E_s[Q_\zeta,U]^2}{\mathcal E_s[Q_\zeta]}
	= \frac{\mathcal E_s[Q_{\zeta_0},U]^2}{\mathcal E_s[Q_{\zeta_0}]} = c_0^2\, \mathcal E_s[Q_{\zeta_0}] = \mathcal E_s[h] \,.
	$$
	Thus, by \eqref{eq:distoptsphere}, $\delta[U]^2 = \mathcal E_s[U] - \mathcal E_s[h]$ and so
	$$
	S_{d,s}\|U-h\|_q^2 \leq \mathcal E_s[U-h] = \delta[U]^2 =\tau_U^2 \left( \mathcal E_s[U] - \delta[U]^2 \right) = \tau_U^2 \mathcal E_s[h] = \tau_U^2 S_{d,s} \|h\|_q^2 \,,
	$$
	as claimed.
	
	\medskip
	
	(c) We assume now that the infimum is attained at $h=1$, that is, at $(c,\zeta)=(1,0)$. Then $1=c=\mathcal E_s[1,U]/\mathcal E_s[1] = |\Sph^d|^{-1}\int_{\Sph^d} U\,d\omega$, where we used \eqref{eq:qzetanorm} and \eqref{eq:qzetael}. This proves the first equality in \eqref{eq:beortho}. Moreover, by \eqref{eq:distoptsphere2},
	$$
	\nabla_\zeta \Big|_{\zeta=0} \int_{\Sph^d} Q_\zeta^{q-1} U\,d\omega =0 \,,
	$$
	which gives the second equality in \eqref{eq:beortho}.
	
	\medskip
	
	(d) We prove now conversely that for sufficiently small $R$, the validity of the orthogonality conditions implies that the distance is attained at the function $1$.
	
	To prove this, we apply the implicit function theorem and find $\epsilon_1,\epsilon_2>0$ such that for $R\in H^s(\Sph^d)$ with $\mathcal E_s[R]<\epsilon_1$ there is a unique $\zeta\in B_{\epsilon_2}(0)$ such that $\nabla_\zeta \int_{\Sph^d} Q_\zeta^{q-1} (1+R)\,d\omega =0$. (The invertibility of the relevant matrix in the application of the implicit function theorem follows by a lengthy, but straightforward computation. Indeed, $D^2_\zeta \big|_{\zeta=0} \int_{\Sph^d} Q_\zeta^{q-1} \,d\omega$ is a nonzero multiple of the identity matrix.) Since, by assumption, the condition $\nabla_\zeta \int_{\Sph^d} Q_\zeta^{q-1} (1+R)\,d\omega =0$ is satisfied at $\zeta=0$, we infer that, when restricted to $B_{\epsilon_2}$, the supremum in \eqref{eq:distoptsphere2} is attained at $\zeta=0$.
	
	Let us show that, by decreasing $\epsilon_1$ if necessary, we can ensure that it is not attained outside of $B_{\epsilon_2}$. We note that
	$$
	\eta:= 1- |\Sph^d|^{-1} \sup_{|\zeta|\geq \epsilon_2} \left| \int_{\Sph^d} Q_\zeta^{q-1}\,d\omega \right| >0 \,.
	$$
	Indeed, by H\"older we have $|\cdot|\leq \|Q_\zeta\|_q^{q-1}|\Sph^d|^{1/q} = |\Sph^d|$ with equality if and only if $Q_\zeta$ is a constant. Since $|\zeta|\geq\epsilon_2$, $Q_\zeta$ cannot be a constant, and by continuity we deduce $\eta>0$. (We note that this argument can be made quantitative via the quantitative version of H\"older's inequality in \cite{CaFrLi}.) Thus, if $|\zeta|\geq \epsilon_2$,
	$$
	\left| \int_{\Sph^d} Q_\zeta^{q-1} (1+R)\,d\omega \right| \leq (1-\eta)|\Sph^d| + \|Q_\zeta\|_q^{q-1}\|R\|_q \leq (1-\eta)|\Sph^d| + |\Sph^d|^{(q-1)/q} S_{d,s}^{-1/2} \mathcal E_s[R]^{1/2}
	$$
	Thus, if $|\zeta|\geq\epsilon_2$ and $\mathcal E_s[R]\leq S_{d,s}|\Sph^d|^{2/q}\eta^2$, then
	$$
	\left| \int_{\Sph^d} Q_\zeta^{q-1} (1+R)\,d\omega \right| \leq |\Sph^d| = \left| \int_{\Sph^d} Q_0^{q-1} (1+R)\,d\omega \right|.
	$$
	This means that in \eqref{eq:distoptsphere2} the supremum can be restricted to $|\zeta|<\epsilon_2$, where it is attained at the origin, as we have seen. This proves (d) with $\epsilon_0:=\min\{\epsilon_1, S_{d,s}|\Sph^d|^{2/q}\eta^2\}$.	
\end{proof}


\subsection*{The lower bound}

We now turn to the proof of Theorem \ref{be}. The main step of the proof is contained in the following proposition, where we abbreviate (suppressing the $s$-dependence)
$$
\delta[u] := \inf_{g\in\mathcal G} \|(-\Delta)^{s/2}(u-g)\|_2 \,.
$$
Also, we write
$$
\tau_u:= \delta[u]/\sqrt{\|(-\Delta)^{s/2}u\|_2^2 - \delta[u]^2}
\qquad\text{if}\ 0\neq u\in \dot H^s(\R^d) \,.
$$
By (a) in Lemma \ref{dist} and conformal invariance, one sees that $\tau_u$ is well defined.

\begin{proposition}\label{beprop}
	Let $0<s<\frac d2$ and $q:=\frac{2d}{d-2s}$. Then, for all $0\neq u\in\dot H^s(\R^d)$,
	$$
	\|(-\Delta)^{s/2} u\|_2^2 - S_{d,s} \| u \|_q^2 - \tfrac{4s}{d + 2s +2}\,\delta[u]^2 \gtrsim - \left( \tau_u^{\min\{q-2,1\}} +  \tau_u^{q-2} \right) \delta[u]^2 \,.
	$$
\end{proposition}

\begin{proof}
	Using the stereographic projection, we cast the inequality in the proposition into an equivalent inequality on the sphere. Namely, for $0\neq U\in H^s(\Sph^d)$,
	\begin{equation}
		\label{eq:bepropsphere}
		\mathcal E_s[U] - S_{d,s} \| U \|_q^2 - \tfrac{4s}{d+2s+2}\, \delta[U]^2
		\gtrsim - \left( \tau_U^{\min\{q-2,1\}} +  \tau_U^{q-2} \right) \delta[U]^2 \,.
	\end{equation}
	By Lemma \ref{dist} the infimum $\delta[U]$ is attained, and by conformal invariance we may assume that it is attained at a constant function $c$. We write
	$$
	U = c + R
	$$
	and recall that $R$ satisfies the orthogonality conditions \eqref{eq:beortho}.

	Using the elementary inequality
	$$
	\left| |a|^q - |b|^q - q |b|^{q-2} b(a-b) - \tfrac12 q(q-1) |b|^{q-2}(a-b)^2 \right|\lesssim |b|^{(q-3)_+} |a-b|^{\min\{q,3\}} + |a-b|^q \,,
	$$
	valid for all $a,b\in\R$, together with the first condition in \eqref{eq:beortho}, we find
	$$
	\left| \int_{\Sph^d} |U|^q\,d\omega - |c|^q |\Sph^d| - \tfrac12 q(q-1) |c|^{q-2} \int_{\Sph^d} R^2\,d\omega \right| \lesssim |c|^{(q-3)_+} \|R\|_{q}^{\min\{q,3\}} + \|R\|_q^q \,.
	$$
	Using the elementary inequality
	$$
	(1+t)^{2/q} \leq 1 + \tfrac 2q\, t \,,
	$$
	valid for all $t\geq 0$, we deduce
	$$
	\| U \|_q^2 \leq |\Sph^d|^{2/q} c^2 + (q-1) |\Sph^d|^{-1+2/q}\! \int_{\Sph^d} R^2\,d\omega + \const |c|^{2-q} \left( |c|^{(q-3)_+} \|R\|_{q}^{\min\{q,3\}} + \|R\|_q^q \right).
	$$
	Meanwhile, by \eqref{eq:qzetael} and \eqref{eq:beortho},
	$$
	\mathcal E_s[U] = c^2\, \mathcal E_s[1] + \mathcal E_s[R] \,.
	$$
	Combining the previous two relations and recalling $\mathcal E_s[1] = S_{d,s}|\Sph^d|^{2/q}$ yields
	\begin{align}
		\label{eq:beexp}
		\mathcal E_s[U] - S_{d,s} \|U\|_q^2 & \geq \mathcal E_s[R] - (q-1) \mathcal E_s[1] |\Sph^d|^{-1} \int_{\Sph^d} R^2\,d\omega \notag \\
		& \quad - \const |c|^{2-q} \left( |c|^{(q-3)_+} \|R\|_{q}^{\min\{q,3\}} + \|R\|_q^q \right).
	\end{align}
	In terms of the decomposition of $R$ into spherical harmonics, the quadratic terms on the right side are equal to
	\begin{align*}
		\mathcal E_s[R] - (q-1) \mathcal E_s[1] |\Sph^d|^{-1}\! \int_{\Sph^d} R^2\,d\omega 
		& = \sum_{\ell=0}^\infty \left( \frac{\Gamma(\ell+\frac d2+s)}{\Gamma(\ell+\frac d2-s)} - 
		(q-1) \mathcal E_s[1] |\Sph^d|^{-1} \right) \| P_\ell R\|_2^2 \\
		& = \sum_{\ell=0}^\infty \left( \frac{\Gamma(\ell+\frac d2+s)}{\Gamma(\ell+\frac d2-s)} - 
		\frac{\Gamma(1+\frac d2+s)}{\Gamma(1+\frac d2-s)} \right) \|P_\ell R\|_2^2 \,.
	\end{align*}
	We now recall that, according to \eqref{eq:beortho},
	$$
	P_0 R = P_1 R = 0 \,,
	$$
	so the above sum can be restricted to $\ell\geq 2$. It follows from \eqref{eq:gammamono} that
	\begin{align*}
		\frac{\Gamma(\ell+\frac d2+s)}{\Gamma(\ell+\frac d2-s)} - 
		\frac{\Gamma(1+\frac d2+s)}{\Gamma(1+\frac d2-s)} 
		& \geq \left( 1 - \frac{\Gamma(2+\frac d2-s)}{\Gamma(2+\frac d2+s)}
		\frac{\Gamma(1+\frac d2+s)}{\Gamma(1+\frac d2-s)} \right)
		\frac{\Gamma(\ell+\frac d2+s)}{\Gamma(\ell+\frac d2-s)} \\
		& = \frac{2s}{1+\frac d2+s} \,
		\frac{\Gamma(\ell+\frac d2+s)}{\Gamma(\ell+\frac d2-s)} \,.
	\end{align*}
	To summarize, we have shown that
	$$
	\mathcal E_s[R] - (q-1) \mathcal E_s[1] |\Sph^d|^{-1} \int_{\Sph^d} R^2\,d\omega
	\geq \frac{2s}{1+\frac d2+s} \,\mathcal E_s[R] = \frac{2s}{1+\frac d2+s} \,\delta[U]^2 \,.
	$$
	
	It remains to deal with the remainder terms in \eqref{eq:beexp}. Using Sobolev's inequality and the inequality $\|R\|_q \leq\tau_U |\Sph^d|^{1/q}|c|$ from (b) in Lemma \ref{dist}, we find
	\begin{align*}
		|c|^{2-q} \left( |c|^{(q-3)_+} \|R\|_{q}^{\min\{q,3\}} + \|R\|_q^q \right)
		& \lesssim |c|^{2-q} \left( |c|^{(q-3)_+} \|R\|_{q}^{\min\{q-2,1\}} + \|R\|_q^{q-2} \right) \delta[U]^2 \\
		& \lesssim \left( \tau_U^{\min\{q-2,1\}} + \tau_U^{q-2} \right) \delta[U]^2 \,.
	\end{align*}
	This leads to the claimed bound \eqref{eq:bepropsphere}.
\end{proof}

\begin{proof}[Proof of Theorem \ref{be}]
	We argue by contradiction, assuming the claimed inequality would not hold. Then there is a sequence $(u_n)\subset\dot H^s(\R^d)$ such that
	\begin{equation}
		\label{eq:becontra}
		\frac{\|(-\Delta)^{s/2}u_n\|_2^2 - S_{d,s}\|u_n\|_q^2}{\delta[u_n]^2}\to 0 \,.
	\end{equation}
	By homogeneity we may assume that $\|(-\Delta)^{s/2}u_n\|_2=1$. Then the Sobolev inequality and the inequality $\delta[u_n]\leq \|(-\Delta)^{s/2}u_n\|_2 = 1$ imply
	$$
	0 \leq 1 - S_{d,s}\|u_n\|_q^2 \leq \frac{\|(-\Delta)^{s/2}u_n\|_2^2 - S_{d,s}\|u_n\|_q^2}{\delta[u_n]^2} \,.
	$$
	By \eqref{eq:becontra} we deduce that $\|u_n\|_q^2\to S_{d,s}^{-1}$. It follows from Lions's theorem proved in the first lecture that $\delta[u_n]\to 0$. Using this information in the inequality in Proposition \ref{beprop}, we obtain
	$$
	\liminf_{n\to\infty} \frac{\|(-\Delta)^{s/2}u_n\|_2^2 - S_{d,s}\|u_n\|_q^2}{\delta[u_n]^2} \geq \frac{4s}{d+2s+2} \,.
	$$
	This contradicts \eqref{eq:becontra} and completes the proof of Theorem \ref{be}.
\end{proof}

\subsection*{Bibliographical remarks}

The main theorem in this lecture for $s=1$, as well as the strategy employed for general $s$, are due to Bianchi and Egnell \cite{BiEg}. They answered a question posed by Brezis and Lieb \cite{BrLi}. The result for general $s$, as well as the observation to use conformal invariance, appeared in \cite{ChFrWe}; earlier results in the local case (that is, for integer $s$) are in \cite{LuWe,BaWeWi}.

The basic ingredients of the Bianchi--Egnell method are a compactness theorem for optimizing sequences and the fact that all zero modes of the Hessian come from symmetries. For functional inequalities for which these two ingredients are available there is a good chance that the Bianchi--Egnell method can be applied. This has been carried out in a large number of cases; see, for instance, the introduction of \cite{Fr} for references.

We do not attempt to give an overview over the works on the stability problem of functional inequalities in the last two decades. Let us just mention the works \cite{FuMaPr,FiMaPr,CiLe} on the isoperimetric inequality, which had a huge impact on the field, as well as the surveys \cite{Fi1,Fi2,DoEs}. Related to the topic of these lectures, we mention the stability result \cite{Ca} for the HLS inequality. Indeed, this is deduced from the main result of the present lecture together with a quantitative version of the duality argument used in the proof of Lemma \ref{dual} in the appendix of the first lecture.

A stability result for the Sobolev inequality in the case $s=1$ appears in \cite{BoDoNaSi}. It is of a somewhat different flavor than Theorem \ref{be} and has an explicit constant in the bound, at the expense of being only valid for functions $u$ with sufficiently fast decay.

After the lectures at the summer school on which these notes are based, there have been some developments concerning the stability theorem, which we briefly describe. Let us denote the optimal constant in the stability theorem by
$$
c^{\rm BE}_{d,s} := \inf_{u\not\in\mathcal G} \frac{\|(-\Delta)^{s/2} u\|_2^2 - S_{d,s}\|u\|_q^2}{\inf_{g\in\mathcal G} \|(-\Delta)^{s/2}(u-g)\|_2^2} \,.
$$
Note that the proof presented in this lecture used compactness and only showed that $c^{\rm BE}_{d,s}$ is positive, without giving any lower bound. In the case $s=1$, the paper \cite{DoEsFiFrLo} provided for the first time an explicit lower bound on $c^{\rm BE}_{d,1}$. This was achieved by replacing the use of Lions's compactness theorem by a more precise argument based on rearrangement methods, using both a discrete and a continuous symmetrization flow. In fact, in \cite{DoEsFiFrLo} it was shown that
$$
c^{\rm BE}_{d,1}\gtrsim d^{-1} \,,
$$
which is optimal with respect to its large $d$-behavior in view of the upper bound $c^{\rm BE}_{d,1}\leq 4/(d+4)$ from \eqref{eq:bereverse}. This was achieved by cutting the remainder $R$ in a suitable way and using the above Taylor expansion of the nonlinearity only where $R$ is sufficiently small compared to $c$. The first part of the argument, namely that giving an explicit lower bound, extends to general $s$, as shown in \cite{ChLuTa}. The optimal behavior of the constant with respect to $d$ for $s=1$ leads to a quantitative version of the logarithmic Sobolev inequality \cite{DoEsFiFrLo}.

Another development occurred in \cite{Ko1,Ko2}, where it was shown that for general $s$ the constant $c_{d,s}^{\rm BE}$ is attained. This involves, among other things, showing that the inequality in \eqref{eq:bereverse} is strict and uses a compactness argument, reminiscent of but much more intricate than those presented in Lecture 1.


\section*{Lecture 4: Nondegenerate and degenerate stability}

In this final lecture we discuss the following one-parameter family of Sobolev-type inequalities
\begin{equation*}
	\int_0^T \int_{\Sph^{d-1}} \left( |\partial_t u |^2 + |\nabla_\omega u|^2 + \tfrac{(d-2)^2}{4} u^2\right)d\omega\,dt \geq S_d(T) \left( \int_0^T \int_{\Sph^{d-1}} |u|^q \,d\omega\,dt \right)^{2/q},
\end{equation*}
valid for functions $u\in H^1((\R/T\Z)\times\Sph^{d-1})$. Here, as always,
$$
d\geq 3
\qquad\text{and}\qquad
q = \tfrac{2d}{d-2} \,.
$$
We abbreviate $\Sigma_T:=(\R/T\Z)\times\Sph^{d-1}$, $dv_g=d\omega\,dt$, $\|u\|_q:= \|u\|_{L^q(\Sigma_T,v_g)}$ and, for $u\in H^1(\Sigma_T)$,
$$
\mathcal E_T[u] := \int_{\Sigma_T} \left( |\partial_t u |^2 + |\nabla_\omega u|^2 + \tfrac{(d-2)^2}{4} u^2\right)dv_g \,.
$$
(There should be no risk of confusing this with $\mathcal E_s[U]$ from the previous two lectures.) By $S_d(T)$ we denote the optimal constant in the above inequality, that is,
\begin{equation}
	\label{eq:l4sob}
	S_d(T) := \inf_{0\neq u\in H^1(\Sigma_T)} \frac{\mathcal E_T[u]}{\|u\|_q^2} \,.
\end{equation}

Attention to these inequalities was drawn by Schoen in connection with the Yamabe problem. A remarkable feature is the existence of a critical parameter
$$
T_* := \tfrac{2\pi}{\sqrt{d-2}}
$$
such that for $T\leq T_*$ the infimum in \eqref{eq:l4sob} is attained precisely at constants, while for $T>T_*$ it is attained at a nonconstant function that is independent of the variable $\omega$. Our goal is to show that for $T\neq T_*$ one has a quadratic stability similar as in the previous lecture, while for $T=T_*$ one only has a quartic stability.

We denote by
$$
\mathcal G_T:= \left\{ g\in H^1(\Sigma_T) :\ \mathcal E_T[g] = S_d(T) \|g\|_q^2 \right\}
$$
the set of all optimizers (and zero).

The main result of this lecture is the following stability theorem for \eqref{eq:l4sob}.

\begin{theorem}\label{stabt}
	Let $d\geq 3$, $q=\frac{2d}{d-2}$ and $T>0$. Then, for all $u\in H^1(\Sigma_T)$,
	$$
	\mathcal E_T[u] -S_d(T)\|u\|_q^2 \gtrsim
	\begin{cases}
		\inf_{g\in\mathcal G_T} \mathcal E_T[u-g] & \text{if}\ T\neq T_* \,, \\
		\inf_{g\in\mathcal G_T} \frac{\mathcal E_T[u-g]^2}{\mathcal E_T[u]} & \text{if}\ T=T_* \,.
	\end{cases}
	$$
\end{theorem}

Of course, the constant implicit in the $\gtrsim$ depends on $T$.

We will also show that the order of vanishing given by the theorem is optimal. That is, we will show that
\begin{equation}
	\label{eq:stabtupper}
	\inf_{u\not\in\mathcal G_T} \frac{\mathcal E_T[u] -S_d(T)\|u\|_q^2}{\inf_{g\in\mathcal G_T} \mathcal E_T[u-g]} \leq c_T
\end{equation}
with a certain constant $c_T<\infty$ defined in \eqref{eq:defct} below. Moreover, at $T=T_*$ we have $c_{T_*}=0$ and we show
\begin{equation}
	\label{eq:stabtuppercrit}
	\inf_{u\not\in\mathcal G_{T_*}} \frac{\mathcal E_{T_*}[u] \left( \mathcal E_{T_*}[u] -S_d(T_*)\|u\|_q^2\right)}{\inf_{g\in\mathcal G_{T_*}} \mathcal E_{T_*}[u-g]^2} \leq \frac{(q+2)(q-2)}{12(q-1)} \,.
\end{equation}
The bounds \eqref{eq:stabtupper} and \eqref{eq:stabtuppercrit} imply that one cannot have a better stability result than a quadratic one if $T\neq T_*$ and a quartic one if $T=T_*$.

We emphasize that we refer to \eqref{eq:stabtupper} (resp.\ \eqref{eq:stabtuppercrit}) as quadratic (resp.\ quartic) stability, since the term $\inf_{g\in\mathcal G_T} \mathcal E_T[u-g]$ vanishes quadratically as $u$ approaches $\mathcal G_T$.

The basic strategy to prove Theorem \ref{stabt} is the same as that in the previous lectures: We prove a compactness theorem, classify the optimizers and the zero modes of their Hessian and then we put these ingredients together. (We will not give a full proof of the classification of optimizers, but refer to the literature at some points.) For $T\neq T_*$ this works in a straightforward way. The case $T=T_*$, however, is different since there is a zero mode of the Hessian that does \emph{not} come from symmetries. This is responsible for the quartic behavior and on a technical level necessitates a certain iteration of the basic strategy, which we will explain.


\subsection*{Optimizing sequences}

We begin by proving relative compactness of optimizing sequences for the optimization problem \eqref{eq:l4sob}.

\begin{proposition}\label{exoptpert}
	Let $T>0$. Let $(u_n)\subset H^1(\Sigma_T)$ with $\mathcal E_T[u_n]=1$ and $\|u_n\|_q^2 \to S_d(T)^{-1}$. Then there is a subsequence that converges in $H^1(\Sigma_T)$ to an optimizer of \eqref{eq:l4sob}.
\end{proposition}

Note that, in contrast to the corresponding theorem in Lecture 1, here there are no noncompact symmetries that could lead to a loss of compactness.

The proof of the proposition relies on the following strict upper bound on $S_d(T)$. We denote by $S_d:=S_{d,1}$ the constant from the first three lectures.

\begin{lemma}\label{strict}
	For all $T>0$, $S_d(T)<S_d$.
\end{lemma}

\begin{proof}[Proof of Lemma \ref{strict}]
	Let $Q(t):= \cosh^{-(d-2)/2}t$ and note that
	$$
	-Q'' + (\tfrac{d-2}{2})^2\, Q = \tfrac{d(d-2)}{4}\, Q^{(d+2)/(d-2)} 
	\qquad\text{in}\ \R \,.
	$$
	Taking $t\mapsto Q(\cdot - T/2)$, considered as a function on $\Sigma_T$, as a trial function, we obtain
	\begin{align*}
		S_d(T) & \leq \frac{2|\Sph^{d-1}| \int_0^{T/2} ((Q')^2 + (\frac{d-2}{2})^2 Q^2)\,dt}{\left( 2 |\Sph^{d-1}| \int_0^{T/2} Q^q\,dt \right)^{2/q}} \\
		& = \frac{d(d-2)}{4} \left( 2 |\Sph^{d-1}| \int_0^{T/2} Q^q\,dt \right)^{1-2/q} + \frac{2|\Sph^{d-1}| Q'(T/2)Q(T/2)}{\left( 2 |\Sph^{d-1}| \int_0^{T/2} Q^q\,dt \right)^{2/q}} \,.
	\end{align*}
	The boundary term is $<0$ since $Q$ is positive and decreasing on $(0,\infty)$, and the bulk term is
	$$
	< \frac{d(d-2)}{4} \left( 2 |\Sph^{d-1}| \int_0^\infty Q^q\,dt \right)^{1-2/q} = S_d \,.
	$$
	The last equality can be seen either by explicit computation and comparison with the value of $S_d$, or by noting that the $Q$ in this proof coincides with the $Q$ in the characterization of optimizers of $S_d$ in introducing logarithmic coordinates (see, for instance, \eqref{eq:logcoord} below).
\end{proof}

\begin{proof}[Proof of Proposition \ref{exoptpert}]
	We proceed similarly as in Step 2 of the proof of the main theorem in Lecture 1. After passing to a subsequence, we may assume that $u_n\rightharpoonup u$ in $H^1(\Sigma_T)$. We write
	$$
	u_n = u + r_n
	\qquad\text{with}\ r_n\rightharpoonup 0 \ \text{in}\ H^1(\Sigma_T) \,.
	$$
	By the same arguments as in Lecture 1, we deduce that
	\begin{equation}
		\label{eq:mm1l4}
		t:= \lim_{n\to\infty} \mathcal E_T[r_n]
		\qquad\text{exists and satisfies}\qquad
		1 = \mathcal E_T[u] + t
	\end{equation}
	and
	\begin{equation}
		\label{eq:mm2l4}
		m:= \lim_{n\to\infty} \| r_n\|_q^q
		\qquad\text{exists and satisfies}\qquad
		S_{d}(T)^{-q/2} = \| u \|_q^q + m \,. 
	\end{equation}
	In contrast to Lecture 1, we will argue more carefully when estimating $t$ from below by $m$. We will show that
	\begin{equation}
		\label{eq:mm3l4}
		t \geq S_{d} m^{2/q} \,,
	\end{equation}
	where $S_d:=S_{d,1}$ is the constant from the first three lectures. We emphasize that \eqref{eq:mm3l4} with $S_d(T)$ instead of $S_d$ would be immediate. Since $S_d>S_d(T)$ by Lemma \ref{strict}, \eqref{eq:mm3l4} as it stands is an improvement of this immediate bound, and this improvement will be crucial in our proof.
	
	Inequality \eqref{eq:mm3l4} follows from an inequality of Aubin, which says that for any $\epsilon>0$ there is a $C_{\epsilon,T}<\infty$ such that for all $v\in H^1(\Sigma_T)$,
	\begin{equation}
		\label{eq:aubin}
		\mathcal E_T[v] \geq (1-\epsilon) S_d\|v\|_q^2 - C_{\epsilon,T} \|v\|_2^2 \,.
	\end{equation}
	In fact, Aubin's inequality is valid on any closed Riemannian manifold $(\mathcal M,g)$ of dimension $d\geq 3$, provided the left side is replaced by $\int_\mathcal M \left( |\nabla v|^2_g + \frac{d-2}{4(d-1)}R_g v^2\right) dv_g$, where $R_g$ denotes the scalar curvature. For the proof of \eqref{eq:aubin} one covers the manifold by finitely many balls, whose radii are so small that in each ball the metric is Euclidean `up to an $\epsilon$'. Then one localizes to these balls using a partition of unity and applies in each ball the Euclidean Sobolev inequality. The error term involving $C_{\epsilon,T}$ comes from the localization error. 
	
	For the proof of \eqref{eq:mm3l4} we apply \eqref{eq:aubin} to $v=r_n$. Since $r_n\rightharpoonup 0$ in $H^1(\Sigma_T)$ implies $r_n\to 0$ in $L^2(\Sigma_T)$, we deduce that $t\geq (1-\epsilon) S_d m^{2/q}$. Since $\epsilon>0$ is arbitrary, we obtain \eqref{eq:mm3l4}.
	
	We now deduce the proposition from \eqref{eq:mm1l4}, \eqref{eq:mm2l4} and \eqref{eq:mm3l4}. We find
	\begin{align*}
		1 & = \mathcal E_T[u] + t \geq \mathcal E_T[u] + S_d(T) m^{2/q} + (1-S_d(T)/S_d) t \\
		& = \mathcal E_T[u] + (1- S_d(T)^{q/2}\|u\|_q^q)^{2/q} + (1-S_d(T)/S_d) t \\
		& \geq \mathcal E_T[u] + 1- S_d(T) \|u\|_q^2 + (1-S_d(T)/S_d) t \,, 
	\end{align*}
	where we used the same elementary inequality \eqref{eq:elem} as in Lecture 1. Using the strict inequality $S_d(T)<S_d$ we deduce that
	\begin{equation}
		\label{eq:exoptpert}
		\mathcal E_T[u] = S_d(T)\|u\|_q^2
		\qquad\text{and}\qquad
		t=0 \,.
	\end{equation}
	Because of \eqref{eq:mm1l4} we deduce from the second condition in \eqref{eq:exoptpert} that $r_n\to 0$ in $H^1(\Sigma_T)$, that is, $u_n\to u$ in $H^1(\Sigma_T)$. In particular, $u\neq 0$ and, by the first condition in \eqref{eq:exoptpert}, $u$ is an optimizer. This completes the proof.
\end{proof}

Note that there is a certain analogy between the above proof of Proposition \ref{exoptpert} and the proof of Theorem \ref{comp} that was presented in the appendix to Lecture 2. In both cases the compactness comes from an improved constant in front of the $L^q$ term at the expense of adding an $L^2$ term.


\subsection*{Optimizers}

According to the previous proposition, for any $T$ there is an optimizer $u_*$ for the optimization problem \eqref{eq:l4sob}. 

We claim that either $u_*\geq 0$ or $u_*\leq 0$. To see this, we recall that by Sobolev space theory, the positive and negative parts $(u_*)_\pm$ of $u_*$ belong to $H^1(\Sigma_T)$ and $\|\nabla u_*\|_2^2 = \|\nabla (u_*)_+\|_2^2 + \|\nabla (u_*)_-\|_2^2$. Thus, if neither $(u_*)_+$ nor $(u_*)_-$ vanish almost everywhere, then
$$
S_d(T) = \frac{\mathcal E_T[u_*]}{\|u_*\|_q^2} = \theta^{2/q} \frac{\mathcal E_T[(u_*)_+]}{\|(u_*)_+\|_q^2} + (1-\theta)^{2/q} \frac{\mathcal E_T[(u_*)_-]}{\|(u_*)_-\|_q^2} \geq \left( \theta^{2/q} + (1-\theta)^{2/q} \right) S_d(T)
$$
with
$$
\theta := \frac{\|(u_*)_+\|_q^q}{\|(u_*)_+\|_q^q + \|(u_*)_-\|_q^q} \,.
$$
Since $0<\theta<1$ and $q>2$ we have $\theta^{2/q}+(1-\theta)^{2/q}>1$, a contradiction. Thus, after changing the sign of $u_*$ if necessary, we may assume that $u_*\geq 0$.

\medskip

The Euler--Lagrange equation of the optimization problem is
$$
-\partial_t^2 u_* -\Delta_\omega u_* + \tfrac{(d-2)^2}{4}\, u_* - \tfrac{\mathcal E_T[u_*]}{\|u_*\|_q^q}\, u_*^{q-1} = 0
\qquad\text{on}\ \Sigma_T \,.
$$
If we define a function $U_*$ on $\R^d\setminus\{0\}$ by
\begin{equation}
	\label{eq:logcoord}
	U_*(x):= |x|^{-(d-2)/2}\, u_*(\ln |x|, x/|x|) \,,
\end{equation}
then, by a straightforward computation,
$$
-\Delta U_* = \tfrac{\mathcal E_T[u_*]}{\|u_*\|_q^q}\, U_*^{q-1}
\qquad\text{in}\ \R^d\setminus\{0\} \,.
$$
Moreover, the singularity at the origin is nonremovable in the sense that
$$
\int_{B_\epsilon(0)} U_*^q\,dx = \infty
\qquad\text{for all}\ \epsilon>0 \,.
$$
(Indeed, the integral
$$
\int_{r<|x|<e^T r} U_*^q\,dx = \| u_* \|_q^q
$$
is independent of $r>0$.) Therefore, one is in position to apply a theorem of Caffarelli, Gidas and Spruck \cite{CaGiSp}, which implies that $U_*$ is radially symmetric about the origin. Equivalently, the function $u_*$ is independent of $\omega$.

Normalizing $u_*$ such that $\tfrac{\mathcal E_T[u_*]}{\|u_*\|_q^q} = \frac{d(d-2)}{4}$, one is therefore led to the ODE
$$
-\partial_t^2 u + \tfrac{(d-2)^2}{4}\, u - \tfrac{d(d-2)}4\, u^{q-1} = 0
\qquad\text{in}\ \R \,,
$$
which can be studied by phase-plane analysis. The equation has the constant solution 
$$
u_0:=((d-2)/d)^{(d-2)/4}
$$
as well as the homoclinic solution $\cosh^{-(d-2)/2}$. For any $\alpha\in(u_0,1)$ there is a unique solution $u= u_\alpha$ with $u(0)=\alpha$ and $u'(0)=0$. This solution is positive and periodic with a certain minimal period $\tau(\alpha)$ and it is symmetric about $t=0$ and decreasing on $[0,\tau(\alpha)/2]$. It is known that $\alpha\mapsto \tau(\alpha)$ is continuous and monotone increasing (see, e.g., \cite{CaChRu}) with
$$
\lim_{\alpha\to u_0} \tau(\alpha) = \frac{2\pi}{\sqrt{d-2}} = T_*
\qquad\text{and}\qquad
\lim_{\alpha\to 1} \tau(\alpha) = \infty \,.
$$

Returning to our solution $u_*$, which is $T$-periodic, we conclude immediately that $u_*=u_0$ if $T\leq T_*$. Assume now $T>T_*$. We deduce from the above analysis that there is an $s\in\R/T\Z$ and a $k\in\N$ with $k<T/T_*$ such that $u_*(t) = u_\alpha(t-s)$, where $\alpha\in(u_0,1)$ is uniquely determined by $k$ via $T = k\tau(\alpha)$. If $T\leq 2T_*$, we have necessarily $k=1$. If $T>2T_*$, a priori more than one value of $k$ is possible, but, using a variational argument based on the stability of the solutions \cite[Section 2, p.\ 134]{Sc}, one can show that the minimizer necessarily has $k=1$. To summarize, we have shown that
\begin{align*}
	\mathcal G_T =
	\begin{cases}
		\{ c:\ c\in\R \} & \text{if}\ T\leq T_* \,, \\
		\{ c\, u_{\tau^{-1}(T)}(\cdot - s) :\ c\in\R \,,\ s\in\R/T\Z \} & \text{if}\ T>T_* \,.
	\end{cases}
\end{align*}
Here $\tau^{-1}:(T_*,\infty)\to(u_0,1)$ denotes the inverse of $\tau$, which exists by the strict monotonicity of $\tau$.


\subsection*{Zero modes of the Hessian}

Having characterized the optimizers, we next turn our attention to the zero modes of the Hessian. As before we work in the normalization $u_*\geq 0$ and $\tfrac{\mathcal E_T[u_*]}{\|u_*\|_q^q} = \frac{d(d-2)}{4}$, so that $u_*$ solves
\begin{equation}
	\label{eq:eql4}
	-\partial_t^2 u_* + \tfrac{(d-2)^2}{4}\, u_* - \tfrac{d(d-2)}4\, u_*^{q-1} = 0
	\qquad\text{in}\ \R \,.
\end{equation}
The Hessian of our optimization problem is the operator
\begin{align*}
	\mathcal L_T & := -\partial_t^2 -\Delta_\omega + \tfrac{(d-2)^2}{4} - (q-1)\, \tfrac{d(d-2)}4\, u_*^{q-2} + (q-2)\,\tfrac{d(d-2)}{4}\, \|u\|_q^{-q} \left| u_*^{q-1}\rangle\langle u_*^{q-1} \right| \\
	& = -\partial_t^2 -\Delta_\omega + \tfrac{(d-2)^2}{4} - \tfrac{d(d+2)}4\, u_*^{q-2} + d\, \|u\|_q^{-q} \left| u_*^{q-1}\rangle\langle u_*^{q-1} \right|,
\end{align*}
considered as a selfadjoint, lower bounded operator in $L^2(\Sigma_T)$ with form domain $H^1(\Sigma_T)$. Our goal is to show that
\begin{equation}
	\label{eq:kernell}
	\ker \mathcal L_T =
	\begin{cases} \{u_*\} & \text{if}\ T<T_* \,,\\
		\spa\{ u_*,\ \sin( \frac{2\pi}{T_*} \, \cdot ),\, \cos( \frac{2\pi}{T_*}\, \cdot ) \} & \text{if}\ T=T_* \,,\\
		\spa\{ u_*,\ \partial_t u_* \} & \text{if}\ T>T_* \,.
	\end{cases} 
\end{equation}

The fact that $u_*$ is in the kernel comes from the homogeneity of the optimization problem. The fact that $\partial_t u_*$ is in the kernel for $T>T_*$ comes from the fact that translates of an optimizer are again optimizers and that optimizers are not constant. From that perspective the claimed elements in the kernel for $T>T_*$ are natural and the thrust of the assertion in this case lies in the fact that there are no other, linearly independent elements in the kernel. Similarly, in the case $T<T_*$ it is shown that there is no element linearly independent from the natural one. What might be surprising is that at the critical value $T=T_*$ the sine and cosine in the kernel are not related to any symmetry of the problem. This is ultimately the reason why in the main theorem in this lecture the stability exponent is $4$ for $T=T_*$, while it is $2$ for $T\neq T_*$.

\begin{proof}[Proof of \eqref{eq:kernell}]
	For $T\leq T_*$ we have $u_* = u_0 = ((d-2)/d)^{(d-2)/4}$ and therefore
	\begin{equation}
		\label{eq:ltsmallt}
		\mathcal L_T = -\partial_t^2 -\Delta_\omega - (d-2) + (d-2) |\Sigma_T|^{-1} \left| 1\rangle\langle 1 \right|.
	\end{equation}
	The assertion follows easily from the spectral properties of $-\partial_t^2$ and $-\Delta_\omega$.
	
	Now assume $T>T_*$. We have to classify all $\phi\in H^2(\Sigma_T)$ such that $\mathcal L_T\phi=0$. We first note that, as a consequence of \eqref{eq:eql4}, $\mathcal L_T u_*=0$ and $\mathcal L_T \partial_t u_* =0$. Now given $\phi$ as above, we consider
	$$
	\widetilde\phi := \phi + c u_*
	\qquad\text{with}\qquad
	c := \|u_*\|_q^{-q} \int_{\Sigma_T} u_*^{q-1}\phi\,dv_g \,.
	$$
	Then a simple computation shows that
	$$
	\widetilde{\mathcal L}_T \widetilde\phi = 0 
	\qquad\text{with}\qquad
	\widetilde{\mathcal L}_T:= -\partial_t^2 -\Delta_\omega + \tfrac{(d-2)^2}{4} - \tfrac{d(d+2)}4\, u_*^{q-2} \,.
	$$
	To solve this equation, we can expand $\widetilde\phi$ with respect to spherical harmonics in the $\omega$-variable and solve the equation for each fixed degree. This leads to the equations
	$$
	\widetilde{\mathcal L}_{T,\ell} \widetilde\phi_\ell = 0 
	\qquad\text{with}\qquad
	\widetilde{\mathcal L}_{T,\ell}:= -\partial_t^2 + \ell(\ell+d-2) + \tfrac{(d-2)^2}{4} - \tfrac{d(d+2)}4\, u_*^{q-2} \,,
	$$
	parametrized by $\ell\in\N_0$, where now $\widetilde{\mathcal L}_{T,\ell}$ is an operator in $L^2(\R/T\Z)$ with form domain $H^1(\R/T\Z)$.
	
	We begin with $\ell=0$. Differentiating \eqref{eq:eql4} with respect to either $t$ or $\alpha$ (recall that $u_* = u_\alpha$ where $\tau(\alpha)=T$), we find two solutions $\partial_t u_*$ and $\partial_\alpha u_*$ of the equation
	$$
	-v'' + \tfrac{(d-2)^2}{4}v - \tfrac{d(d+2)}4\, u_*^{q-2} v =0 \,.
	$$
	Since $\partial_t u_*(0)=0$ and $\partial_\alpha u_*(0)=1$ (since $u_*(0)=\alpha$), these two solutions are linearly independent. Thus $\widetilde\phi_0$ is a linear combination of these two functions. Differentiating the equation $u_*(t+T)=u_*(t)$ with respect to $\alpha$, we obtain $\partial_\alpha u_*(t+T) = \partial_\alpha u_*(t)  - u_*'(t) \tau'(\alpha)$. Since $\tau'(\alpha)\neq 0$ (see, e.g., \cite{CaChRu}), we see that $\partial_\alpha u_*$ is not periodic, so in fact $\widetilde\phi_0$ is a multiple of $\partial_t u_*$.
	
	Now we consider $\ell=1$. A computation shows that the two functions
	$$
	e^{\pm t} (u_*' \pm \tfrac{d-2}{2} u_*)
	$$
	satisfy the equation
	$$
	-v'' +(d-1)v + \tfrac{(d-2)^2}{4}v - \tfrac{d(d+2)}4\, u_*^{q-2} v =0 \,.
	$$
	At the two infinities, one of them is exponentially growing and one is exponentially decaying. They are clearly linearly independent. Since no nontrivial linear combination of them is periodic, we conclude that $\widetilde\phi_1=0$.
	
	Finally, we consider $\ell\geq 2$. Since $u_*$ is a minimizer, we know that $\mathcal L_T$ is positive semidefinite. Since the rank-one contribution to $\mathcal L_T$ only affects $\ell=0$, we deduce that the operators $\widetilde{\mathcal L}_{T,\ell}$ are positive semidefinite for $\ell\geq 1$. Since $\widetilde{\mathcal L}_{T,1}$ has compact resolvent, the fact that its kernel is trivial implies that it is positive definite. Since $\widetilde{\mathcal L}_{T,\ell}$ with $\ell\geq 2$ differs from $\widetilde{\mathcal L}_{T,1}$ by a positive constant, we deduce that $\widetilde{\mathcal L}_{T,\ell}$ is positive definite as well and, in particular, has trivial kernel.
	
	To summarize, we have shown that $\widetilde\phi = C \partial_t u_*$ for some $C\in\R$, that is, $\phi = -c u_* + C\partial_t u_*$, as claimed.	
\end{proof}

For $T>0$, let
\begin{equation}
	\label{eq:defct}
	c_T:= \inf \left\{ \frac{\langle v,\mathcal L_T v \rangle}{\mathcal E_T[v]} : \ v\in H^1(\Sigma_T) \,,\ \mathcal E_T[u_*,v]=\mathcal E_T[\partial_t u_*,v] =0 \right\}.
\end{equation}
Here $\mathcal E_T[\cdot,\cdot]$ denotes the bilinear form associated to the quadratic form $\mathcal E_T[\cdot]$. We note that for $T\leq T_*$ the function $u_*$ is a constant, so the second orthogonality condition in \eqref{eq:defct} is trivially satisfied in this case.

\begin{lemma}\label{nondeg}
	If $T\leq T_*$, then
	$$
	c_T = \frac{\min\{(\frac{2\pi}{T})^2, d-1\} - (d-2)}{\min\{(\frac{2\pi}{T})^2, d-1\} + (\frac{d-2}{2})^2} \,.
	$$
	In particular, $c_T>0$ if $T<T_*$, and $c_{T_*}=0$. If $T> T_*$, then $c_T>0$.
\end{lemma}

\begin{proof}
	For $T\leq T_*$, $u_*$ is a constant, so the orthogonality conditions in \eqref{eq:defct} reduce to $\int_{\Sigma_T} v\,dv_g=0$ and the operator $\mathcal L_T$ takes the form \eqref{eq:ltsmallt}. Diagonalizing $-\partial_t^2$ and $-\Delta_\omega$, we see that for $T\leq T_*$
	$$
	c_T = \inf_{k\in\Z,\,\ell\in\N_0,\, (k,\ell)\neq (0,0)} \frac{(\frac{2\pi}{T})^2 k^2 + \ell(\ell+d-2)-(d-2)}{(\frac{2\pi}{T})^2 k^2 + \ell(\ell+d-2)+ (\frac{d-2}{2})^2} \,.
	$$
	The claimed result follows by a simple computation.
	
	For $T\geq T_*$ we argue more qualitatively. It is easy to see that the infimum defining $c_T$ is attained by some $v_*\neq 0$. If we had $c_T=0$, then $\langle v_*,\mathcal L_T v_*\rangle =0$ and therefore, since $\mathcal L_T\geq 0$, $\mathcal L_T v_*=0$. By \eqref{eq:kernell}, $v_*$ is a linear combination of $u_*$ and $\partial_t u_*$. Using the equation for $u_*$, we find $\mathcal E_T[u_*,v] = \frac{d(d-2)}{4} \langle u_*^{q-1},v \rangle$ and $\mathcal E_T[\partial_t u_*,v] = \frac{d(d+2)}{4} \langle u_*^{q-2}\partial_t u_*,v \rangle$. In particular, $v_*$ is $L^2$-orthogonal to $u_*$ and $u_*^{q-2}\partial_t u_*$, which implies that $v_*=0$, a contradiction.
\end{proof}


\subsection*{Nondegenerate stability: The upper bound}

We turn to the question of stability and prove Theorem \ref{stabt}. We begin with the simpler case $T\neq T_*$, where we will show nondegenerate stability in the form of a quadratic bound.

As in the previous lecture, it is instructive to first prove the upper bound, namely \eqref{eq:stabtupper}. As there, we make the ansatz $u=u_*+\epsilon r$ with $r$ to be determined, and we find
$$
\lim_{\epsilon\to 0} \epsilon^{-2} \left( \mathcal E_T[u_*+\epsilon r] - S_d(T)\|u_*+\epsilon r\|_q^2 \right) = \langle r,\mathcal L_T r \rangle \,.
$$
Moreover, arguing as in the proof of Lemma \ref{dist}, we find that, if $r$ satisfies
$$
\mathcal E_T[u_*,r]=\mathcal E_T[\partial_t u_*,r] =0
$$
and if $\epsilon$ is sufficiently small, depending on $r$, then
$$
\inf_{g\in\mathcal G_T} \mathcal E_T[u_*+\epsilon r - g] = \epsilon^2 \mathcal E_T[r] \,.
$$
Thus,
$$
\lim_{\epsilon\to 0} \frac{\mathcal E_T[u_*+\epsilon r] - S_d(T)\|u_*+\epsilon r\|_q^2}{\inf_{g\in\mathcal G_T} \mathcal E_T[u_*+\epsilon r - g] } = \frac{\langle r,\mathcal L_T r \rangle}{\mathcal E_T[r]} \,.
$$
By definition, the infimum over the right side with respect to $r$ gives the constant $c_T$ defined in \eqref{eq:defct}.

This proves the expected result that stability cannot hold with a smaller exponent than 2. (Concerning our counting of the vanishing exponent, we note that $\inf_{g\in\mathcal G_T} \mathcal E_T[u - g]$ vanishes quadratically -- that is, with exponent 2 -- as $u$ approaches $\mathcal G_T$.) Moreover, since $c_{T_*}=0$, the above argument shows that at $T=T_*$ no quadratic stability can hold. We will discuss an upper bound for $T=T_*$ later, but first we prove that for $T\neq T_*$ one does indeed have quadratic stability.


\subsection*{Nondegenerate stability: The lower bound}

We are now ready to give the proof of Theorem \ref{stabt} for $T\neq T_*$. We abbreviate, suppressing the $T$-dependence,
$$
\delta[u] := \inf_{g\in\mathcal G_T} \sqrt{ \mathcal E_T[u-g] } \,.
$$

\begin{proposition}
	Let $T\neq T_*$. Then, for all $0\neq u\in H^1(\Sigma_T)$,
	$$
	\mathcal E_T[u] - S_d(T) \|u\|_q^2 - c_T \, \delta[u]^2 \gtrsim - \left( \tau_u^{\min\{q-2,1\}} + \tau_u^{q-2} \right) \delta[u]^2 \,,
	$$
	where $\tau_u := \delta[u] /\sqrt{\mathcal E_T[u] - \delta[u]^2}$ and where $c_T$ is defined in \eqref{eq:defct}.
\end{proposition}

Once we have proved this proposition, we obtain Theorem \ref{stabt} by the same argument as in the previous lecture, based on the relative compactness of optimizing sequences.

\begin{proof}
	It is easy to see that the infimum $\delta[u]$ is attained. After a possible sign change and, in case $T>T_*$, a translation, we may assume that it is attained at $c u_*$, with the normalization $u_*\geq 0$ and $\frac{\mathcal E_T[u_*]}{\|u_*\|_q^q}=\frac{d(d-2)}{4}$. We write
	$$
	U = c u_* + R
	$$
	and observe the orthogonality conditions
	\begin{equation}
		\label{eq:betortho}
		\mathcal E_T[u_*,R] = \mathcal E_T[\partial_t u_*,R] = 0 \,.
	\end{equation}
	Of course, the second condition here is trivial if $T<T_*$ (in which case $u_*$ is constant).
	
	Expanding the $q$-norm as in the previous lecture, we arrive at
	\begin{align*}
		\mathcal E_T[u] - S_d(T)\|u\|_q^2 & \geq \mathcal E_T[R] - (q-1) \int_{\Sigma_T} u_*^{q-2} R^2\,dv_g \\
		& \quad - \const |c|^{2-q} \left( |c|^{(q-3)_+} \|R\|_q^{\min\{q-3\}} + \|R\|_q^q \right).
	\end{align*}
	For the quadratic term we have by definition \eqref{eq:defct} and the orthogonality conditions \eqref{eq:betortho}
	$$
	\mathcal E_T[R] - (q-1) \int_{\Sigma_T} u_*^{q-2} R^2\,dv_g = \langle R,\mathcal L_T R\rangle \geq c_T\, \mathcal E_T[R] = c_T\,\delta[u]^2 \,. 
	$$
	For the remainder term we bound, just like in the previous lecture,
	$$
	\|R\|_q \leq \tau_u \|u_*\|_q |c| \,.
	$$
	Thus, again as in the previous lecture,
	$$
	|c|^{2-q} \left( |c|^{(q-3)_+} \|R\|_q^{\min\{q-3\}} + \|R\|_q^q \right)
	\lesssim \left( \tau_u^{\min\{q-2,1\}} + \tau_u^{q-2} \right) \delta[u]^2 \,.
	$$
	This proves the claimed inequality.
\end{proof}


\subsection*{Degenerate stability: The upper bound}

In the remainder of these lectures, we discuss the degenerate stability in the case $T=T_*$.

We begin with the proof of the upper bound \eqref{eq:stabtuppercrit}, which shows that stability cannot hold with an exponent smaller than 4. If we were only interested in showing this, we could simply take the same trial function $u_*+\epsilon r$ as in the case $T\neq T_*$ and expand the $q$-norm to higher order. It is however instructive and helpful for the understanding of the following proof of Theorem \ref{stabt} to consider a more general family of trial states
$$
u = u_* + \epsilon r_* + \epsilon^2 s \,.
$$
We recall that $u_* = u_0 = ((d-2)/d)^{(d-2)/4}$. The function $r_*$ is chosen so as to minimize $\langle r,\mathcal L_{T_*} r \rangle/\mathcal E_{T_*}[r]$ under the orthogonality condition $\mathcal E_{T_*}[u_*,r]=0$. As shown in \eqref{eq:kernell}, this leads to the choice of $r_*$ as a linear combination of $\sin(\frac{2\pi}{T_*}\,\cdot)$ and $\cos(\frac{2\pi}{T_*}\,\cdot)$. By translation invariance the choice of the linear combination parameters is immaterial, and we choose to take
$$
r_*=\cos(\frac{2\pi}{T_*}\,\cdot) \,.
$$
The function $s$ is our variational parameter that we will optimize over at the end. We assume that
\begin{equation}
	\label{eq:stabtuppercritortho1}
	\mathcal E_{T_*}[u_*,s]=0
\end{equation}
(which is the same as $\int_{\Sigma_T} s\,dv_g =0$) and that
\begin{equation}
	\label{eq:stabtuppercritortho2}
	\mathcal E_{T_*}[r_*,s]=\mathcal E_{T_*}[\partial_t r_*,s] = 0
\end{equation}
(which, using \eqref{eq:stabtuppercritortho1}, is the same as  $\int_{\Sigma_T} r_* s\,dv_g = \int_{\Sigma_T} (\partial_t r_*) s\,dv_g =0$).

The motivation for requiring the first orthogonality condition in \eqref{eq:stabtuppercritortho2} is that, if $s$ contained `a part of $r_*$', we could simply absorb this part into $\epsilon r_*$ by redefining $\epsilon$. The motivation for the second condition is that, if $s$ contained `a part of $\partial_t r_*$', we could essentially absorb this term by translating $r_*$.

The orthogonality condition \eqref{eq:stabtuppercritortho1} guarantees that $\mathcal E_{T_*}[u_*,r_*+\epsilon s] =0$, so as in the proof of Lemma 12 we find that, if $\epsilon$ is sufficiently small depending on ($r_*$ and) $s$,
$$
\inf_{g\in\mathcal G_{T_*}} \mathcal E_{T_*}[u_*+\epsilon r_* + \epsilon^2 s - g] = \mathcal E_{T_*}[\epsilon r_*+\epsilon^2 s] = \epsilon^2 \mathcal E_{T_*}[r_*] + \epsilon^4 \mathcal E_{T_*}[s] \,.
$$
The second equality uses the first orthogonality condition in \eqref{eq:stabtuppercritortho2}.

A lengthy, but straightforward computation shows that
\begin{align*}
	& \lim_{\epsilon\to 0} \epsilon^{-4} \left( \mathcal E_{T_*}[u_* + \epsilon r_* + \epsilon^2 s] - S_d(T_*)\|u_* + \epsilon r_* + \epsilon^2 s\|_q^2\right) \\
	& = \langle s,\mathcal L_{T_*} s \rangle - \tfrac{(d-2)^2}{4} (q-1)(q-2) u_0^{-1} \langle r_*^2,s \rangle
	+ C_*
\end{align*}
with
$$
C_* := \tfrac{(d-2)^2}{4}\, \tfrac{(q-1)(q-2)}{4}\,u_0^{-2} |\Sigma_T| \left( - \tfrac13\,(q-3) \int_{\Sigma_T} r_*^4\,\frac{dv_g}{|\Sigma_T|} + (q-1) \left( \int_{\Sigma_T} r_*^2\,\frac{dv_g}{|\Sigma_T|} \right)^2 \right).
$$
At this point, in order to obtain an upper bound that is as small as possible, we need to solve the optimization problem
$$
\inf\left\{ \langle s,\mathcal L_{T_*} s \rangle - 2 \langle f_*,s \rangle :\ s \ \text{satisfies}\ \eqref{eq:stabtuppercritortho1} \ \text{and}\ \eqref{eq:stabtuppercritortho2} \right\}
$$
with $f_*:= \tfrac{(d-2)^2}{8} (q-1)(q-2) u_0^{-1} r_*^2$. According to \eqref{eq:kernell}, the orthogonality conditions on $s$ mean that $s$ is $L^2$-orthogonal to the kernel of $\mathcal L_{T_*}$. Therefore, denoting by $P^\bot$ the orthogonal projection onto the orthogonal complement of this kernel and by $\mathcal L_{T_*}^\bot$ the restriction of $\mathcal L_{T_*}$ to the range of $P^\bot$, where it is invertible, we can write
$$
\langle s,\mathcal L_{T_*} s \rangle - 2 \langle f_*,s \rangle = \left\| \left(\mathcal L_{T_*}^\bot\right)^{1/2} s - \left(\mathcal L_{T_*}^\bot\right)^{-1/2} P^\bot f_* \right\|_2^2 - \left\langle P^\bot f_*, (\mathcal L_{T_*}^{\bot})^{-1} P^\bot f_* \right\rangle. 
$$
This is minimized by taking $s= (\mathcal L_{T_*}^\bot)^\bot P^\bot f_*$. With this choice one obtains
$$
\lim_{\epsilon\to 0} \epsilon^{-4} \left( \mathcal E_{T_*}[u_* + \epsilon r_* + \epsilon^2 s] - S_d(T_*)\|u_* + \epsilon r_* + \epsilon^2 s\|_q^2 \right) = C_* - \left\langle P^\bot f_*, (\mathcal L_{T_*}^\bot)^{-1} P^\bot f_* \right\rangle .
$$
Consequently,
\begin{align*}
	& \lim_{\epsilon\to 0} \frac{\mathcal E_{T_*}[u_*+\epsilon r_*+\epsilon^2 s] \left(\mathcal E_{T_*}[u_* + \epsilon r_* + \epsilon^2 s] - S_d(T_*)\|u_* + \epsilon r_* + \epsilon^2 s\|_q^2\right)}{\inf_{g\in\mathcal G_{T_*}} \mathcal E_{T_*}[u_*+\epsilon r_* + \epsilon^2 s - g]^2} \\
	& = \frac{\mathcal E_{T_*}[u_*] \left( C_* - \left\langle P^\bot f_*, (\mathcal L_{T_*}^\bot)^{-1} P^\bot f_* \right\rangle\right)}{\mathcal E_{T_*}[r_*]^2} \,.
\end{align*}

The fact that $u_*$ is a minimizer implies that $C_* - \left\langle P^\bot f_*, (\mathcal L_{T_*}^\bot)^{-1} P^\bot f_* \right\rangle\geq 0$. If this difference vanished, then quadratic stability would be violated. A direct computation, however, shows that this difference is strictly positive,
\begin{equation}
	\label{eq:secondarynondeg}
	C_* - \left\langle P^\bot f_*, (\mathcal L_{T_*}^\bot)^{-1} P^\bot f_* \right\rangle > 0 \,.
\end{equation}
We do not give the details of this computation, but as an intermediate step we mention that
$$
C_* = \tfrac{(d-2)^2}{4}\, \tfrac{1}{32}(q-1)(q-2)(q+1)\, |\Sigma_{T_*}|\, u_0^{-2} \,,
$$
as well as
$$
(\mathcal L_{T_*}^\bot)^{-1} P^\bot f_*(t,\omega) = \tfrac{d-2}{48}\, (q-1)(q-2)\,u_0^{-1} \, \cos \tfrac{4\pi}{T_*}t \,,
$$
so
$$
\left\langle P^\bot f_*, (\mathcal L_{T_*}^\bot)^{-1} P^\bot f_* \right\rangle = \tfrac{(d-2)^2}{4}\, \tfrac{1}{96} (q-1)^2(q-2)\, |\Sigma_{T_*}|\, u_0^{-2}
$$
In this way we obtain the claimed value in the upper bound \eqref{eq:stabtuppercrit}.

As we will see in the proof of the lower bound given momentarily, the positivity in \eqref{eq:secondarynondeg} is the reason why we have quartic stability, rather than only stability of a higher order. We think of \eqref{eq:secondarynondeg} as a \emph{secondary nondegeneracy condition}. For $T=T_*$ the primary nondegeneracy condition, which says that elements in the kernel of $\mathcal L_{T_*}$ come from symmetries, is violated, and therefore no quadratic stability can hold. The secondary nondegeneracy condition \eqref{eq:secondarynondeg}, however, is satisfied and therefore one does have quartic stability. It is conceivable that there is a Sobolev-type functional inequality where both the primary and secondary nondegeneracy conditions fail and where the validity of a sextic stability result depends on a tertiary nondegeneracy condition, although no such example is known to the author.


\subsection*{Degenerate stability: The lower bound}

Finally, we sketch the proof of Theorem \ref{stabt} in the case $T=T_*$. As in the previous lecture, given the precompactness of optimizing sequences, it suffices to prove the following asymptotic lower bound, where we set
$$
\tau_u:=\frac{\delta[u]}{\sqrt{\mathcal E_{T_*}[u]-\delta[u]^2}} 
= \frac{\delta[u]}{\frac{d-2}2\, |\Sigma_{T_*}|^{-1/2} |\int_{\Sigma_{T_*}} u\,dv_g|} \,.
$$

\begin{proposition}
	Let $d\geq 3$ and $T=T_*$. Then there is a $\tau_*>0$ such that, for all $0\neq u\in H^1(\Sigma_{T_*})$ with $\int_{\Sigma_{T_*}} u\,dv_g \neq 0$ and $\tau_u\leq\tau_*$,
	$$
	\mathcal E_{T_*}[u] \left( \mathcal E_{T_*}[u] - S_d(T_*)\|u\|_q^2 \right) - \frac{(q+2)(q-2)}{12(q-1)}\, \delta[u]^4 \gtrsim - \tau_u\, \delta[u]^4 \,.
	$$
\end{proposition}

\begin{proof}
	We denote by $u_*$ the constant function $u_0=((d-2)/d)^{(d-2)/4}$ and set $c := u_0^{-1} |\Sigma_{T_*}|^{-1} \int_{\Sigma_{T_*}} u\,dv_g$ and $r:= u - c u_*$, so that
	$$
	u = c u_* + r \,,
	\qquad
	\int_{\Sigma_{T_*}} r\,dv_g = 0
	\qquad\text{and}\qquad
	\delta[u]^2 = \mathcal E_{T_*}[r] \,.
	$$
	
	\medskip
	
	\emph{Step 1.} We show that, by choosing $\tau_*>0$ sufficiently small, depending only on $d$, we may assume that $r$ depends only on $t$ and not on $\omega$. To do so, we decompose
	$$
	r = r_0 + r'
	\qquad\text{with}\ r_0(t) := |\Sph^{d-1}|^{-1} \int_{\Sph^{d-1}} r(t,\omega)\,d\omega \,.
	$$
	Similarly as in the proof of Lemma \ref{dist} we have
	\begin{align*}
		S_d(T_*) \left( \|r_0\|_q^2 + \|r'\|_q^2\right) & \leq \mathcal E_{T_*}[r_0] + \mathcal E_{T_*}[r'] = \mathcal E_{T_*}[r] = \tau_u^2 \left( \mathcal E_{T_*}[u]- \mathcal E_{T_*}[r] \right) \\
		& = \tau_u^2 \mathcal E_{T_*}[cu_*] = \tau_u^2 S_d(T_*) c^2 \|u_*\|_q^2
	\end{align*}
	Thus, by choosing $\tau_*>0$ small, we can ensure that both $\|r_0\|_q$ and $\|r'\|_q$ are as small as we wish with respect to $|c|$.
		
	We have
	$$
	\mathcal E_{T_*}[u] = \mathcal E_{T_*}[cu_*+r_0] + \mathcal E_{T_*}[r']
	$$
	and, by a second-order Taylor expansion,
	\begin{align*}
		& \left| \| u \|_q^2 - \|cu_*+r_0\|_q^2 -(q-1) |\Sigma_{T_*}|^{-1+2/q} \int_{\Sigma_{T_*}} r'^2\,dv_g \right| \\
		& \lesssim |c|^{\max\{2-q,-1\}} \left( \|r'\|_q^{\min\{q,3\}} + \|r_0\|_q^{\min\{q-2,1\}} \|r'\|_q^2 \right). 
	\end{align*}
	Here we used the smallness of $\tau_*$ in order to bound the difference between $\|cu_* + r_0\|_q^{2-q}$ and $\|c u_*\|_q^{2-q}$. We use this smallness again to bound
	$$
	|c|^{\max\{2-q,-1\}} \left( \|r'\|_q^{\min\{q,3\}} + \|r_0\|_q^{\min\{q-2,1\}} \|r'\|_q^2 \right)
	\lesssim \tau_*^{\min\{q-2,1\}} \|r'\|_q^2 \,.
	$$
	Thus, we obtain
	\begin{align*}
		\mathcal E_{T_*}[u] - S_d(T_*) \|u\|_q^2 & \geq \mathcal E_{T_*}[c u_*+r_0] - S_d(T_*) \|cu_*+r_0\|_q^2 \\
		& \quad + \langle r',\mathcal L_{T_*} r'\rangle - \const \tau_*^{\min\{q-2,1\}} \|r'\|_q^2 \,. 
	\end{align*}
	Since $r'$ is $L^2$-orthogonal to the kernel of $\mathcal L_{T_*}$, we have $\langle r',\mathcal L_{T_*} r'\rangle \gtrsim \mathcal E_{T_*}[r']$, and therefore, if $\tau_*>0$ is sufficiently small,
	$$
	\mathcal E_{T_*}[u] - S_d(T_*) \|u\|_q^2 \geq \mathcal E_{T_*}[c u_*+r_0] - S_d(T_*) \|cu_*+r_0\|_q^2  + \gamma\, \mathcal E_{T_*}[r']
	$$
	with some constant $\gamma>0$.
	
	Assume now that we can prove the desired bound for the function $cu_*+r_0$, which only depends on the $t$-variable. Then the right side above is
	$$
	\geq \frac{(q+2)(q-2)}{12(q-1)} \left( 1 - \const \tau_u \right) \frac{\mathcal E_{T_*}[r_0]^2}{\mathcal E_{T_*}[cu_*+r_0]} + \gamma\, \mathcal E_{T_*}[r'] \,.
	$$
	This gives the desired bound for the function $u$, since for all sufficiently small $\tau_*>0$,
	$$
	\frac{\mathcal E_{T_*}[r_0]^2}{\mathcal E_{T_*}[cu_*+r_0]} + \gamma'\, \mathcal E_{T_*}[r'] \geq \frac{\mathcal E_{T_*}[r]^2}{\mathcal E_{T_*}[cu_*+r_0]} \geq \frac{\mathcal E_{T_*}[r]^2}{\mathcal E_{T_*}[u]} \,.
	$$
	Indeed, the first inequality here is equivalent to $\gamma' \mathcal E_{T_*}[cu_*+r_0] \geq 2 \mathcal E_{T_*}[r_0] + \mathcal E_{T_*}[r']$, and this holds since $\mathcal E_{T_*}[r_0]$ and $\mathcal E_{T_*}[r']$ can be chosen as small as we wish with respect to $|c|$.
	
	\medskip
	
	\emph{Step 2.} According to Step 1, we may assume that $r$ is independent of $t$. We decompose further
	$$
	r= r_\parallel + r_\perp
	\qquad\text{with}\ r_\parallel\in\ker\mathcal L_{T_*} \,,\ r_\perp\in(\ker\mathcal L_{T_*})^\bot \,.
	$$
	In this step we will show that, by choosing $\tau_*>0$ sufficiently small, depending only on $d$, we may assume that $\mathcal E_{T_*}[r_\parallel]\geq \mathcal E_{T_*}[r_\perp]$.
	
	Indeed, by a second-order Taylor expansion as in Step 1, one finds
	\begin{align*}
		\mathcal E_{T_*}[u] - S_d(T_*) \|u\|_q^2 & \geq \langle r,\mathcal L_{T_*} r\rangle - \const \tau_*^{\min\{q-2,1\}} \|r\|_q^2 \,. 
	\end{align*}
	Note that
	$$
	\langle r,\mathcal L_{T_*} r\rangle = \langle r_\perp,\mathcal L_{T_*} r_\perp\rangle \gtrsim \mathcal E_{T_*}[r_\perp] \,.
	$$
	Moreover, assuming $\mathcal E_{T_*}[r_\parallel]\leq \mathcal E_{T_*}[r_\perp]$, we have
	$$
	\|r\|_q^2 \lesssim \mathcal E_{T_*}[r] = \mathcal E_{T_*}[r_\parallel] + \mathcal E_{T_*}[r_\perp] \leq 2 \mathcal E_{T_*}[r_\perp] \,.
	$$
	Thus, if $\tau_*$ is sufficiently small, we have
	$$
	\mathcal E_{T_*}[u] - S_d(T_*) \|u\|_q^2 \gtrsim \mathcal E_{T_*}[r_\perp] \geq \tfrac12\, \mathcal E_{T_*}[r]
	$$
	Since $\mathcal E_{T_*}[u]/\delta[u]^2 = \tau_u^{-2} + 1$, this is a stronger bound than the claimed one.

	\medskip
	
	\emph{Step 3.} We now come to the main part of the proof. We recall that $r$ depends only on $t$ and decompose it as in Step 2 with $\mathcal E_{T_*}[r_\perp]\leq\mathcal E_{T_*}[r_\parallel]$. Note that
	$$
	r_\parallel = a \cos(\tfrac{2\pi}{T_*}\,\cdot) + b \sin\cos(\tfrac{2\pi}{T_*}\,\cdot) = \sqrt{a^2+b^2}\,\cos(\tfrac{2\pi}{T_*}(\cdot-\tilde t))
	$$
	and after a translation of $u$ we may assume that $\tilde t=0$. Thus, $r_\parallel= \epsilon r_*$ with $r_*=\cos(\tfrac{2\pi}{T_*}\,\cdot)$ and $\epsilon:=\sqrt{a^2+b^2}$.
	
	As we are dealing with functions of a single variable, we have a Sobolev embedding into $L^\infty$, which gives
	$$
	\epsilon^2 + \| r_\perp \|_\infty^2 \lesssim \mathcal E_{T_*}[r_\parallel] + \mathcal E_{T_*}[r_\perp] = \mathcal E_{T_*}[r] = \tau_u^2 \mathcal E_{T_*}[cu_*] = \tau_u^2 S_d(T_*) c^2 \| u_*\|_q^2 \,.
	$$
	Thus, by choosing $\tau_*>0$ small, we can ensure that both $\epsilon$ and $\|r_\perp\|_\infty$ are as small as we wish with respect to $|c|$.
	
	Thanks to this $L^\infty$-bound, we can Taylor expand the norm to fourth order (note that $q<4$ if $d>4$) and obtain eventually
	\begin{align*}
		\| u \|_q^2 & = c^2 \|u_*\|_q^2 + (q-1) \|u_*\|_q^{2-q} u_0^{q-2} \int_{\Sigma_{T_*}} (\epsilon^2 r_*^2 + r_\perp^2)\,dv_g \\
		& \quad + (q-1)(q-2) \|u_*\|_q^{2-q} u_0^{q-3} |c|^{-1} \epsilon^2 \int_{\Sigma_{T_*}} r_*^2 r_\perp \,dv_g \\
		& \quad + \tfrac1{12} (q-1)(q-2)(q-3) \|u_*\|_q^{2-q} u_0^{q-4} c^{-2} \epsilon^4 \int_{\Sigma_{T_*}} r_*^4\,dv_g \\
		& \quad - \tfrac14(q-1)^2(q-2) \|u_*\|_q^{2-2q} u_0^{2q-4} c^{-2} \epsilon^4 \left( \int_{\Sigma_{T_*}} r_*^2\,dv_g \right)^2 \\
		& \quad + \mathcal O(|c|^{-3}\epsilon^5 + |c|^{-1}\epsilon \|r_\bot\|_\infty^2) \,.
	\end{align*}
	(In this bound we controlled a term $|c|^{-1}\|r_\bot\|_\infty^3$ by $|c|^{-1}\epsilon\|r_\bot\|_\infty^2$ using our assumption $\|r_\bot\|^2_\infty \lesssim \mathcal E_{T_*}[r_\bot]\leq \mathcal E_{T_*}[r_\parallel] = \const \epsilon^2$. Moreover, we controlled a term $c^{-2}\epsilon^3\|r_\bot\|_\infty$ by $|c|^{-3}\epsilon^5 + |c|^{-1}\epsilon \|r_\bot\|_\infty^2$ using Schwarz.) Thus,
	\begin{align*}
		\mathcal E_{T_*}[u] - S_d(T_*)\|u\|_q^2 
		& = \langle r_\bot, \mathcal L_{T_*} r_\bot \rangle - 2 |c|^{-1} \epsilon^2 \left\langle f_*, r_\perp \right\rangle + c^{-2} \epsilon^4 C_* \\
		& \quad + \mathcal O(|c|^{-3}\epsilon^5 + |c|^{-1}\epsilon \|r_\bot\|_\infty^2) \,,
	\end{align*}
	where $C_*$ and $f_*$ are as in the proof of the upper bound. By completing a square, much like in the proof of the upper bound,
	\begin{align*}
		& \langle r_\bot, \mathcal L_{T_*} r_\bot \rangle - 2 |c|^{-1} \epsilon^2 \left\langle f_*, r_\perp \right\rangle 
		 - \const |c|^{-1}\epsilon \|r_\bot\|_\infty^2 \\
		& \quad \geq - c^{-2} \epsilon^4 \langle P^\bot f_*,(\mathcal L_{T_*}^\bot)^{-1} P^\bot f_* \rangle - \const |c|^{-3} \epsilon^5 \,. 
	\end{align*}
	This shows that we are almost in the situation of the upper bound, and we find
	\begin{align*}
		\frac{\mathcal E_{T_*}[u] \left( \mathcal E_{T_*}[u] - S_d(T_*)\|u\|_q^2 \right)}{\mathcal E_{T_*}[\epsilon r_*]^2} \geq \frac{(q+2)(q-2)}{12(q-1)} - \const \tau_u \,.
	\end{align*}
	This is almost the claimed bound, except that we have $\mathcal E_{T_*}[\epsilon r_*]^2$ in the denominator instead of $\mathcal E_{T_*}[\epsilon r_*+r_\bot]^2$. If $\mathcal E_{T_*}[r_\bot] \lesssim c^{-2}\epsilon^4$ (with any fixed implied constant), we have $\mathcal E_{T_*}[\epsilon r_*+r_\bot]^2 \leq (1+ \const \tau_u^2)\, \mathcal E_{T_*}[\epsilon r_*]^2$, so this difference is harmless.
	
	Meanwhile, if $\mathcal E_{T_*}[r_\bot] \geq M c^{-2}\epsilon^4$ with a sufficiently large constant $M>0$, we argue slightly differently and use
	$$
	\langle r_\bot, \mathcal L_{T_*} r_\bot \rangle - 2 |c|^{-1} \epsilon^2 \left\langle f_*, r_\perp \right\rangle \geq \gamma\, \mathcal E_{T_*}[r_\bot]
	$$
	with some constant $\gamma>0$. Thus,
	$$
	\mathcal E_{T_*}[u] - S_d(T_*)\|u\|_q^2 \geq \left( \gamma\,  \mathcal E_{T_*}[r_\bot] + c^{-2}\epsilon^4 C_* \right) \left( 1- \const \tau_u \right).
	$$
	Here, as before, we want to replace $\epsilon^4 = \mathcal E_{T_*}[\epsilon r_*]^2/\mathcal E_{T_*}[r_*]^2$ by $\mathcal E_{T_*}[\epsilon r_* + r_\bot]^2/\mathcal E_{T_*}[r_*]^2$. (We could also use the fact that $C_*$ is strictly larger than the constant we want to obtain, but we do not need this.) Thus the claimed bound will follow if we can show that
	$$
	\gamma\,  \mathcal E_{T_*}[r_\bot] + c^{-2}\epsilon^4 C_* \geq c^{-2}\, \frac{\mathcal E_{T_*}[\epsilon r_* + r_\bot]^2}{\mathcal E_{T_*}[r_*]^2}\, C_* \,.
	$$
	This is equivalent to $c^2 (\gamma/C_*) \mathcal E_{T_*}[r_*]^2 \geq 2 \mathcal E_{T_*} [\epsilon r_*] + \mathcal E_{T_*}[r_\bot]$ and, since $\mathcal E_{T_*}[r_\bot]\leq \mathcal E_{T_*}[\epsilon r_*]$, this follows from $\epsilon \lesssim \tau_u |c|$ by choosing $\tau_*>0$ small enough. This concludes the proof of Theorem \ref{stabt} in case $T=T_*$.	
\end{proof}

\begin{remark}
	The above proof actually yields the stronger stability result
	$$
	\mathcal E_{T_*}[u] - S_d(T_*)\|u\|_q^2 \gtrsim \frac{\mathcal E_{T_*}[ \pi u]^2}{\mathcal E_{T_*}[u]} + \mathcal E_{T_*}[ \pi^\bot u] \,,
	$$
	where $\pi$ denotes the orthogonal projection in $L^2(\Sigma_{T_*})$ onto the span of $1$, $\sin(\frac{2\pi}{T_*}\,\cdot)$ and $\cos(\frac{2\pi}{T_*}\,\cdot)$ and where $\pi^\bot = 1-\pi$. Thus the quartic behavior appears only on a low-dimensional subspace, whereas on its orthogonal complement we have quadratic stability.
\end{remark}


\subsection*{Bibliographic remarks}

Quantitative stability for minimizing Yamabe metrics on closed manifolds was studied by Engelstein, Neumayer and Spolaor \cite{EnNeSp}. Their result specialized to $\Sigma_T$, plus the known fact that the optimizers are, in their terminology, nondegenerate (if $T<T_*$) or integrable (if $T>T_*$), implies the main theorem of this lecture in the case $T\neq T_*$. In the critical case $T=T_*$ their paper yields an inequality with an unspecified power $\geq 4$. The fact that this power can be chosen to be equal to $4$ is from \cite{Fr}.

The explicit study of the Yamabe problem on $\Sigma_T$ is due to Schoen \cite{Sc}, where also the phase-plane analysis appears; for more details see also \cite{CaChRu}, including a proof of monotonicity of the period map following \cite{BVBo}. The classification of zero modes of the Hessian uses ideas from \cite{MaPa,MaPoUh}. The analysis of the solutions on $\Sigma_T$ and of their linearization plays a role in the description of the asymptotic behavior of positive solutions of $-\Delta u = u^{(d+2)/(d-2)}$ near isolated singularities; see,e.g., \cite{CaGiSp,KoMaPaSc}.

The fact that there is an optimizer for $S_d(T)$ under the assumption $S_d(T)<S_d$ is a special case of a result of Aubin \cite{Au0}. Our proof is different and yields relative compactness of optimizing sequences. It is related to Lieb's proof in \cite[Lemma 1.2]{BrNi}.

An optimal, quartic stability result for the subcritical Sobolev inequality \eqref{eq:bvvb} on $\Sph^d$, which we discussed in the appendix of Lecture 2, can be obtained by the same method as in this lecture \cite{Fr}. The fact that for this inequality one has quadratic stability away from a low-dimensional subspace was observed in \cite{BrDoSi}, where in addition explicit constants were obtained by avoiding the use of compactness.


\bibliographystyle{amsalpha}

\begin{thebibliography}{90}

\bibitem{AbSt} M. Abramowitz, I. A. Stegun, \textit{Handbook of mathematical functions with formulas, graphs, and mathematical tables}. Reprint of the 1972 edition. Dover Publications, New York, 1992.

\bibitem{Au} T. Aubin, \textit{Probl\`emes isop\'erim\'etriques et espaces de Sobolev}. (French) J. Differential Geometry \textbf{11} (1976), no. 4, 573--598.

\bibitem{Au0} T. Aubin, \textit{\'Equations diff\'erentielles non lin\'eaires et probl\`eme de Yamabe concernant la courbure scalaire}. (French) J. Math. Pures Appl. (9) \textbf{55} (1976), no. 3, 269--296.

\bibitem{Au2} T. Aubin, \textit{Meilleures constantes dans le th\'eor\`eme d'inclusion de Sobolev et un th\'eor\`eme de Fredholm non lin\'eaire pour la transformation conforme de la courbure scalaire}. (French) J. Functional Analysis \textbf{32} (1979), no. 2, 148--174. 

\bibitem{BaChDa} H. Bahouri, J.-Y. Chemin, R. Danchin, \textit{Fourier analysis and nonlinear partial differential equations}. Grundlehren der mathematischen Wissenschaften, \textbf{343}. Springer, Heidelberg, 2011.

\bibitem{Ba} D. Bakry, \textit{L'hypercontractivit\'e et son utilisation en th\'eorie des semigroupes}. (French) In: Lectures on probability theory (Saint-Flour, 1992), 1--114, Lecture Notes in Math., \textbf{1581}, Springer, Berlin, 1994.

\bibitem{BaEm} D. Bakry, M. \'Emery, \textit{Diffusions hypercontractives}. In: S\'eminaire de probabilit\'es, XIX, 1983/84, 177--206, Lecture Notes in Math. \textbf{1123}, Springer, Berlin, 1985.

\bibitem{BaWeWi} T. Bartsch, T. Weth, M. Willem, \textit{A Sobolev inequality with remainder term and critical equations on domains with topology for the polyharmonic operator.} Calc. Var. Partial Differential Equations \textbf{18} (2003), no. 3, 253--268. 

\bibitem{Be} W. Beckner, \textit{Sharp Sobolev inequalities on the sphere and the Moser--Trudinger inequality}. Ann. of Math. (2) \textbf{138} (1993), no. 1, 213--242.

\bibitem{BeFrVi} J. Bellazzini, R. L. Frank, N. Visciglia, \textit{Maximizers for Gagliardo-Nirenberg inequalities and related non-local problems}. Math. Ann. \textbf{360} (2014), no. 3-4, 653--673. 

\bibitem{BePe} R. Benedetti, C. Petronio, \textit{Lectures on hyperbolic geometry}. Universitext. Springer-Verlag, Berlin, 1992.

\bibitem{BiEg} G. Bianchi, H. Egnell, \textit{A note on the Sobolev inequality}. J. Funct. Anal. \textbf{100} (1991), no. 1, 18--24.

\bibitem{BVBo} M.-F. Bidaut-V\'eron, M. Bouhar, \textit{On characterization of solutions of some nonlinear differential equations and applications}. SIAM J. Math. Anal. \textbf{25} (1994), no. 3, 859--875.

\bibitem{BVVe} M.-F. Bidaut-V\'eron, L. V\'eron, \textit{Nonlinear elliptic equations on compact Riemannian manifolds and asymptotics of Emden equations}. Invent. Math. \textbf{106} (1991), no. 3, 489--539. Erratum: ibid. \textbf{112} (1993), no. 2, 445.

\bibitem{Bl} G. A. Bliss, \textit{An integral inequality}. J. London Math. Soc. \textbf{5} (1930), no. 1, 40--46.

\bibitem{BoDoNaSi} M. Bonforte, J. Dolbeault, B. Nazaret, N. Simonov, \textit{Stability in Gagliardo--Nirenberg--Sobolev inequalities: Flows, regularity and the entropy method}. Preprint (2021), arXiv:2007.03674.

\bibitem{BrLi} H. Br\'ezis, E. H. Lieb, \textit{Sobolev inequalities with remainder terms}. J. Funct. Anal. \textbf{62} (1985), no. 1, 73--86. 

\bibitem{BrNi} H. Br\'ezis, L. Nirenberg, \textit{Positive solutions of nonlinear elliptic equations involving critical Sobolev exponents}. Comm. Pure Appl. Math. \textbf{36} (1983), no. 4, 437--477.

\bibitem{BrDoSi} G. Brigati, J. Dolbeault, N. Simonov, \textit{Logarithmic Sobolev and interpolation inequalities on the sphere: constructive stability results}. Memoirs of the AMS, to appear. Preprint (2022), arXiv:2211.13180.

\bibitem{CaGiSp} L. A. Caffarelli, B. Gidas, J. Spruck, \textit{Asymptotic symmetry and local behavior of semilinear elliptic equations with critical Sobolev growth}. Comm. Pure Appl. Math. \textbf{42} (1989), no. 3, 271--297. 

\bibitem{Ca} E. A. Carlen, \textit{Duality and stability for functional inequalities}. Ann. Fac. Sci. Toulouse Math. (6) \textbf{26} (2017), no. 2, 319--350.

\bibitem{CaCaLo} E. A. Carlen, J. A. Carrillo, M. Loss, \textit{Hardy--Littlewood--Sobolev inequalities via fast diffusion flows}. Proc. Natl. Acad. Sci. USA \textbf{107} (2010), no. 46, 19696--19701.

\bibitem{CaFrLi} E. A. Carlen, R. L. Frank, E. H. Lieb, \textit{Stability estimates for the lowest eigenvalue of a Schr\"odinger operator}. Geom. Funct. Anal. \textbf{24} (2014), no. 1, 63--84.

\bibitem{CaLo} E. A. Carlen, M. Loss, \textit{Extremals of functionals with competing symmetries}. J. Funct. Anal. \textbf{88} (1990), no. 2, 437--456.

\bibitem{CaLo1} E. A. Carlen, M. Loss, \textit{On the minimization of symmetric functionals}. Rev. Math. Phys. \textbf{6} (1994), no. 5A, 1011--1032.

\bibitem{CaChRu} A. Carlotto, O. Chodosh, Y. A. Rubinstein, \textit{Slowly converging Yamabe flows}. Geom. Topol. \textbf{19} (2015), no. 3, 1523--1568.

\bibitem{Cas} J. S. Case, \textit{The Frank--Lieb approach to sharp Sobolev inequalities}. Commun. Contemp. Math. \textbf{23} (2021), no. 3, Paper No. 2050015, 16 pp.

\bibitem{CaWa} J. S. Case, Y. Wang, \textit{Towards a fully nonlinear sharp Sobolev trace inequality}. J. Math. Study \textbf{53} (2020), no. 4, 402--435.

\bibitem{ChHa} S.-Y. A. Chang, F. Hang, \textit{Improved Moser-Trudinger-Onofri inequality under constraints}. Comm. Pure Appl. Math. \textbf{75} (2022), no. 1, 197--220.

\bibitem{ChYa} S.-Y. A. Chang, P. C. Yang, \textit{A perturbation result in prescribing scalar curvature on $\Sph^n$}. Duke Math. J. \textbf{64} (1991), no. 1, 27--69.

\bibitem{ChLuTa} L. Chen, G. Lu, H. Tang, \textit{Stability of Hardy-Littlewood-Sobolev inequalities with explicit lower bounds}. Preprint (2023), arXiv:2301.04097.

\bibitem{ChFrWe} S. Chen, R. L. Frank, T. Weth, \textit{Remainder terms in the fractional Sobolev inequality}. Indiana Univ. Math. J. \textbf{62} (2013), no. 4, 1381--1397. 

\bibitem{ChLiOu} W. Chen, C. Li, B. Ou, \textit{Classification of solutions for an integral equation}. Comm. Pure Appl. Math. \textbf{59} (2006), no. 3, 330--343. Corrigendum: ibid., no. 7, 1064. 

\bibitem{ChWeWu} X. Chen, W. Wei, N. Wu, \textit{Almost sharp Sobolev trace inequalities in the unit ball under constraints}. Preprint (2021), arXiv:2107.08647.

\bibitem{CiLe} M. Cicalese, G. P. Leonardi, \textit{A selection principle for the sharp quantitative isoperimetric inequality}. Arch. Ration. Mech. Anal. \textbf{206} (2012), no. 2, 617--643.

\bibitem{CENaVi} D. Cordero-Erausquin, B. Nazaret, C. Villani, \textit{A mass-transportation approach to sharp Sobolev and Gagliardo--Nirenberg inequalities}. Adv. Math. \textbf{182} (2004), no. 2, 307--332. 

\bibitem{De} J. Demange, \textit{Improved Gagliardo--Nirenberg--Sobolev inequalities on manifolds with positive curvature}. J. Funct. Anal. \textbf{254} (2008), no. 3, 593--611.

\bibitem{DoEs} J. Dolbeault, M. J. Esteban, \textit{Hardy--Littlewood--Sobolev and related inequalities: stability}. In: The physics and mathematics of Elliott Lieb -- the 90th anniversary. Vol. I, 247--268, EMS Press, Berlin, 2022.

\bibitem{DoEsFiFrLo} J. Dolbeault, M. J. Esteban, A. Figalli, R. L. Frank, M. Loss, \textit{Sharp stability for Sobolev and log-Sobolev inequalities, with optimal dimensional dependence}. Preprint (2022), arXiv:2209.08651.

\bibitem{DoEsKoLo} J. Dolbeault, M. J. Esteban, M. Kowalczyk, M. Loss, \textit{Sharp interpolation inequalities on the sphere: new methods and consequences}. Chin. Ann. Math. Ser. B \textbf{34} (2013), no. 1, 99--112.

\bibitem{DoEsLo1} J. Dolbeault, M. J. Esteban, M. Loss, \textit{Nonlinear flows and rigidity results on compact manifolds}. J. Funct. Anal. \textbf{267} (2014), no. 5, 1338--1363.

\bibitem{DoEsLo2} J. Dolbeault, M. J. Esteban, M. Loss, \textit{Rigidity versus symmetry breaking via nonlinear flows on cylinders and Euclidean spaces}. Invent. Math. \textbf{206} (2016), no. 2, 397--440.

\bibitem{DoEsLo3} J. Dolbeault, M. J. Esteban, M. Loss, \textit{Interpolation inequalities on the sphere: linear vs. nonlinear flows}. Ann. Fac. Sci. Toulouse Math. (6) \textbf{26} (2017), no. 2, 351--379.

\bibitem{EnNeSp} M. Engelstein, R. Neumayer, L. Spolaor, \textit{Quantitative stability for minimizing Yamabe metrics}. Trans. Amer. Math. Soc. Ser. B \textbf{9} (2022), 395--414.

\bibitem{Es} J. F. Escobar, \textit{Sharp constant in a Sobolev trace inequality}. Indiana Univ. Math. J. \textbf{37} (1988), 687--698.

\bibitem{Fi1} A. Figalli, \textit{Stability in geometric and functional inequalities}. In: European Congress of Mathematics, 585--599, Eur. Math. Soc., Z\"urich, 2013.

\bibitem{Fi2} A. Figalli, \textit{Quantitative stability results for the Brunn-Minkowski inequality}. In: Proceedings of the International Congress of Mathematicians—Seoul 2014. Vol. III, 237--256, Kyung Moon Sa, Seoul, 2014.

\bibitem{FiMaPr} A. Figalli, F. Maggi, A. Pratelli, \textit{A mass transportation approach to quantitative isoperimetric inequalities}. Invent. Math. \textbf{182} (2010), no. 1, 167--211.

\bibitem{Fr} R. L. Frank, \textit{Degenerate stability of some Sobolev inequalities}. Ann. Inst. H. Poincar\'e C Anal. Non Lin\'eaire \textbf{39} (2022), no. 6, 1459--1484.

\bibitem{FrLaWe} R. L. Frank, A. Laptev, T. Weidl, \textit{Schr\"odinger operators: eigenvalues and Lieb--Thirring inequalities}. Cambridge Studies in Advanced Mathematics, \textbf{200}. Cambridge University Press, Cambridge, 2023.

\bibitem{FrLi00} R. L. Frank, E. H. Lieb, \textit{Inversion positivity and the sharp Hardy--Littlewood--Sobolev inequality}. Calc. Var. Partial Differential Equations \textbf{39} (2010), no. 1-2, 85--99. 

\bibitem{FrLi01} R. L. Frank, E. H. Lieb, \textit{Spherical reflection positivity and the Hardy--Littlewood--Sobolev inequality}. Concentration, functional inequalities and isoperimetry, 89--102, Contemp. Math., 545, Amer. Math. Soc., Providence, RI, 2011.

\bibitem{FrLi0} R. L. Frank, E. H. Lieb, \textit{A new, rearrangement-free proof of the sharp Hardy--Littlewood--Sobolev inequality}. Spectral theory, function spaces and inequalities, 55--67, Oper. Theory Adv. Appl., \textbf{219}, Birkh\"auser/Springer Basel AG, Basel, 2012.

\bibitem{FrLi} R. L. Frank, E. H. Lieb, \textit{Sharp constants in several inequalities on the Heisenberg group}. Ann. of Math. (2) \textbf{176} (2012), no. 1, 349--381.

\bibitem{FuMaPr} N. Fusco, F. Maggi, A. Pratelli, \textit{The sharp quantitative isoperimetric inequality}. Ann. of Math. (2) \textbf{168} (2008), no. 3, 941--980.

\bibitem{Ge} P. G\'erard, \textit{Description du d\'efaut de compacit\'e de l'injection de Sobolev}. (French) ESAIM Control Optim. Calc. Var. \textbf{3} (1998), 213--233. 

\bibitem{GeMeOr} P. G\'erard, Y. Meyer, F. Oru, \textit{In\'egalit\'es de Sobolev pr\'ecis\'ees}. (French) S\'eminaire sur les \'Equations aux D\'eriv\'es Partielles, 1996--1997, Exp. No. IV, 11 pp., \'Ecole Polytech., Palaiseau, 1997. 

\bibitem{GiNiNi} B. Gidas, W. M. Ni, L. Nirenberg, \textit{Symmetry of positive solutions of nonlinear elliptic equations in $\R^n$}. Mathematical analysis and applications, Part A, pp. 369--402, Adv. in Math. Suppl. Stud., 7a, Academic Press, New York-London, 1981.

\bibitem{GiSp} B. Gidas, J. Spruck, \textit{Global and local behavior of positive solutions of nonlinear elliptic equations}. Comm. Pure Appl. Math. \textbf{34} (1981), no. 4, 525--598. 

\bibitem{GrZw} C. R. Graham, M. Zworski, \textit{Scattering matrix in conformal geometry}. Invent. Math. \textbf{152} (2003), no. 1, 89--118.

\bibitem{GrJeMaSp} C. R. Graham, R. Jenne, L. J. Mason, G. A. J. Sparling, \textit{Conformally invariant powers of the Laplacian. I. Existence}. J. London Math. Soc. (2) \textbf{46} (1992), no. 3, 557--565. 

\bibitem{HaWa} F. Hang, X. Wang, \textit{Improved Sobolev inequality under constraints}. Int. Math. Res. Not. IMRN 2022, no. 14, 10822--10857.

\bibitem{He} J. Hersch, \textit{Quatre propri\'et\'es isop\'erim\'etriques de membranes sph\'eriques homog\`enes}. C. R. Acad. Sci. Paris S\'er. A-B \textbf{270} (1970), A1645--A1648.

\bibitem{KiVi} R. Killip, M. Vi\c{s}an, \textit{Nonlinear Schr\"odinger equations at critical regularity}. Evolution equations, 325--437, Clay Math. Proc., \textbf{17}, Amer. Math. Soc., Providence, RI, 2013.

\bibitem{Ko1} T. K\"onig, \textit{On the sharp constant in the Bianchi--Egnell stability inequality}. Preprint (2022), arXiv:2210.08482.

\bibitem{Ko2} T. K\"onig, \textit{Stability for the Sobolev inequality: existence of a minimizer}. Preprint (2022), arXiv:2211.14185.

\bibitem{KoMaPaSc} N. Korevaar, R. Mazzeo, F. Pacard, R. Schoen, \textit{Refined asymptotics for constant scalar curvature metrics with isolated singularities}. Invent. Math. \textbf{135} (1999), no. 2, 233--272.

\bibitem{Le} M. Ledoux, \textit{On improved Sobolev embedding theorems}. Math. Res. Lett. \textbf{10} (2003), no. 5-6, 659--669.

\bibitem{LiY} Y. Y. Li, \textit{Remark on some conformally invariant integral equations: the method of moving spheres}. J. Eur. Math. Soc. (JEMS) \textbf{6} (2004), no. 2, 153--180.

\bibitem{Li0} E. H. Lieb, \textit{Existence and uniqueness of the minimizing solution of Choquard's nonlinear equation}. Studies in Appl. Math. \textbf{57} (1976/77), no. 2, 93--105.

\bibitem{Li} E. H. Lieb, \textit{Sharp constants in the Hardy--Littlewood--Sobolev and related inequalities}. Ann. of Math. (2) \textbf{118} (1983), no. 2, 349--374.
	
\bibitem{Li1} E. H. Lieb, \textit{On the lowest eigenvalue of the Laplacian for the intersection of two domains}. Invent. Math. \textbf{74} (1983), 441--448.
	
\bibitem{LiLo} E. H. Lieb, M. Loss, \textit{Analysis}. Second edition. Graduate Studies in Mathematics, \textbf{14}. American Mathematical Society, Providence, RI, 2001.

\bibitem{Lio1} P.-L. Lions, \textit{The concentration-compactness principle in the calculus of variations. The limit case. I}. Rev. Mat. Iberoamericana \textbf{1} (1985), no. 1, 145--201. 

\bibitem{Lio2} P.-L. Lions, \textit{The concentration-compactness principle in the calculus of variations. The limit case. II}. Rev. Mat. Iberoamericana \textbf{1} (1985), no. 2, 45--121.

\bibitem{LuWe} G. Lu, J. Wei, \textit{On a Sobolev inequality with remainder terms}. Proc. Amer. Math. Soc. \textbf{128} (2000), no. 1, 75--84. 

\bibitem{MaPa} R. Mazzeo, F. Pacard, \textit{Constant scalar curvature metrics with isolated singularities}. Duke Math. J. \textbf{99} (1999), no. 3, 353--418.

\bibitem{MaPoUh} R. Mazzeo, D. Pollack, K. Uhlenbeck, \textit{Moduli spaces of singular Yamabe metrics}. J. Amer. Math. Soc. \textbf{9} (1996), no. 2, 303--344.

\bibitem{Na} B. Nazaret, \textit{Best constant in Sobolev trace inequalities on the half-space}. Nonlinear Anal. \textbf{65} (2006), no. 10, 1977--1985.

\bibitem{Ob} M. Obata, \textit{The conjectures on conformal transformations of Riemannian manifolds}. J. Differential Geometry \textbf{6} (1971/72), 247--258.

\bibitem{Pe} J.-P. Penot, \textit{Analysis. From concepts to applications}. Universitext. Springer, Cham, 2016.

\bibitem{Rod} E. Rodemich, \textit{The Sobolev inequality with best possible constant}. Analysis Seminar Caltech, Spring 1966.

\bibitem{Ro} G. Rosen, \textit{Minimum value for $c$ in the Sobolev inequality $\|\phi^3\|\leq c\|\nabla\phi\|^3$}. SIAM J. Appl. Math. \textbf{21} (1971), 30--32.
	
\bibitem{Sa} J. Sabin, \textit{Compactness methods in Lieb's work}. In: The physics and mathematics of Elliott Lieb -- the 90th anniversary. Vol. II, 219--251, EMS Press, Berlin, 2022.

\bibitem{Sc} R. M. Schoen, \textit{Variational theory for the total scalar curvature functional for Riemannian metrics and related topics}. In: Topics in calculus of variations (Montecatini Terme, 1987), Lecture Notes in Math., vol. \textbf{1365}, Springer, Berlin, 1989, pp. 120--154.

\bibitem{StWe} E. M. Stein, G. Weiss, \textit{Introduction to Fourier analysis on Euclidean spaces}. Princeton Mathematical Series, No. \textbf{32}. Princeton University Press, Princeton, N.J., 1971.

\bibitem{Ta} G. Talenti, \textit{Best constant in Sobolev inequality}. Ann. Mat. Pura Appl. (4) \textbf{110} (1976), 353--372. 

\bibitem{Tr} N. S. Trudinger, \textit{Remarks concerning the conformal deformation of Riemannian structures on compact manifolds}. Ann. Scuola Norm. Sup. Pisa Cl. Sci. (3) \textbf{22} (1968), 265--274.

\bibitem{vN} J. van Neerven, \textit{Functional analysis}. Cambridge Studies in Advanced Mathematics, \textbf{201}. Cambridge University Press, Cambridge, 2022.

\bibitem{Ya} Z. Yan, \textit{Improved Sobolev inequalities on CR sphere}. Preprint (2023), 	arXiv:2301.07170.

\end{thebibliography}

\end{document}